\documentclass{amsart}

\usepackage{ae,aecompl}
\usepackage{esint}
\usepackage{graphicx}
\usepackage[all,cmtip]{xy}
\usepackage{amsmath, amscd}
\usepackage{booktabs}
\usepackage{amsthm}
\usepackage{amssymb}
\usepackage{amsfonts}
\usepackage{qsymbols}
\usepackage{latexsym}
\usepackage{mathrsfs}
\usepackage{cite}
\usepackage{color}
\usepackage{url}
\usepackage{enumerate}
\usepackage{inputenc,todonotes,esint}
\usepackage{verbatim}
\usepackage[draft=false, colorlinks=true]{hyperref}
\usepackage[new]{old-arrows}

\usepackage{cancel}

\allowdisplaybreaks

\newcommand {\A}{{\mathcal{A}}}

\newcommand {\bmo}{\mathrm{bmo}}
\newcommand {\BMO}{\mathrm{BMO}}

\newcommand {\C}{{\mathbb C}}
\newcommand {\Ca}{\mathcal{C}}

\newcommand {\con}{M}
\newcommand {\Cr}{C^{r}_{-}}
\newcommand {\Crtwo}{C^{2}_{-}}

\newcommand {\Da}{\mathcal{D}}

\newcommand {\ud}{\mathrm{d}}

\newcommand {\veps}{\varepsilon}

\newcommand {\ex}{\mathbf{e}}

\newcommand {\F}{\mathcal{F}}

\newcommand {\HT}{\mathcal{H}}
\newcommand {\Hp}{\mathcal{H}^{p}_{FIO}(\Rn)}

\newcommand {\Hps}{\mathcal{H}^{s,p}_{FIO}(\Rn)}

\newcommand {\Hpsm}{\mathcal{H}^{s-1,p}_{FIO}(\Rn)}
\newcommand {\Hpsmm}{\mathcal{H}^{s-2,p}_{FIO}(\Rn)} \newcommand {\Hrinf}{\mathcal{H}^{r,\infty}}

\newcommand {\ind}{\mathbf{1}}

\newcommand {\J}{\mathcal{J}}
\newcommand {\ka}{\kappa}

\newcommand {\la}{\lambda}
\newcommand {\rb}{\rangle}
\newcommand {\lb}{\langle}
\newcommand {\La}{\mathcal{L}}
\newcommand {\loc}{\mathrm{loc}}

\newcommand {\Ma}{\mathcal{M}}
\newcommand {\N}{{\mathbb N}}

\newcommand {\ph}{\varphi}

\newcommand {\R}{\mathbb R}

\newcommand {\Rn}{\mathbb{R}^{n}}

\newcommand {\supp}{\mathrm{supp}}

\newcommand {\Sp}{S^{*}\Rn}
\newcommand {\Spp}{S^{*}_{+}\Rn}
\newcommand {\St}{{\mathrm{St}}}
\newcommand {\Sw}{\mathcal{S}}

\newcommand {\w}{\omega}
\newcommand {\W}{\mathrm{W}}

\newcommand {\wh}{\widehat}
\newcommand {\wt}{\widetilde}

\newcommand {\Z}{\mathbb Z}

\newcommand {\vanish}[1]{\relax}

\newcommand {\cs}{\mathbf{c}}
\newcommand {\sn}{\mathbf{s}}

\DeclareFontFamily{U}{mathx}{\hyphenchar\font45}
\DeclareFontShape{U}{mathx}{m}{n}{
      <5> <6> <7> <8> <9> <10>
      <10.95> <12> <14.4> <17.28> <20.74> <24.88>
      mathx10
      }{}
\DeclareSymbolFont{mathx}{U}{mathx}{m}{n}
\DeclareFontSubstitution{U}{mathx}{m}{n}
\DeclareMathAccent{\widecheck}{0}{mathx}{"71}

\DeclareMathOperator{\Real}{Re}

\newtheorem{theorem}{Theorem}[section]
\newtheorem{lemma}[theorem]{Lemma}
\newtheorem{proposition}[theorem]{Proposition}
\newtheorem{corollary}[theorem]{Corollary}

\theoremstyle{definition}
\newtheorem{definition}[theorem]{Definition}
\newtheorem{remark}[theorem]{Remark}

\numberwithin{equation}{section}
\protected\def\ignorethis#1\endignorethis{}
\let\endignorethis\relax

\title[$\HT^{1}$ and $\bmo$ regularity for rough wave equations]{$\HT^{1}$ and $\bmo$ regularity for wave equations with rough coefficients}

\author{Naijia Liu}
\address{Department of Mathematics,
    Sun Yat-sen University,
    Guangzhou, 510275,
    P.R.~China}
\email{liunj@mail2.sysu.edu.cn}

\author{Jan Rozendaal}
\address{Institute of Mathematics, Polish Academy of Sciences\\
Ul.~\'{S}niadeckich 8\\
00-656 Warsaw\\
Poland}
\email{jrozendaal@impan.pl}

\author{Liang Song}
\address{Department of Mathematics,
    Sun Yat-sen University,
    Guangzhou, 510275,
    P.R.~China}
\email{songl@mail.sysu.edu.cn}

\keywords{Rough wave equations, $\HT^{1}$ and $\bmo$ regularity, Hardy spaces for Fourier integral operators}

\subjclass[2020]{Primary 35R05. Secondary 35L05, 42B37, 35A27}

\thanks{This research was funded in part by the National Science Center, Poland, grant 2021/43/D/ST1/00667. N.~ J. Liu is supported by China Postdoctoral Science Foundation (No. 2024M763732). J. Rozendaal is partially supported by NCN grant UMO-2023/49/B/ST1/01961. L. Song is  supported by  NNSF of China (No. 12471097).}

\begin{document}

\begin{abstract}
We consider second-order hyperbolic equations with rough time-independent coefficients. Our main result is that such equations are well posed on the  Hardy spaces $\HT^{s,1}_{FIO}(\Rn)$ and $\HT^{s,\infty}_{FIO}(\Rn)$ for Fourier integral operators if the coefficients have $C^{1,1}\cap C^{r}$ regularity in space, for $r>\frac{n+1}{2}$, where $s$ ranges over an 
$r$-dependent interval. As a corollary, we obtain the sharp fixed-time $\HT^{1}(\Rn)$ and $\bmo(\Rn)$ regularity for such equations, extending work by Seeger, Sogge and Stein in the case of smooth coefficients. 
\end{abstract}
	
\maketitle

\section{Introduction}\label{sec:intro}	

In this article we obtain the sharp fixed-time  regularity of the solution operators to rough wave equations, in the local Hardy space $\HT^{1}(\Rn)$ and in its dual, $\bmo(\Rn)$. 

\subsection{Setting} 
We consider second-order differential operators of the form
\begin{equation}\label{eq:defL}
Lf(x):=\sum_{i,j=1}^{n} D_{i} (a_{ij}D_{j} f)(x)+\sum_{j=1}^{n}a_{j}(x)D_{j}f(x)+a_{0}(x)f(x).
\end{equation}
Here the $a_{ij}:\Rn\to\R$ are uniformly elliptic, bounded and real-valued, for $1\leq i,j\leq n$, each $a_{j}:\Rn\to\C$ is bounded, for $0\leq j\leq n$, and we write $D_{j}=-i\partial_{x_{j}}$ for $x=(x_{1},\ldots, x_{n})\in\Rn$. Crucially, the coefficients $a_{ij}$ and $a_{j}$ are not assumed to be infinitely smooth; instead, they are contained in $C^{r}(\Rn)$ for some finite $r>0$.  

We study the associated inhomogeneous initial value problem on $\R\times\Rn$: 
\begin{equation}\label{eq:wave-eqn}
\begin{gathered}
D_{t}^{2} u(t,x)-Lu(t,x) = F(t,x),\\ 
u(0,x) = f(x), \qquad \partial_{t}u(0,x) = g(x). 
\end{gathered}
\end{equation}
Here 
$u,F:\R^{n+1}\to\C$ and $D_{t}:=-i\partial_{t}$. We suppose that $f$, $g$ and $F(t)=F(t,\cdot)$ are elements of suitable function spaces on $\Rn$, and we will determine the corresponding regularity of the solution $u(t)=u(t,\cdot)$ at a fixed time $t\in\R$.
  
\subsection{Previous work}

If $a_{ij}=a_{ji}$ for all $1\leq i,j\leq n$, and $a_{j}=0$ for $0\leq j\leq n$, then $L$ is a positive operator on $L^{2}(\Rn)$, 
and the solution to \eqref{eq:wave-eqn} is given by
\begin{equation}\label{eq:spectralsol}
u(t)=\cos(t\sqrt{L})f+\frac{\sin(t\sqrt{L})}{\sqrt{L}}g-\int_{0}^{t}\frac{\sin((t-s)\sqrt{L})}{\sqrt{L}}F(s)\ud s.
\end{equation}
By the spectral theorem and form theory, \eqref{eq:wave-eqn} is well posed on $L^{2}(\Rn)$, even if the $a_{ij}$ are merely bounded and measurable. More precisely, if $f\in L^{2}(\Rn)$, $g\in W^{-1,2}(\Rn)$ and $F\in L^{1}_{\loc}(\R;W^{-1,2}(\Rn))$, then $u(t)\in L^{2}(\Rn)$ for all $t\in\R$.

Upon considering the $L^{p}(\Rn)$ regularity of \eqref{eq:wave-eqn} for $p\neq 2$, the problem immediately ceases to be trivial. Indeed, even when dealing with the flat Laplacian $\Delta$, for which $a_{ij}=\delta_{ij}$, 
the operators $\cos(t\sqrt{-\Delta})$ and $\sin(t\sqrt{-\Delta})$ are not bounded on $L^{p}(\Rn)$ unless $t=0$, $p=2$ or $n=1$. As such, \eqref{eq:wave-eqn} is not well posed on $L^{p}(\Rn)$.

Instead, it was shown by Peral \cite{Peral80} and Miyachi \cite{Miyachi80a} that 
\begin{equation}\label{eq:Laplacian}
\begin{aligned}
\cos(t\sqrt{-\Delta})&:W^{2s(p),p}(\Rn)\to L^{p}(\Rn),\\
\frac{\sin(t\sqrt{-\Delta})}{\sqrt{-\Delta}}&:W^{2s(p)-1,p}(\Rn)\to L^{p}(\Rn),
\end{aligned}
\end{equation}
for all $t\in\R$ and $1<p<\infty$. Here and throughout, we write
\[
s(p):=\frac{n-1}{2}\Big|\frac{1}{p}-\frac{1}{2}\Big|
\]
for $0<p\leq \infty$. Moreover, the exponent $2s(p)$ in \eqref{eq:Laplacian} is sharp, for all $t\neq 0$.

In \cite{SeSoSt91}, \eqref{eq:Laplacian} was vastly generalized by Seeger, Sogge and Stein, who determined the sharp fixed-time $L^{p}(\Rn)$ regularity of \eqref{eq:wave-eqn} for smooth coefficients. In fact, \cite{SeSoSt91} concerns mapping properties of Fourier integral operators (FIOs), a class that contains the solution operators to \eqref{eq:wave-eqn}. Loosely speaking, it was shown that a Fourier integral operator $T$ of order $m\in\R$ satisfies
\begin{equation}\label{eq:FIOmap}
T:W^{2s(p)+m,p}(\Rn)\to L^{p}(\Rn)
\end{equation}
for all $1<p<\infty$. If $(a_{ij})_{i,j=1}^{n},(a_{j})_{j=0}^{n}\subseteq C^{\infty}(\Rn)$, then the solution to \eqref{eq:wave-eqn} is 
\begin{equation}\label{eq:solUj}
u(t)=U_{0}(t)f+U_{1}(t)g-\int_{0}^{t}U_{1}(t-s)F(s)\ud s,
\end{equation}
for $U_{k}(t)$ a Fourier integral operator of order $-k$. 
Hence, if $f\in W^{2s(p),p}(\Rn)$, $g\in W^{2s(p)-1,p}(\Rn)$ and $F\in L^{1}_{\loc}(\R;W^{2s(p)-1,p}(\Rn))$, then $u(t)\in L^{p}(\Rn)$ for all $t\in\R$. The exponent $2s(p)$ is again sharp here, for all but a discrete set of $t$.

Although \cite{SeSoSt91} already dealt with the $L^{p}(\Rn)$ regularity of \eqref{eq:wave-eqn} for smooth coefficients back in 1991, for a long time it was unclear whether these results could be extended to rough coefficients. The theory of spectral multipliers (see e.g.~\cite{DuOuSi02}) yields regularity bounds for divergence-form $L$ assuming only Gaussian heat kernel bounds. However, in this case the exponent $2s(p)$ 
is increased to at least $n|\frac{1}{p}-\frac{1}{2}|$, leading to weaker results than \eqref{eq:Laplacian} even in the case of the flat Laplacian. 

Recently, in \cite{Hassell-Rozendaal23} it was shown  that the regularity theory for smooth wave equations does indeed extend to a large class of rough wave equations. For example, if $(a_{ij})_{i,j=1}^{n}\subseteq C^{1,1}(\Rn)$ and $a_{j}=0$ for $0\leq j\leq n$, then the solution to \eqref{eq:wave-eqn} can be expressed as in \eqref{eq:solUj}, for operator families $(U_{0}(t))_{t\in\R}$ and $(U_{1}(t))_{t\in\R}$ satisfying
\begin{equation}\label{eq:Ujintro}
U_{k}(t):W^{2s(p)-k,p}(\Rn)\to L^{p}(\Rn)
\end{equation}
for all $t\in\R$, $k\in \{0,1\}$ and all $p$ in an open interval containing $2$. Moreover, \eqref{eq:Ujintro} holds for all $1<p<\infty$ if $(a_{ij})_{i,j=1}^{n}\subseteq C^{1,1}(\Rn)\cap C^{\frac{n+1}{2}+\veps}(\Rn)$ for $\veps>0$ arbitrarily small. However, \cite{Hassell-Rozendaal23} does not contain any regularity statements for the endpoints 
$p=1$ and $p=\infty$.

It might seem surprising that the endpoint exponents were not dealt with in \cite{Hassell-Rozendaal23}, given that the case $p=1$ plays a key role in the results for smooth coefficients. Indeed, the heart of the proof of \eqref{eq:FIOmap} is to show that a Fourier integral operator of order $-\frac{n-1}{2}$ maps the Hardy space $H^{1}(\Rn)$ of Fefferman and Stein into $L^{1}(\Rn)$. To this end, one can rely on the powerful 
atomic decomposition of $H^{1}(\Rn)$. Then one can combine interpolation of analytic families of operators with the fact that Fourier integral operators of order $0$ are bounded on $L^{2}(\Rn)$, to deduce \eqref{eq:FIOmap} for $1<p\leq 2$. Finally, duality yields \eqref{eq:FIOmap} for $2<p<\infty$ and implies that a Fourier integral operator of order $-\frac{n-1}{2}$ maps $L^{\infty}(\Rn)$ into $\BMO(\Rn)$. 

In fact, one can 
show that a Fourier integral operator $T$ of order $m\in\R$ satisfies 
\begin{equation}\label{eq:FIOmap2}
T:\HT^{s+2s(p)+m,p}(\Rn)\to \HT^{s,p}(\Rn)
\end{equation}
for all $1\leq p\leq \infty$ and $s\in\R$. Here 
$\HT^{s,p}(\Rn)=W^{s,p}(\Rn)=\lb D\rb^{-s}L^{p}(\Rn)$ for $1<p<\infty$, while $\HT^{s,1}(\Rn)=\lb D\rb^{-s}\HT^{1}(\Rn)$ for $\HT^{1}(\Rn)$ the local Hardy space,  and $\lb D\rb^{-s}=(1-\Delta)^{-s/2}$. Also, $\HT^{s,\infty}(\Rn)=\lb D\rb^{-s}\bmo(\Rn)$, where $\bmo(\Rn)$ is the dual of $\HT^{1}(\Rn)$. For $1<p<\infty$, \eqref{eq:FIOmap2} is just \eqref{eq:FIOmap}, but at the endpoints $p=1$ and $p=\infty$ it yields a stronger regularity result than above, given that $H^{1}(\Rn)\subsetneq\HT^{1}(\Rn)\subsetneq L^{1}(\Rn)$ and $L^{\infty}(\Rn)\subsetneq\bmo(\Rn)\subsetneq \BMO(\Rn)$. 


\subsection{Hardy spaces for Fourier integral operators}

Although the exponent $2s(p)$ in \eqref{eq:FIOmap2} is sharp under natural assumptions, and in particular for the solution operators to wave equations, it was observed by Smith in \cite{Smith98a} that \eqref{eq:FIOmap2} can nonetheless be improved, by measuring regularity using different function spaces. He introduced a function space, denoted $\HT^{1}_{FIO}(\Rn)$, which is invariant under Fourier integral operators of order zero and which satisfies the Sobolev embeddings\footnote{Strictly speaking, the space in \cite{Smith98a} coincides with what we will denote by $\HT^{s(1),1}_{FIO}(\Rn)$, and as such it satisfies slightly different embeddings. This makes no difference for the rest of the theory.}
\begin{equation}\label{eq:Sobolevp1}
\HT^{s(1),1}(\Rn)\subseteq\HT^{1}_{FIO}(\Rn)\subseteq\HT^{-s(1),1}(\Rn).
\end{equation}
By combining these embeddings with the invariance of $\HT^{1}_{FIO}(\Rn)$ under Fourier integral operators, one recovers the key $H^{1}(\Rn)\to L^{1}(\Rn)$ estimate in the proof of \eqref{eq:FIOmap}. However, given that the embeddings in \eqref{eq:Sobolevp1} are strict, the resulting regularity statement for Fourier integral operators in fact improves \eqref{eq:FIOmap2} for $p=1$. 

In \cite{HaPoRo20}, Smith's construction was extended to a full scale $(\Hp)_{1\leq p\leq \infty}$ of Hardy spaces for Fourier integral operators, and the associated Sobolev spaces are $\Hps=\lb D\rb^{-s}\Hp$ for $s\in\R$. These spaces are all invariant under Fourier integral operators of order zero, and 
the embeddings
\begin{equation}\label{eq:Sobolevintro}
\HT^{s+s(p),p}(\Rn)\subseteq\Hps\subseteq \HT^{s-s(p),p}(\Rn)
\end{equation}
hold. Again, the combination of \eqref{eq:Sobolevintro} and the invariance of $\Hps$ under Fourier integral operators 
improves \eqref{eq:FIOmap2}. However, the invariance of $\Hps$ under the solution operators to wave equations is powerful in its own right, and it is one of the key tools in the fixed-time regularity theory for rough wave equations, both in \cite{Hassell-Rozendaal23} and in this article.

\subsection{Main result}

We formulate our main result using the function space $C^{r}_{-}(\Rn)$ from Definition \ref{def:Crminus}, for $r>0$. Loosely speaking, $C^{r}_{-}(\Rn)$ consists of all bounded functions that have $C^{r}(\Rn)$ regularity if $r\notin\N$, and $C^{r-1,1}(\Rn)$ regularity if $r\in\N$. 
Our main result regarding the well-posedness of \eqref{eq:wave-eqn} is then as follows.

\begin{theorem}\label{thm:mainintro}
Suppose that $(a_{ij})_{i,j=1}^{n},(a_{j})_{j=0}^{n}\subseteq C^{r}_{-}(\Rn)$ for some $r\geq2$. Then there exist unique collections $(U_{0}(t))_{t\in\R},(U_{1}(t))_{t\in\R}$ such that, for all $p\in[1,\infty]$ and $s\in\R$ with $2s(p)+1<r$ and $-r+s(p)+1< s< r-s(p)$, and for all $u_{0}\in\Hps$, $u_{1}\in \Hpsm$, $F\in L^{1}_{\loc}(\R;\Hpsm)$ and $t_{0}>0$, the following 
holds:
\begin{enumerate}
\item\label{it:mainintro1} $U_{k}(t):\HT^{s-k,p}_{FIO}(\Rn)\to\Hps$ is bounded for all $t\in\R$ and $k\in\{0,1\}$, and $\sup_{|t|\leq t_{0}}\|U_{k}(t)\|_{\La(\HT^{s-k,p}_{FIO}(\Rn),\Hps)}<\infty$; 
\item\label{it:mainintro3} If $p<\infty$, then $U_{0}(\cdot)u_{0},U_{1}(\cdot)u_{1}\in C^{k}(\R; {\mathcal{H}^{s-k,p}_{FIO}(\Rn)})$ for 
$k\in\{0,1,2\}$;
\item\label{it:mainintro4} Set 
$u(t):=U_{0}(t)u_{0}+U_{1}(t)u_{1}-\int_{0}^{t}U_{1}(t-s)F(s)\ud s
$ 
for $t\in\R$. If $p<\infty$, then 
\[
u\in C(\R;\Hps)\cap C^{1}(\R;\Hpsm)\cap W^{2,1}_{\loc}(\R;\HT^{s-2,p}_{FIO}(\Rn)),
\]
$u(0)=u_{0}$, $\partial_{t}u(0)=u_{1}$, and 
$(D_{t}^{2}-L)u(t)=F(t)$ 
for almost all $t\in \R$.
\end{enumerate}
In particular, 
\begin{equation}\label{eq:regularityLpintro}
U_{k}(t):\HT^{s+s(p)-k,p}(\Rn)\to \HT^{s-s(p),p}(\Rn)
\end{equation}
for all $t\in\R$ and $k\in\{0,1\}$.
\end{theorem}

For $p<\infty$, the existence and uniqueness statement is a special case of Theorem \ref{thm:mainwave}. 
See Corollary \ref{cor:waveinfty} for $p=\infty$; 
here versions of \eqref{it:mainintro3} and \eqref{it:mainintro4} hold in a weak-star sense. 
Note that \eqref{eq:regularityLpintro} follows directly from \eqref{it:mainintro1} 
and \eqref{eq:Sobolevintro} (see also Corollary \ref{cor:divLp}). It implies 
that
\begin{align}
\label{eq:H1L1} U_{k}(t)\lb D\rb^{-\frac{n-1}{2}+k}&:H^{1}(\Rn)\to L^{1}(\Rn),\\
\nonumber U_{k}(t)\lb D\rb^{-\frac{n-1}{2}+k}&:L^{\infty}(\Rn)\to \BMO(\Rn),
\end{align}
for all $t\in\R$ and $k\in\{0,1\}$ if $n=r=2$, or if $n\geq3$ and $r>\frac{n+1}{2}$. 

One may in fact weaken the regularity assumptions on the coefficients somewhat. 
For $r>2$, the conclusion of Theorem \ref{thm:mainintro} holds if the $a_{ij}$ and $a_{j}$ are elements of the Zygmund space $C^{r}_{*}(\Rn)=B^{r}_{\infty,\infty}(\Rn)$ from Definition \ref{def:Zygmund}. This space coincides with $C^{r}_{-}(\Rn)$ for $r\notin\N$ but strictly contains it for $r\in\N$.  
Moreover, for general $r\geq 2$ one may include the endpoints of the Sobolev interval for $s$ if $(a_{ij})_{i,j=1}^{n},(a_{j})_{j=0}^{n}\in \HT^{r,\infty}(\Rn)\subsetneq C^{r}_{*}(\Rn)$. In particular, these endpoint values are allowed in Theorem \ref{thm:mainintro} if $r\in\N$, given that then $C^{r}_{-}(\Rn)\subseteq \HT^{r,\infty}(\Rn)$.

We assume that the $a_{j}$ have the same regularity as the $a_{ij}$ because it leads to the same restriction on $s$ as we would get if $a_{j}=0$. One may assume less regularity of the lower-order terms at the cost of shrinking the interval for $s$, cf.~Remark \ref{rem:lowerreg}. For example, $(a_{j})_{j=1}^{n}\subseteq C^{r-1}_{*}(\Rn)$ is allowed, and if $2s(p)\leq 1$ then one may let $a_{0}\in C^{\rho}_{*}(\Rn)$ for some $\rho>s(p)$. In fact, $a_{0}\in\HT^{s(p),\infty}(\Rn)$ is also allowed for $s=1$. 

The restriction $k\leq 2$ in \eqref{it:mainintro3} arises from the fact that we allow $r=2$; for larger $r$ one may consider larger $k$, cf.~Remark \ref{rem:boundedder}.

By comparing \eqref{eq:spectralsol} and Theorem \ref{thm:mainintro}, one obtains bounds for the spectral multipliers $\cos(t\sqrt{L})$ and $\sin(t\sqrt{L})L^{-1/2}$ if $L$ is a positive operator (see Corollary \ref{cor:spectral}).

For smooth coefficients, \eqref{eq:regularityLpintro} is a consequence of \cite{SeSoSt91}, and the $U_{k}(t)$ are Fourier integral operators, cf.~\eqref{eq:solUj} and \eqref{eq:FIOmap2}. The stronger regularity statements in \eqref{it:mainintro1}-\eqref{it:mainintro4} were obtained in \cite{Smith98a} for $p=1$, and in \cite{HaPoRo20} for $p>1$, again in the smooth setting. 
For $1<p<\infty$ and $a_{j}=0$, $0\leq j\leq n$, Theorem \ref{thm:mainintro} is contained in \cite{Hassell-Rozendaal23}. The main contribution of this article to Theorem \ref{thm:mainintro} concerns the endpoints $p=1$ and $p=\infty$. We emphasize that, for rough coefficients, the operators $U_{k}(t)$ in Theorem \ref{thm:mainintro} are not Fourier integral operators in the classical sense (see also Remark \ref{rem:parametrixwave}). 

\subsection{Paradifferential calculus and parametrices}

We will now indicate the main ingredients of the proof of Theorem \ref{thm:mainintro}. Doing so will also illuminate why the endpoint cases are more challenging for the $L^{p}(\Rn)$ regularity theory of rough waves. 
We follow a template that was introduced on $L^{2}(\Rn)$ by Smith in \cite{Smith98b}, 
and that was subsequently used in modified forms to study various problems involving rough wave equations (see e.g.~\cite{Tataru00,Tataru01,Tataru02,Smith06,Smith14,Chen-Smith19}).
In \cite{Hassell-Rozendaal23} these methods were adapted to the 
fixed-time $L^{p}(\Rn)$ regularity theory. 

Suppose for the moment that $a_{j}=0$ for $0\leq j\leq n$, as in \cite{Hassell-Rozendaal23}. One can use paradifferential calculus to decompose the second-order differential operator $L$ from \eqref{eq:defL} as $L=L_{1}+L_{2}$, where $L_{1}$ is a smooth second-order pseudodifferential operator, and $L_{2}$ is a rough pseudodifferential operator of lower differential order. More precisely, if $(a_{ij})_{i,j=1}^{n}\subseteq C^{r}_{-}(\Rn)$, then $L_{2}$ has differential order $2-r/2$. One can now rewrite \eqref{eq:wave-eqn} as $(D_{t}^{2}-L_{1})u(t)=F(t)+L_{2}u(t)$, and if $r\geq 2$ then $L_{2}u(t)$ should heuristically behave as the inhomogeneity $F(t)$, which is one order rougher than $u(t)$. Moreover, one can solve the smooth homogeneous equation $D_{t}^{2}u(t)=L_{1}u(t)$. Finally, one can insert the solution operators to the latter equation into Duhamel's principle, and iterate to remove error terms,  to obtain a solution to the full equation. 

This procedure does not change fundamentally if the $a_{j}$ are nonzero. One gets a decomposition $L=L_{1}+L_{2}+L_{3}$ for an $L_{3}$ which is of lower differential order by assumption; it is not necessary to apply a paradifferential decomposition to $L_{3}$.

There are three main difficulties in making this heuristic precise. Firstly, the term $L_{2}u(t)$ is of course not a bona fide inhomogeneity, and one has to iterate to get rid of the error that arises in this heuristic. In Duhamel's principle, the solution operators to the homogeneous equation are applied to the inhomogeneity (see e.g.~\eqref{eq:solUj}), and these solution operators are not bounded on $L^{p}(\Rn)$ unless $p=2$ or $n=1$. Hence such an iteration process will not converge on $L^{p}(\Rn)$. 
In \cite{Hassell-Rozendaal23}, this problem was dealt with by performing the iteration procedure on $\Hps$ instead, after which the Sobolev embeddings in \eqref{eq:Sobolevintro} yield the sharp $L^{p}(\Rn)$ regularity of the solution. 
It is for this reason that the Hardy spaces for Fourier integral operators play a vital role in \cite{Hassell-Rozendaal23} and in this article.

Secondly, although $L_{1}$ has a smooth, elliptic and real-valued symbol, it is not a differential operator and its symbol is not homogeneous. Hence one cannot directly apply the theory of Fourier integral operators to solve the equation $D_{t}^{2}u(t)=L_{1}u(t)$. Instead, if $(a_{ij})_{i,j=1}^{n}\subseteq C^{1,1}(\Rn)$, then 
one may use wave packet transforms and bicharacteristic flows to build a parametrix for this equation, and 
iterate to remove the error term that comes with this parametrix. In fact, it turns out to be convenient to first take an approximate square root of $L_{1}$ and solve the associated half-wave equation using this approach, but this does not make a fundamental difference. 
Of course, one has to prove 
that the iteration process converges, on $\Hps$. To do so, one proves concrete kernel estimates, the same as those satisfied by Fourier integral operators. Then one uses an atomic decomposition to prove the required statement for $p=1$, after which interpolation and duality deal with the remaining $p$. This step 
is similar to the proof 
of bounds for Fourier integral operators, 
and the exponent $p=1$ plays a key role. We will rely on \cite{Hassell-Rozendaal23} for this part of the argument.  

Finally, although $L_{2}$ has differential order at most $1$ if $r\geq 2$, it is a rough pseudodifferential operator, the symbol of which has the same spatial regularity as the $a_{ij}$. Hence one cannot rely on the theory of smooth pseudodifferential operators to conclude that $L_{2}$ behaves as a first-order operator on the relevant function spaces. 
There are standard results in paradifferential calculus that guarantee boundedness of pseudodifferential operators with very rough symbols on $L^{2}(\Rn)$, and even on $L^{p}(\Rn)$. However, as noted above, it is essential 
that our operators map between Hardy spaces for Fourier integral operators. A first boundedness result for rough pseudodifferential operators on $\Hps$ was obtained in \cite{Rozendaal23a}; this was subsequently improved in \cite{Rozendaal22}. Here the endpoint exponents appear to be more problematic, and it is this part of the argument that restricted \cite{Hassell-Rozendaal23} to $1<p<\infty$.

\subsection{Rough pseudodifferential operators}

In this article, we 
extend the bounds for rough pseudodifferential operators from \cite{Rozendaal23a,Rozendaal22} to the endpoints $p=1$ and $p=\infty$. 
We work with a class $C^{r}_{*}S^{m}_{1,1/2}$, introduced in Definition \ref{def:symbolrough}, of symbols that behave like elements of H\"{o}rmander's $S^{m}_{1,1/2}$ class but have the same spatial regularity as elements of the space $C^{r}_{*}(\Rn)$ from before. 
Our main result for the pseudodifferential operator $a(x,D)$ associated with such a symbol $a$ 
is as follows.

\begin{theorem}\label{thm:pseudointro}
Let $r>0$, $m\in\R$, $p\in[1,\infty]$, $a\in C^{r}_{*}S^{m}_{1,1/2}$ and $-r/2+s(p)<s<r-s(p)$. Then the following statements hold for each $\veps>0$:
\begin{enumerate}
\item\label{it:pseudointro1} If $r>4s(p)$, then $a(x,D):\HT^{s+m,p}_{FIO}(\Rn)\to \Hps$;
\item\label{it:pseudointro2} If $r\leq 4s(p)$, then $a(x,D):\HT^{s+(4s(p)-r)/2+\veps+m,p}_{FIO}(\Rn)\to \Hps$.
\end{enumerate} 
\end{theorem}

This result is a special case of Theorem \ref{thm:mainpseudo}. One may include the endpoint $s=r-s(p)$ of the Sobolev interval if $a$ is an element of the symbol class $\HT^{r,\infty}S^{m}_{1,1/2}\subsetneq C^{r}_{*}S^{m}_{1,1/2}$ from Definition \ref{def:symbolrough}. 

Combined with a paradifferential decomposition and bounds for $a(x,D)^{*}$, Theorem \ref{thm:pseudointro} implies that multiplication by a function in $C^{r}_{*}(\Rn)$ is bounded on $\Hps$ for $r>2s(p)$ and $-r+s(p)<s<r-s(p)$ (see Remark \ref{rem:multiplication}).

The importance of Theorem \ref{thm:pseudointro} for the proof of Theorem \ref{thm:mainintro} arises from the fact that the operator $L_{2}$ from before has a symbol in $C^{r}_{*}S^{2-r/2}_{1,1/2}$ if $(a_{ij})_{i,j=1}^{n}\subseteq C^{r}_{*}(\Rn)$, and in $\HT^{r,\infty}S^{2-r/2}_{1,1/2}$ if $(a_{ij})_{i,j=1}^{n}\subseteq \HT^{r,\infty}(\Rn)$. 
Moreover, results about multiplication operators can be used to deal with the lower-order term $L_{3}$.

For $1<p<\infty$, Theorem \ref{thm:pseudointro} is contained in \cite{Rozendaal22}; our contribution mainly concerns the endpoints $p=1$ and $p=\infty$. To deal with these, we do not directly improve upon the techniques in \cite{Rozendaal22,Rozendaal23a}, 
which do not seem to apply here. 

Instead, we rely on an extension of the definition of $\Hps$ from $1\leq p\leq \infty$ to all $0<p\leq \infty$. This extension of the theory to the full range of exponents is contained in the companion article \cite{LiRoSo25b}. Here, we use a framework from \cite{Rozendaal23a} and maximal function bounds to prove, in Theorem \ref{thm:pseudosmallp}, an initial estimate for rough pseudodifferential operators on $\Hps$ for $p\leq 1$. This estimate is weaker than the one in Theorem \ref{thm:pseudointro} for $p=1$. However, as in \cite{Rozendaal22}, one can use interpolation of rough symbol classes to obtain significant improvements. More precisely, by interpolating between the weaker estimates in Theorem \ref{thm:pseudosmallp} as $p\downarrow0$, and the bounds from \cite{Rozendaal22} as $p\downarrow 1$, one arrives at the statement in Theorem \ref{thm:pseudointro} for $p=1$. Finally, by applying the same line of ideas to $a(x,D)^{*}$ and using duality, one can also take care of the endpoint $p=\infty$.

\subsection{Organization of this article}

In Section \ref{sec:preliminaries} we collect background for the rest of the article. More precisely, we define the classical function spaces that we will encounter, we discuss a homogeneous structure on the cosphere bundle which will appear in several key places, we introduce the wave packets that appear in the definition of the Hardy spaces for Fourier integral operators, and we define the latter spaces and collect their basic properties. In Section \ref{sec:pseudos} we then prove our results on rough pseudodifferential operators, and in particular Theorem \ref{thm:pseudointro}. Finally, Section \ref{sec:roughwaves} contains our results for wave equations, including Theorem \ref{thm:mainintro}. 

\subsection{Notation and terminology}\label{subsec:notation}

The natural numbers are $\N=\{1,2,\ldots\}$, and $\Z_{+}:=\N\cup\{0\}$. Throughout, we fix an $n\in\N$ with $n\geq 2$. Our techniques also apply for $n=1$, but in this case simpler arguments suffice (see also \cite{Frey-Portal20}). 
We denote by $\lceil\gamma\rceil\in\Z$ the smallest integer larger than, or equal to, a $\gamma\in\R$.

For $\xi\in\Rn$ we write $\lb\xi\rb=(1+|\xi|^{2})^{1/2}$, and $\hat{\xi}=\xi/|\xi|$ if $\xi\neq0$. We use multi-index notation, where $\partial_{\xi}=(\partial_{\xi_{1}},\ldots,\partial_{\xi_{n}})$ and $\partial^{\alpha}_{\xi}=\partial^{\alpha_{1}}_{\xi_{1}}\ldots\partial^{\alpha_{n}}_{\xi_{n}}$
for $\xi=(\xi_{1},\ldots,\xi_{n})\in\Rn$ and $\alpha=(\alpha_{1},\ldots,\alpha_{n})\in\Z_{+}^{n}$. Moreover, $D_{j}:=-i\partial_{\xi_{j}}$ for $1\leq j\leq n$, and $D_{t}:=-i\partial_{t}$. 

The bilinear duality between a Schwartz function $g\in\Sw(\Rn)$ and a tempered distribution $f\in\Sw'(\Rn)$ is denoted by $\lb f,g\rb_{\Rn}$, whereas $\lb f,g\rb:=\lb f,\overline{g}\rb_{\Rn}$.  
The Fourier transform of $f$ is denoted by $\F f$ or $\widehat{f}$, and the Fourier multiplier with symbol $\ph$ is denoted by $\ph(D)$. 

The measure of a measurable subset $B$ of a measure space $\Omega$ will be denoted by $|B|$, and its indicator function by $\ind_{B}$.  
The H\"{o}lder conjugate of $p\in[1,\infty]$ is $p'$. 

The quasi-Banach space of continuous linear operators between quasi-Banach spaces $X$ and $Y$ is $\La(X,Y)$, 
and $\La(X):=\La(X,X)$. 

We write $f(s)\lesssim g(s)$ to indicate that $f(s)\leq C g(s)$ for all $s$ and a constant $C \geq0$ independent of $s$, and similarly for $f(s)\gtrsim g(s)$ and $g(s)\eqsim f(s)$.

\section{Preliminaries}\label{sec:preliminaries}
	
In this section we first collect some background material on classical function spaces, a homogeneous structure on the cosphere bundle and wave packets, and then we introduce the Hardy spaces for Fourier integral operators.

\subsection{Classical function spaces}\label{subsec:spacesclass}

Here we introduce the classical function spaces that appear in this article. 


Throughout, fix a $q\in C^{\infty}_{c}(\Rn)$ satisfying $q(\xi)=1$ for all $\xi\in\Rn$ with $|\xi|\leq 2$. For $0<p\leq \infty$, the local Hardy space $\HT^{p}(\Rn)$ from \cite{Goldberg79} consists of all $f\in\Sw'(\Rn)$ such that $q(D)f\in L^{p}(\Rn)$ and $(1-q(D))f\in H^{p}(\Rn)$, endowed with the quasi-norm
\[
\|f\|_{\HT^{p}(\Rn)}:=\|q(D)f\|_{L^{p}(\Rn)}+\|(1-q(D))f\|_{H^{p}(\Rn)}.
\]
Here $H^{p}(\Rn)$, for $0<p<\infty$, is the real Hardy space of Fefferman and Stein \cite{Fefferman-Stein72}, and $H^{\infty}(\Rn):=BMO(\Rn)=H^{1}(\Rn)^{*}$. We also write 
\[
\HT^{s,p}(\Rn):=\lb D\rb^{-s}\HT^{p}(\Rn)
\]
for $s\in\R$. Then $\HT^{s,p}(\Rn)=W^{s,p}(\Rn)$ for all $1<p<\infty$ and $s\in\R$, and $\HT^{s,p}(\Rn)^{*}=\HT^{-s,p'}(\Rn)$ for $1\leq p<\infty$. Moreover, $\HT^{\infty}(\Rn)=\bmo(\Rn)$. 





Next, fix a Littlewood--Paley decomposition $(\psi_{j})_{j=0}^{\infty}\subseteq C^{\infty}_{c}(\Rn)$. That is, 
\begin{equation}\label{eq:LittlePaley}
\sum_{j=0}^{\infty}\psi_{j}(\xi)=1
\end{equation}
for all $\xi\in\Rn$, $\psi_{0}(\xi)=0$ if $|\xi|>1$, $\psi_{1}(\xi)=0$ if $|\xi|\notin [\frac{1}{2},2]$, and $\psi_{j}(\xi)=\psi_{1}(2^{-j+1}\xi)$ for $j>1$. In fact, we may suppose that 
\[
\psi_{j}(\xi)=\psi_{0}(2^{-j}\xi)-\psi_{0}(2^{1-j}\xi)
\]
for all $j\geq1$ and $\xi\in\Rn$, as is implicitly used in the proof of Lemma \ref{lem:smoothing}. 

\begin{definition}\label{def:Zygmund}
For $s\in\R$, the \emph{Zygmund space} $C^{s}_{*}(\Rn)$ consists of those $f\in\Sw'(\Rn)$ such that $\psi_{j}(D)f\in L^{\infty}(\Rn)$ for all $j\geq0$, and 
\[
\|f\|_{C^{s}_{*}(\Rn)}:=\sup_{j\geq0}2^{js}\|\psi_{j}(D)f\|_{L^{\infty}(\Rn)}<\infty.
\]
\end{definition}
Note that $C^{s}_{*}(\Rn)$ is equal to the Besov space $B^{s}_{\infty,\infty}(\Rn)$. However, the present notation is more convenient for us, and it has been used frequently in paradifferential calculus (see e.g.~\cite{Taylor23c}).  

Let $s>0$, and write $s=l+t$ for $l\in\Z_{+}$ and $t\in (0,1]$. Recall that, if $s\notin \N$, then $C^{s}(\Rn)$ consists of all $f\in C^{l}(\Rn)$ such that for each $\alpha\in\Z_{+}^{n}$ with $|\alpha|=l$, the partial derivative $\partial^{\alpha}f$ is H\"{o}lder continuous with parameter $t$. Moreover, $C^{l,1}(\Rn)$ consists of all $f\in C^{l}(\Rn)$ such that $\partial^{\alpha}f$ is Lipschitz for each $\alpha\in\Z_{+}^{n}$ with $|\alpha|=l$. It will be notationally convenient to consider a single scale to denote these spaces. 

\begin{definition}\label{def:Crminus}
Let $s=l+t>0$ for $l\in\Z_{+}$ and $t\in (0,1]$. Then $C^{s}_{-}(\Rn)$ consists of all $l$ times continuously differentiable $f:\Rn\to\C$ such that 
\[
\|f\|_{C^{s}_{-}(\Rn)}:=\max_{|\alpha|\leq l}\sup_{x\in\Rn}|\partial_{x}^{\alpha}f(x)|+\max_{|\alpha|=l}\sup_{x\neq y}\frac{|\partial_{x}^{\alpha}f(x)-\partial_{x}^{\alpha}f(y)|}{|x-y|^{t}}<\infty.
\]
\end{definition}


It is instructive to compare these spaces using embeddings (see \cite{Triebel10}). For example, 
\[
C^{s+\veps}_{*}(\Rn)\subsetneq \HT^{s,\infty}(\Rn)\subsetneq C^{s}_{*}(\Rn)
\]
for all $s\in\R$ and $\veps>0$. Moreover, let
 $s=l+t>0$ for $l\in\Z_{+}$ and $t\in(0,1]$. Then 
\[
\HT^{s,\infty}(\Rn)\subsetneq C^{s}_{*}(\Rn)=C^{s}_{-}(\Rn)=C^{s}(\Rn)\cap L^{\infty}(\Rn)
\]
if $s\notin\N$, i.e.~if $t\in(0,1)$, and
\[
C^{l,1}(\Rn)\cap L^{\infty}(\Rn)=C^{s}_{-}(\Rn)\subsetneq \HT^{s,\infty}(\Rn)\subsetneq C^{s}_{*}(\Rn)
\]
if $s\in\N$, i.e.~if $t=1$. 

\subsection{A homogeneous structure on the cosphere bundle}\label{subsec:metric}

In this subsection, we collect some background on a metric measure space that plays a crucial role in the theory of the Hardy spaces for Fourier integral operators. The metric arises from contact geometry, but in this article we will only use a few of its properties. 
See \cite{HaPoRo20} for more on the material presented here.

The cotangent bundle $T^{*}\Rn$ of $\Rn$ is identified with $\Rn\times\Rn$, and 
\[
o:=\Rn\times\{0\}\subseteq T^{*}\Rn
\]
is the zero section. We denote elements of the unit sphere $S^{n-1}\subseteq\Rn$ by $\w$ or $\nu$, and we endow $S^{n-1}$ with the unit normalized measure $\ud\w$ and the standard Riemannian metric $g_{S^{n-1}}$. 
Let $\Sp:=\Rn\times S^{n-1}$ be the cosphere bundle of $\mathbb{R}^{n}$, endowed with the measure $\ud x\ud\w$ and the product metric $dx^{2}+g_{S^{n-1}}$. The $1$-form $\w\cdot dx$ on $\Sp$ determines a contact structure on $\Sp$, 
which in turn gives rise to the following sub-Riemannian metric:
\[
d((x,\omega),(y,v)):={\rm \inf\limits_{\gamma}}\int_{0}^{1}|\gamma^{\prime}(s)|\ud s,
\]
for $(x,\omega),(y,\nu)\in \Sp$. Here the infimum is taken over all piecewise Lipschitz $\gamma:[0,1]\rightarrow \Sp$ such that $\gamma(0)=(x,\omega)$, $\gamma(1)=(y,\nu)$ and $\w\cdot dx(\gamma^{\prime}(s))=0$ for almost all $s\in[0,1]$. Moreover, $|\gamma^{\prime}(s)|$ is the length of $\gamma^{\prime}(s)$ with respect to $dx^{2}+dg_{S^{n-1}}$.

It is shown in \cite[Lemma 2.1]{HaPoRo20} that
\begin{equation}\label{eq:equivmetric}
d((x,\omega),(y,\nu))\eqsim \big{(}|\omega\cdot(x-y)|+|x-y|^{2}+|\omega-\nu|^{2}\big{)}^{1/2}
\end{equation}
for implicit constants independent of $(x,\w),(y,\nu)\in\Sp$, and we will only work with this equivalent expression for the metric.


Denote by $B_{\tau}(x,\w)$ the open ball around $(x,\w)\in\Sp$ of radius $\tau>0$ with respect to the metric $d$. 
Then $|B_{\tau}(x,\w)|\eqsim \tau^{2n}$ if $0<\tau<1$, and $|B_{\tau}(x,\w)|\eqsim \tau^{n}$ if $\tau\geq 1$ (see \cite[Lemma 2.3]{HaPoRo20}). This implies in particular that 
$(\Sp,d,\ud x\ud\omega)$ is a doubling metric measure space.



Now, given that $(\Sp,d,\ud x\ud\w)$ is a doubling metric measure space, the (centered) {vector-valued} Hardy--Littlewood maximal operator $\Ma$ acts boundedly on $L^{p}(\Sp;L^{q}(0,\infty))$ for all $p,q\in(1,\infty)$, where $(0,\infty)$ is endowed with the Haar measure $\frac{\ud\sigma}{\sigma}$ (see e.g.~\cite[Theorem 1.5]{Tozoni04}). We will use this when dealing with the maximal operator $\Ma_{\la}$ of index $\lambda>0$, given by 
\begin{equation}\label{eq:maxHL}
\mathcal{M}_{\lambda}g(x,\omega):=\big(\Ma(|g|^\lambda)(x,\omega)\big)^{1/{\lambda}}
\end{equation}
for $g\in L^{\lambda}_{\loc}(\Sp)$ and $(x,\w)\in\Sp$. 

Finally, the homogeneous structure on the cosphere bundle also allows one to rely on the established theory of tent spaces. Apart from a brief appearance in the proof of Theorem \ref{thm:parametrix}, these spaces will not play an explicit role in this article, but they are fundamental when deriving various basic properties of the Hardy spaces for Fourier integral operators.

\subsection{Wave packet transforms}\label{subsec:transforms}

In this subsection we introduce the wave packets and parabolic cutoffs that are used to define the Hardy spaces for Fourier integral operators. We refer to \cite[Section 4]{HaPoRo20} and \cite[Section 3]{Rozendaal21} for more on this material.

Throughout, let $\Psi\in C^{\infty}_{c}(\mathbb{R}^{n})$ be a non-negative radial function such that $\Psi(\xi)=0$ for all $\xi\in\Rn$ with $|\xi|\notin[\frac{1}{2},2]$, and
\[
\int_{0}^{\infty}\Psi(\sigma\xi)^{2}\frac{\ud\sigma}{\sigma}=1
\]
if $\xi\neq 0$. Fix a non-negative radial $\varphi \in C_{c}^{\infty}(\mathbb{R}^{n})$ such that $\varphi\equiv 1$ in a small neighborhood of zero, and $\varphi(\xi)=0$ for $|\xi|>1$. 
Set $c_{\sigma}:=\big{(}\int_{S^{n-1}}\varphi(\frac{e_{1}-\nu}{\sqrt{\sigma}})^{2}\ud\nu\big{)}^{-1/2}$ for $\sigma>0$, where $e_{1}$ is the first basis vector of $\Rn$. 
For $\w\in S^{n-1}$, $\sigma>0$ and $\xi\in\Rn\setminus\{0\}$, set 
\begin{equation}\label{eq:defpackets}
\psi_{\w,\sigma}(\xi)=\Psi(\sigma\xi)c_{\sigma}\ph\big(\tfrac{\hat{\xi}-\w}{\sqrt{\sigma}}\big)
\end{equation}
and $\psi_{\w,\sigma}(0):=0$. Moreover, let
\[
\rho(\xi):=\Big(1-\int_{0}^{1}\Psi(\sigma \xi)^{2}\frac{\ud\sigma}{\sigma}\Big)^{1/2}
\]
for $\xi\in\Rn$. Then $\rho\in C^{\infty}_{c}(\Rn)$, with $\rho(\xi)=1$ for $|\xi|\leq 1/2$, and $\rho(\xi)=0$ if $|\xi|\geq 2$.

As shown in \cite[Lemma 4.1]{HaPoRo20}, these wave packets have the following properties. 

\begin{lemma}\label{lem:packetbounds}
For all $\w\in S^{n-1}$ and $\sigma\in(0,1)$, one has $\psi_{\w,\sigma}\in C^{\infty}_{c}(\Rn)$. Each 
$\xi\in\supp(\psi_{\w,\sigma})$ satisfies $\frac{1}{2}\sigma^{-1}\leq |\xi|\leq 2\sigma^{-1}$ and $|\hat{\xi}-\w|\leq 2\sqrt{\sigma}$. 
For each $N\geq0$ there exists a $C_{N}\geq0$, independent of $\w$ and $\sigma$, such that
\[
|\F^{-1}(\psi_{\w,\sigma})(x)|\leq C_{N}\sigma^{-\frac{3n+1}{4}}(1+\sigma^{-1}|x|^{2}+\sigma^{-2}(\w\cdot x)^{2})^{-N}
\]
for all $x\in\Rn$.
\end{lemma}

In \cite{HaPoRo20} and \cite{LiRoSo25b}, these functions are used to define a wave packet transform
\begin{equation}\label{eq:defW}
Wf(x,\w,\sigma):=\begin{cases}\psi_{\w,\sigma}(D)f(x)&\text{if }0<\sigma<1,\\
\ind_{[1,e]}(\sigma)\rho(D)f(x)&\text{if }\sigma\geq1,\end{cases}
\end{equation}
for $f\in\Sw'(\Rn)$, $(x,\w)\in\Sp$ and $\sigma>0$. This wave packet transform and its adjoint $V$ can be used to derive various properties of the Hardy spaces for Fourier integral operators from those of tent spaces over the cosphere bundle. When doing so one relies crucially on the identity
\begin{equation}\label{eq:repro}
VWf=f,
\end{equation}
for $f\in\Sw'(\Rn)$. Here we merely mention this connection, given that the transforms $W$ and $V$ will only appear in the proof of Theorem \ref{thm:parametrix}. 

For $f\in\Sw'(\Rn)$, $(x,\w)\in\Sp$ and $\sigma>0$, we will also use the notation $W_{\sigma}f(x,\w):=Wf(x,\w,\sigma)$ in the following result from \cite{LiRoSo25b}, concerning the maximal function $\mathcal{M}_{\la}$ from \eqref{eq:maxHL}.

\begin{lemma}\label{lem:maxineq}
Let $\lambda>0$ and $N>n/\la$. Then there exists a $C\geq0$ such that
\[
\sigma^{-n}\int_{\Sp}\frac{|\psi_{\nu,{\sigma}}(D)f(y)|}{(1+{\sigma}^{-1}d((x,\omega),(y,\nu))^{2})^{N}}\ud y\ud\nu
\leq C\Ma_{\lambda}(W_{\sigma}f)(x,\omega)
\]
for all $f\in \Sw'(\Rn)$, $(x,\w)\in\Sp$ and $\sigma\in(0,1)$.
\end{lemma}

\begin{remark}\label{rem:maxineq}
Lemma \ref{lem:maxineq} also holds for different wave packets than those in \eqref{eq:defpackets}, as long as they have properties similar to those in Lemma \ref{lem:packetbounds}. 
This follows from the proof of Lemma \ref{lem:maxineq} in \cite{LiRoSo25b}, and it will be used in the proof of Theorem \ref{thm:pseudosmallp}.  
\end{remark}

Next, we introduce parabolic cutoffs associated with these wave packets. For $\w\in S^{n-1}$ and $\xi\in\Rn$, set
\begin{equation}\label{eq:phw}
\ph_{\omega}(\xi):=\int_{0}^{4}\psi_{\w,\tau}(\xi)\frac{\ud\tau}{\tau}.
\end{equation}
Some properties of $(\ph_{\w})_{\w\in S^{n-1}}\subseteq C^{\infty}(\Rn)$ are as follows (see \cite[Remark 3.3]{Rozendaal21}):
\begin{enumerate}
\item\label{it:phiproperties1} For all $\w\in S^{n-1}$ and $\xi\neq0$ one has $\ph_{\w}(\xi)=0$ if $|\xi|<\frac{1}{8}$ or $|\hat{\xi}-\w|>2|\xi|^{-1/2}$.
\item\label{it:phiproperties2} For all $\alpha\in\Z_{+}^{n}$ and $\beta\in\Z_{+}$ there exists a $C_{\alpha,\beta}\geq0$ such that
\[
|(\w\cdot \partial_{\xi})^{\beta}\partial^{\alpha}_{\xi}\ph_{\w}(\xi)|\leq C_{\alpha,\beta}|\xi|^{\frac{n-1}{4}-\frac{|\alpha|}{2}-\beta}
\]
for all $\w\in S^{n-1}$ and $\xi\neq0$.
\item\label{it:phiproperties3}
There exists a radial $m\in S^{(n-1)/4}(\Rn)$ such that, for each $f\in\Sw'(\Rn)$ satisfying $\supp(\wh{f}\,)\subseteq \{\xi\in\Rn\mid |\xi|\geq\frac{1}{2}\}$, one has
\begin{equation}\label{eq:phiproperties3}
f=\int_{S^{n-1}}m(D)\ph_{\nu}(D)f\ud\nu.
\end{equation}
\end{enumerate}
In \eqref{it:phiproperties3}, $S^{(n-1)/4}(\Rn)$ consists of the standard pseudodifferential symbols of order $(n-1)/4$ that only depend on the fiber variable.

\subsection{Hardy spaces for Fourier integral operators}\label{subsec:HpFIO}

In this section we define the Hardy spaces for Fourier integral operators, and we collect their basic properties. Proofs of the statements below can be found in \cite{HaPoRo20,LiRoSo25b}.

We define $\Hp$ using the collection $(\ph_{\w})_{\w\in S^{n-1}}$ from \eqref{eq:phw}. Also recall that $q\in C^{\infty}_{c}(\Rn)$ satisfies $q(\xi)=1$ for all $\xi\in\Rn$ with $|\xi|\leq 2$. 

\begin{definition}\label{def:HpFIO}
Let $0<p\leq \infty$. Then $\Hp$ consists of all $f\in\Sw'(\Rn)$ such that $q(D)f\in L^{p}(\Rn)$, $\ph_{\w}(D)f\in \HT^{p}(\Rn)$ for almost all $\w\in S^{n-1}$, and 
\[
\|f\|_{\Hp}:=\|q(D)f\|_{L^{p}(\Rn)}+\Big(\int_{S^{n-1}}\|\ph_{\w}(D)f\|_{\HT^{p}(\Rn)}^{p}\ud\w\Big)^{1/p}<\infty.
\]
Moreover, $\Hps:=\lb D\rb^{-s}\Hp$ for $s\in\R$. 
\end{definition}

Now, $\Hps$ is a quasi-Banach space for all $0<p\leq \infty$ and $s\in\R$, and a Banach space if $p\geq 1$. Up to quasi-norm equivalence, $\Hp$ is independent of the choice of low-frequency cutoff $q$ and of the choice of wave packets $\psi_{\w,\sigma}$ in \eqref{eq:defpackets}, which in turn are used to define the $\ph_{\w}$. 

The continuous embeddings 
\[
\Sw(\Rn)\subseteq \Hps\subseteq \Sw'(\Rn)
\]
hold for all $0<p\leq\infty$ and $s\in\R$, and the first embedding is dense if $p<\infty$. In fact, the Schwartz functions with compactly supported Fourier transform then lie dense in $\Hps$.

It is of crucial importance in this article that the Hardy spaces for Fourier integral operators form a complex interpolation scale. That is, let $p_{0},p_{1}\in(0,\infty]$ be such that $(p_{0},p_{1})\neq (\infty,\infty)$, and let $p\in(0,\infty)$, $s_{0},s_{1},s\in\R$ and $\theta\in(0,1)$ be such that $\frac{1}{p}=\frac{1-\theta}{p_{0}}+\frac{\theta}{p_{1}}$ and $s=(1-\theta)s_{0}+\theta s_{1}$. Then
\begin{equation}\label{eq:intHpFIO}
[\HT^{s_{0},p_{0}}_{FIO}(\Rn),\HT^{s_{1},p_{1}}_{FIO}(\Rn)]_{\theta}=\Hps.
\end{equation}
It should also be noted that the theory of complex interpolation is more involved for quasi-Banach spaces than it is for Banach spaces. We refer to \cite{Kalton-Mitrea98,LiRoSo25b} for more on these subtleties.

The spaces $\Hps$ behave in a natural manner under duality: 
\begin{equation}\label{eq:dualHpFIO}
(\Hps)^{*}=\HT^{-s,p'}_{FIO}(\Rn)
\end{equation}
for all $1\leq p<\infty$ and $s\in\R$, with equivalent norms. Here the duality pairing is the standard distributional pairing $\lb f,g\rb_{\Rn}$ between $f\in \HT^{-s,p'}_{FIO}(\Rn)\subseteq\Sw'(\Rn)$ and $g\in\Sw(\Rn)\subseteq\Hps$. One can also give a somewhat explicit description of the dual of $\Hps$ for $p<1$, but that characterization will not play a role in the present article.

For all $0<p\leq \infty$ and $s\in\R$, as the name suggests,  $\Hps$ is invariant under Fourier integral operators of order zero, associated with a local canonical graph and having a compactly supported Schwartz kernel. In fact, the assumption of compact support can be dropped for operators with kernels in a suitable standard form. As a specific example, pseudodifferential operators with symbols in the class $S^{0}_{1,1/2}$ (see \eqref{eq:Hormanderclass}) are bounded on $\Hps$. 

Next, the Hardy spaces for Fourier integral operators satisfy suitable Sobolev embeddings. The first such embedding extends \eqref{eq:Sobolevintro} to all $0<p\leq \infty$ and $s\in\R$:
\begin{equation}\label{eq:Sobolev}
\HT^{s+s(p),p}(\Rn)\subseteq\Hps\subseteq \HT^{s-s(p),p}(\Rn).
\end{equation} 
Moreover,  
\begin{equation}\label{eq:fracintHpFIO1}
\HT^{s+n(\frac{1}{p_{0}}-\frac{1}{p_{1}}),p_{0}}_{FIO}(\Rn)\subseteq \HT^{s,p_{1}}_{FIO}(\Rn)
\end{equation}
for all $0<p_{0}\leq p_{1}\leq \infty$, and
\[
\HT^{-s+n(\frac{1}{p_{0}}+\frac{1}{p_{2}}-1),p_{0}}_{FIO}(\Rn)\subseteq (\HT^{s,p_{2}}_{FIO}(\Rn))^{*}
\]
for $0<p_{2}<1$.

We also note that, for $0<p\leq 1$, any element of $\Hp$ can be decomposed into so-called coherent molecules, associated with balls in the cosphere bundle. However, this decomposition will not play an explicit role in this article.

Finally, we note that the Hardy spaces for Fourier integral operators were not originally defined as above. Instead, in \cite{HaPoRo20,LiRoSo25b}, $\Hp$ is implicitly viewed as a subspace of a tent space $T^{p}(\Sp)$ over the cosphere bundle, using the wave packet transform from \eqref{eq:defW}. It was then shown in \cite{Rozendaal22,FaLiRoSo23,LiRoSo25b} that one may also characterize these spaces as in Definition \ref{def:HpFIO}. Moreover, for all $0<p<\infty$ and $s\in\R$, one has
\begin{equation}\label{eq:equivchar}
\|f\|_{\Hps}\eqsim\|q(D)f\|_{L^{p}(\Rn)}+\Big(\int_{S^{n-1}}\|\ph_{\w}(D)f\|_{\HT^{s,p}(\Rn)}^{p}\ud\w\Big)^{1/p}
\end{equation}
for all $f\in\Sw'(\Rn)$ such that either side is finite.

\section{Rough pseudodifferential operators}\label{sec:pseudos}

In this section we introduce the relevant classes of pseudodifferential symbols, and we prove our main result for rough pseudodifferential operators.

\subsection{Symbol classes and symbol smoothing}\label{subsec:symbols}

In this subsection we collect some background on pseudodifferential symbols that arise in paradifferential calculus.

First recall that, for $m\in\R$ and $\rho,\delta\in[0,1]$, H\"{o}rmander's class $S^{m}_{\rho,\delta}$ consists of all $a\in C^{\infty}(\R^{2n})$ such that 
\begin{equation}\label{eq:Hormanderclass}
\sup_{(x,\eta)\in\R^{2n}}\lb \eta\rb^{-m+|\alpha|\rho-|\beta|\delta}|\partial_{x}^{\beta}\partial_{\eta}^{\alpha}a(x,\eta)|<\infty
\end{equation}
for all $\alpha,\beta\in\Z_{+}^{n}$. The pseudodifferential operator $a(x,D):\Sw(\Rn)\to\Sw'(\Rn)$ with symbol $a$ is then given by
\begin{equation}\label{eq:pseudodef}
a(x,D)f(x):=\frac{1}{(2\pi)^{n}}\int_{\Rn}e^{ix\cdot\eta}a(x,\eta)\wh{f}(\eta)\ud\eta
\end{equation}
for $f\in\Sw(\Rn)$ and $x\in\Rn$.

We now introduce symbols that have less regularity than elements of $S^{m}_{1,\delta}$ from \eqref{eq:Hormanderclass}, measured in terms of the function spaces from Section \ref{subsec:spacesclass}.

\begin{definition}\label{def:symbolrough}
Let $r>0$, $m\in\R$, $\delta\in[0,1]$ and $l\in\Z_{+}$, and let $X\in\{C^{r}_{*},\HT^{r,\infty},C^{r}_{-}\}$. Then $XS^{m,l}_{1,\delta}$ consists of all $a:\R^{2n}\to\C$ 
such that the following properties hold:
\begin{enumerate}
\item\label{it:symbolrough1} $a(x,\cdot)\in C^{l}(\Rn)$ for all $x\in\Rn$, and
\[
\sup_{(x,\eta)\in\R^{2n}}\lb\eta\rb^{-m+|\alpha|}|\partial_{\eta}^{\alpha}a(x,\eta)|<\infty
\]
for each $\alpha\in\Z_{+}^{n}$ with $|\alpha|\leq l$;
\item\label{it:symbolrough2} $\partial_{\eta}^{\alpha}a(\cdot,\eta)\in X(\Rn)$ for all $\eta\in\Rn$ and $\alpha\in\Z_{+}^{n}$ with $|\alpha|\leq l$, and
{{\[
\sup_{\eta\in\Rn}\lb\eta\rb^{-m+|\alpha|-r\delta}\|\partial_{\eta}^{\alpha}a(\cdot,\eta)\|_{X(\Rn)}<\infty.
\]
If $X=C^{r}_{-}$, then additionally
\[
\sup_{\eta\in\Rn}\lb\eta\rb^{-m+|\alpha|-j\delta}\|\partial_{\eta}^{\alpha}a(\cdot,\eta)\|_{C^{j}_{-}(\Rn)}<\infty
\]
for all integer $0\leq j\leq r$.}}
\end{enumerate} 
Moreover, $XS^{m}_{1,\delta}:=\cap_{l\in\N}XS^{m,l}_{1,\delta}$. 
\end{definition}

Clearly, we can extend the definition of the pseudodifferential operator $a(x,D):\Sw(\Rn)\to\Sw'(\Rn)$ from \eqref{eq:pseudodef} to symbols $a$ as in Definition \ref{def:symbolrough}.

We will also work with symbols that have more regularity than a typical element of $S^{m}_{1,\delta}$, in the following sense.

\begin{definition}\label{def:symbolsmooth}
Let $r>0$, $m\in\R$ and $\delta\in[0,1]$. Then $\A^{r}S^{m}_{1,\delta}$ consists of all $a\in S^{m}_{1,\delta}$ such that 
\[
\sup_{\eta\in\Rn}\lb\eta\rb^{-m+|\alpha|-s\delta}\|\partial_{\eta}^{\alpha}a(\cdot,\eta)\|_{C^{r+s}_{-}(\Rn)}<\infty
\]
for all $s\geq0$. 
\end{definition}

The following lemma shows that this class behaves well with respect to duality. 

\begin{lemma}\label{lem:Aduality}
Let $m\in\R$, $\delta\in[0,1)$ and $a\in S^{m}_{1,\delta}$. Then there exists an $\tilde{a}\in S^{m}_{1,\delta}$ such that $a(x,D)^{*}=\tilde{a}(x,D)$. 
If $a\in\A^{r}S^{m}_{1,\delta}$ for $r>0$, then $\tilde{a}\in \A^{r}S^{m}_{1,\delta}$ as well. Moreover, if $r\geq 1$, then $\tilde{a}-\overline{a}\in S^{m-1}_{1,\delta}$.
\end{lemma}
\begin{proof}
The first statement is classical, cf.~\cite[Proposition 0.3.B]{Taylor91}, and it follows from a standard asymptotic expansion. The same expansion yields the remaining statements, with the second statement also being contained in \cite[Lemma 4.11]{Hassell-Rozendaal23}. 
\end{proof}

For $r>0$, $m\in\R$ and $\delta\in[0,1]$, an $a\in C^{r}_{*}S^{m}_{1,\delta}$ is \emph{elliptic} if there exist $\kappa,R>0$ such that $|a(x,\eta)|>\kappa|\eta|^{m}$ for all $x,\eta\in\Rn$ with $|\eta|\geq R$. We say that $a$ is \emph{homogeneous of degree $m$ for $|\eta|\geq1$} if $a(x,\lambda\eta)=\lambda^{m}a(x,\eta)$ for all $x,\eta\in\Rn$ and $\lambda\geq1$ with $|\eta|\geq1$. The following weaker notion of homogeneity arises naturally in paradifferential calculus (see e.g.~Lemma \ref{lem:smoothing}).

\begin{definition}\label{def:asymphom}
Let $r>0$, $m\in\R$, $\delta\in[0,1]$ and $b\in S^{m}_{1,\delta}$. Then $b$ is \emph{asymptotically homogeneous of degree $m$} if there exists an $a\in \Cr S^{m}_{1,0}$, homogeneous of degree $m$ for $|\eta|\geq1$, such that $[(x,\eta)\mapsto (\eta\cdot \partial_{\eta}-m)b(x,\eta)]\in S^{m-1}_{1,\delta}$ and $a-b\in \Cr S^{m-1}_{1,\delta}$. In this case, $a$ is a \emph{limit} of $b$.
\end{definition}

In this definition, $a(x,\eta)$ is uniquely determined for $|\eta|\geq1$, cf.~\cite[Remark 4.4]{Hassell-Rozendaal23}.

Next, we include a statement, from \cite[Proposition 4.2]{Rozendaal22}, that will be used to interpolate in the proof of the main result of this section. Throughout, we write
\[
\St:=\{z\in\C\mid 0<\Real(z)<1\}.
\]

\begin{lemma}\label{lem:symbolinter}
Let $r,c>0$, $m\in\R$, $\delta\in[0,1]$, $l\in\N$ and $\kappa,\lambda\in\R$ be such that {{$r>\max(\lambda,\ka+\lambda)$}}. Then there exists a $C\geq0$ such that the following holds. Let $a\in C^{r}_{*}S^{m,l}_{1,\delta}$ be such that $\supp(\F a(\cdot,\eta))\subseteq\{\xi\in\Rn\mid |\xi|\geq c|\eta|^{\delta}\}$ for all $\eta\in\Rn$. For $z\in\overline{\St}$ and $x,\eta\in\Rn$, set
\[
a_{z}(x,\eta):=e^{(\kappa z+\lambda)^{2}}\lb \eta\rb^{-\delta (\kappa z+\lambda)}\big(\lb D\rb^{\kappa z+\lambda}a(\cdot,\eta)\big)(x).
\]
Then $a_{z}\in C^{r-\Real(\kappa z+\lambda)}_{*}S^{m,l}_{1,\delta}$ and $\|a_{z}\|_{C^{r-\Real(\kappa z+\lambda)}_{*}S^{m,l}_{1,\delta}}\leq C\|a\|_{C^{r}_{*}S^{m,l}_{1,\delta}}$.
\end{lemma}

The following lemma will be used frequently in this section. It concerns a symbol decomposition that goes back to \cite{Coifman-Meyer78} and that was also used in e.g.~\cite{Marschall88,Rozendaal22,Rozendaal23a}. 

\begin{lemma}\label{lem:symdecomp}
Let $r>0$, $\delta\in[0,1]$, $p\in(0,1]$ and $l\in\N$ be such that $l\geq 1+{\lceil n/p\rceil} $. Then there exists a $C\geq0$ such that the following holds. For each $a\in C^{r}_{*}S^{0,l}_{1,\delta}$, there exist sequences $(\lambda_{\beta})_{\beta\in\Z^{n}}\in\ell^{p}(\Z^{n})$, $(a_{k,\beta})_{k\in\Z_{+},\beta\in\Z^{n}}\subseteq C^{r}_{*}(\Rn)$ and $(\chi_{k,\beta})_{k\in\Z_{+},\beta\in\Z^{n}}\subseteq C^{\infty}_{c}(\Rn)$ with the following properties:
\begin{enumerate}
\item\label{it:symdecomp1} $a(x,\eta)=\sum_{\beta\in\Z^{n}}\lambda_{\beta}\sum_{k=0}^{\infty}a_{k,\beta}(x)\chi_{k,\beta}(\eta)$ for all $(x,\eta)\in\Rn$;
\item\label{it:symdecomp2} $\|a_{k,\beta}\|_{C^{r}_{*}(\Rn)}\leq C2^{k\delta r}\|a\|_{C^{r}_{*}S^{0,l}_{1,\delta}}$ for all $k\in\Z_{+}$ and $\beta\in\Z^{n}$;
\item\label{it:symdecomp3} $\lb\eta\rb^{|\alpha|}|\partial_{\eta}^{\alpha}\chi_{k,\beta}(\eta)|\leq C\|a\|_{C^{r}_{*}S^{0,l}_{1,\delta}}$ for all $k\in\Z_{+}$, $\beta\in\Z^{n}$, $\alpha\in\Z_{+}^{n}$ with $|\alpha|\leq l-1-\lceil n/p\rceil$, and $\eta\in\Rn$;
\item\label{it:symdecomp4} For all $\beta\in\Z^{n}$ and $\eta\in\Rn$, one has $\chi_{k,\beta}(\eta)=\chi_{1,\beta}(2^{-k+1}\eta)$ for $k\geq1$, $\chi_{1,\beta}(\eta)=0$ if $|\eta|\notin[1/2,2]$, $|\chi_{1,\beta}(\eta)|=1$ if $|\eta|=1$, and $\chi_{0,\beta}(\eta)=0$ if $|\eta|>1$.
\end{enumerate}
Moreover, if there exist $c>0$ and $\gamma\in [1/2,1]$ such that
\begin{equation}\label{eq:Fouriersupport}
\supp(\F a(\cdot,\eta))\subseteq\{\xi\in\Rn\mid c|\eta|^{1/2}\leq |\xi|\leq  \tfrac{1}{16}(1+|\eta|)^{\gamma}\}
\end{equation}
for all $\eta\in\Rn$, then one may also suppose that $\supp(\F a_{0,\beta})\subseteq \{\xi\in\Rn\mid |\xi|\leq 2\}$ and $\supp(\F a_{k,\beta})\subseteq\{\xi\in\Rn\mid c2^{(k-2)/2}\leq |\xi|\leq 2^{k\gamma-3}\}$ for all $k\in\N$ and $\beta\in\Z^{n}$. 
\end{lemma}
\begin{proof}
The statement is almost contained in \cite[Proposition 2.1]{Marschall88}, although there it is assumed a priori that $a\in \HT^{r,\infty}S^{0}_{1,1/2}$, and there is no claim regarding \eqref{eq:Fouriersupport}. The same proof can be used here (see also the proof of \cite[Theorem 4.1]{Rozendaal23a}). 
\end{proof}

\begin{remark}\label{rem:decompadditional}
Suppose that \eqref{eq:Fouriersupport} holds. It then follows from a classical Sobolev embedding that, due to the compactness of the support of $\F a_{k,\alpha}$, one has $a_{k,\alpha}\in C^{\infty}(\Rn)$ and $a_{k,\alpha}\chi_{k,\alpha}\in C^{t}_{*}S^{0,l-1-\lceil n/p\rceil}_{1,\delta}$ for all $k\in\Z_{+}$, $\alpha\in\Z^{n}$ and $t>0$. Moreover,
\[
\|a_{k,\alpha}\chi_{k,\alpha}\|_{C^{t}_{*}S^{0,l-1-\lceil n/p\rceil}_{1,\delta}}\lesssim \|a\|_{C^{r}_{*}S^{0,l}_{1,\delta}}
\]
for an implicit constant dependent on $k$ and $t$ but independent of $a$ and $\alpha$.
\end{remark}

Next, we describe a symbol smoothing procedure from e.g.~\cite{Taylor91,Taylor23c} that decomposes a rough symbol into a sum of a smooth part and a rough part with additional decay. Let $(\psi_{j})_{j=0}^{\infty}\subseteq C^{\infty}_{c}(\Rn)$ be the Littlewood--Paley decomposition from \eqref{eq:LittlePaley}, and recall from Section \ref{subsec:transforms} that $\ph\in C^{\infty}_{c}(\Rn)$ satisfies $\ph\equiv1$ near zero. For $r>0$, $m\in\R$, $\gamma,\delta\in[0,1]$ with $\gamma\geq\delta$, $a\in C^{r}_{*}S^{m}_{1,\delta}$ and $x,\eta\in\Rn$, set
\[
a^{\sharp}_{\gamma}(x,\eta):=\sum_{k=0}^{\infty}\big(\ph(2^{-\gamma k}D)a(\cdot,\eta)\big)(x)\psi_{k}(\eta)
\]
and
\[
a^{\flat}_{\gamma}(x,\eta):=a(x,\eta)-a^{\sharp}_{\gamma}(x,\eta)=\sum_{k=0}^{\infty}\big((1-\ph)(2^{-\gamma k}D)a(\cdot,\eta)\big)(x)\psi_{k}(\eta).
\]
As is shown in \cite[Lemma 3.4]{Rozendaal22} and \cite[Lemma 4.6]{Hassell-Rozendaal23}, this decomposition has the following properties.

\begin{lemma}\label{lem:smoothing}
Let $r>0$, $m\in\R$, $\gamma,\delta\in[0,1]$ with $\gamma\geq\delta$, and $a\in C^{r}_{*}S^{m}_{1,\delta}$. Then $a^{\sharp}_{\gamma}\in S^{m}_{1,\gamma}$ and $a^{\flat}_{\gamma}\in C^{r}_{*} S^{m-(\gamma-\delta)r}_{1,\gamma}$. Moreover, if $a\in \HT^{r,\infty}S^{m}_{1,\delta}$, then $a^{\flat}_{\gamma}\in \HT^{r,\infty}S^{m-(\gamma-\delta)r}_{1,\gamma}$. If $a\in C^{2}_{-}S^{m}_{1,0}$ is elliptic and homogeneous of degree $m$ for $|\eta|\geq1$, then $a^{\sharp}_{1/2}\in \A^{2}S^{m}_{1,1/2}$ is elliptic and asymptotically homogeneous of degree $m$ with limit $a$.
\end{lemma}

We conclude with a lemma, proved in \cite[Proposition 4.9]{Hassell-Rozendaal23}, that will be used in Section \ref{sec:roughwaves} to move from a second-order equation to a first-order one.

\begin{lemma}\label{lem:squareroot}
Let $A\in \Crtwo S^{2}_{1,0}$ be non-negative, elliptic and homogeneous of degree $2$ for $|\eta|\geq1$. Then there exist a real-valued elliptic $b\in \A^{2}S^{1}_{1,1/2}$ and an $e\in S^{1}_{1,1/2}$ with the following properties:
\begin{enumerate}
\item\label{it:squareroot1} $A^{\sharp}_{1/2}(x,D)=b(x,D)^{2}+e(x,D)$.
\item\label{it:squareroot2} $b$ is asymptotically homogeneous of degree $1$ with real-valued limit $a\in \Crtwo S^{1}_{1,0}$ such that $a(x,\eta)=\sqrt{A(x,\eta)}$ for all $x,\eta\in\Rn$ with $|\eta|\geq 1$.
\end{enumerate}
\end{lemma}

\subsection{Preliminary results}\label{subsec:prelimpseudo}

In this subsection we collect a few results, mostly from previous work, that will play a role in the proof of the main result of this section.

Firstly, the following lemma will be used to deal with the smooth term that arises from the symbol smoothing procedure from the previous subsection. 

\begin{lemma}\label{lem:pseudosmooth}
Let $a\in S^{m}_{1,1/2}$, $p\in(0,\infty]$ and $s\in\R$. Then $a(x,D):\HT^{s+m,p}_{FIO}(\Rn)\to \Hps$ and $a(x,D)^{*}:\Hps\to \HT^{s-m,p}_{FIO}(\Rn)$ are bounded. Moreover, if $a$ is real-valued and elliptic, then there exists a $c>0$ such that $a(x,D)+ic:\HT^{s+m,p}_{FIO}(\Rn)\to\Hps$ is invertible.
\end{lemma}
\begin{proof}
The first statement follows from the invariance of $\Hps$ under Fourier integral operators of order zero in standard form, using the additional condition that the phase function of $a(x,D)$ is linear in the fiber variable (see \cite{LiRoSo25b}). The second statement is proved in the same manner. Indeed, although $a(x,D)^{*}$ is not directly expressed as a Fourier integral operator in standard form, its Schwartz kernel satisfies the same type of bounds as that of $a(x,D)$ (see \cite[Proposition A.1]{Hassell-Rozendaal23} or the proof of \cite[Theorem 5.1]{HaPoRo20}), allowing for an application of \cite[Proposition 2.16]{LiRoSo25b}. 

The final statement is contained in \cite[Lemma 4.10]{Hassell-Rozendaal23} for $p\geq1$. The proof for $p<1$ is identical, using standard microlocal techniques and the first statement.
\end{proof}

Next, the following proposition, about rough pseudodifferential operators on classical function spaces, extends part of the main result of \cite{Marschall88} (see also \cite{Bourdaud82,Taylor91}).

\begin{proposition}\label{prop:pseudoclass}
Let $r>0$, $m,s\in\R$, $\delta\in[0,1]$ and $p\in(0,\infty]$ be such that $(2-\delta)r>n(\frac{1}{p}-1)$. Then there exist $l\in\N$ and $C\geq0$ such that the following statements hold for each $a\in C^{r}_{*}S^{m,l}_{1,\delta}$.
\begin{enumerate} 
\item\label{it:pseudoclass1} If $n\max(0,\frac{1}{p}-1)-(1-\delta)r<s<r$, then 
\begin{equation}\label{eq:pseudoclass1}
a(x,D):\HT^{s+m,p}(\Rn)\to\HT^{s,p}(\Rn)
\end{equation}
and $\|a(x,D)\|_{\La(\HT^{s+m,p}(\Rn),\HT^{s,p}(\Rn))}\leq C\|a\|_{C^{r}_{*}S^{m,l}_{1,\delta}}$.
\item\label{it:pseudoclass2} If $n\max(0,\frac{1}{p}-1)-r<s<(1-\delta)r$, then
\begin{equation}\label{eq:pseudoclass2}
a(x,D)^{*}:\HT^{s,p}(\Rn)\to\HT^{s-m,p}(\Rn)
\end{equation}
and $\|a(x,D)^{*}\|_{\La(\HT^{s,p}(\Rn),\HT^{s-m,p}(\Rn))}\leq C\|a\|_{C^{r}_{*}S^{m,l}_{1,\delta}}$.
\item\label{it:pseudoclass3} If $\delta<1$ and $a\in \Hrinf S^{m}_{1,\delta}$, then \eqref{eq:pseudoclass1} also holds for $s=r$. 
If, additionally, $p\geq1$, then \eqref{eq:pseudoclass2} also holds for $s=-r$. 
\item\label{it:pseudoclass4} If $\delta<1$ and $a=b^{\flat}_{\delta}$ for some $b\in\Hrinf (\Rn)$, and if $p\geq1$, then \eqref{eq:pseudoclass1} holds for all $-(1-\delta)r\leq s\leq r$, with $m=-\delta r$.
\end{enumerate}
\end{proposition}

The assumption $n(\frac{1}{p}-1)< (2-\delta)r$ ensures that the conditions on $s$ in \eqref{it:pseudoclass1} and \eqref{it:pseudoclass2} can be satisfied. 
\begin{proof}
For $0<p<\infty$, apart from the statements regarding the quasi-norm bounds, \eqref{it:pseudoclass1} is \cite[Theorem 2.3]{Marschall88}, and \eqref{it:pseudoclass2} is \cite[Theorem 2.4]{Marschall88}. The quasi-norm bounds can be obtained 
from the proof. 
Then \eqref{it:pseudoclass1} and \eqref{it:pseudoclass2} follow for $p=\infty$ by duality. Moreover, the first statement of \eqref{it:pseudoclass3} is contained in \cite[Theorem 2.2]{Marschall88}, which implies the second statement for $1<p\leq \infty$. Finally, \eqref{it:pseudoclass4} is contained in \cite[Lemma 3.1]{Rozendaal23a}. Hence it remains to prove the second statement in \eqref{it:pseudoclass3} for $p=1$. 

We may suppose that $m=0$. As already noted, \eqref{it:pseudoclass2} is not void, which implies that $a(x,D)^{*}f\in \HT^{-r,1}(\Rn)$ for all $f\in\Sw(\Rn)$. We may then apply the first statement in \eqref{it:pseudoclass3} to write
\begin{equation}\label{eq:dualnotsmallp}
\begin{aligned}
\|a(x,D)^{*}f\|_{\HT^{-r,1}(\Rn)}&\eqsim\sup |\lb a(x,D)^{*}f,g\rb|=\sup |\lb f,a(x,D)g\rb|\\
&\leq \|f\|_{\HT^{-r,1}(\Rn)}\|a(x,D)g\|_{\HT^{r,\infty}(\Rn)}\lesssim \|f\|_{\HT^{-r,1}(\Rn)}, 
\end{aligned}
\end{equation} 
where the suprema are taken over all $g\in\HT^{r,\infty}(\Rn)$ with $\|g\|_{\HT^{r,\infty}(\Rn)}\leq 1$. This suffices, since the Schwartz functions are dense in $\HT^{-r,1}(\Rn)$. 
\end{proof}

\begin{remark}\label{rem:pseudopinfty}
Although $a(x,D)f$ was defined in \eqref{eq:pseudodef} only for $f\in\Sw(\Rn)$, in the proof of Proposition \ref{prop:pseudoclass} we implicitly extended the definition to $f\in\HT^{s,\infty}(\Rn)$, by adjoint action. The same will apply to the action of $a(x,D)$ on $\HT^{s,\infty}_{FIO}(\Rn)$. 

Note also that the analogue of the first equivalence in \eqref{eq:dualnotsmallp} fails for $\HT^{-r,p}(\Rn)$ if $p<1$, because then $\HT^{-r,p}(\Rn)$ is not a Banach space. More concretely, 
\[
(\HT^{-r,p}(\Rn))^{*}=C^{n(\frac{1}{p}-1)+r}_{*}(\Rn)=B^{n(\frac{1}{p}-1)+r}_{\infty,\infty}(\Rn)=(B_{1,1}^{-n(\frac{1}{p}-1)-r}(\Rn))^{*},
\]
by \cite[Theorems 2.11.2 and 2.11.3]{Triebel10}. Since $B_{1,1}^{-n(\frac{1}{p}-1)-r}(\Rn)$ is a Banach space, 
\[
\sup\{|\lb a(x,D)^{*}f,g\rb|\mid \|g\|_{(\HT^{-r,p}(\Rn))^{*}}\leq 1\}\eqsim \|a(x,D)^{*}f\|_{B_{1,1}^{-n(1/p-1)-r}(\Rn)}
\] 
for all $f\in \Sw(\Rn)$. 
\end{remark}

By combining Proposition \ref{prop:pseudoclass} with \eqref{eq:Sobolev}, one obtains the following extension of \cite[Proposition 3.3]{Rozendaal23a} and \cite[Proposition 4.5]{Rozendaal22}. 

\begin{corollary}\label{cor:pseudoclassFIO}
Let $r>0$, $m,s\in\R$, $\delta\in[0,1]$ and $p\in(0,\infty]$ be such that $(2-\delta)r>n(\frac{1}{p}-1)$. Then there exist $l\in\N$ and $C\geq0$ such that the following statements hold for each $a\in C^{r}_{*}S^{m,l}_{1,\delta}$.
\begin{enumerate} 
\item\label{it:pseudoclassFIO1} If $n\max(0,\frac{1}{p}-1)-(1-\delta)r-s(p)<s<r-s(p)$, then
\begin{equation}\label{eq:pseudoclassFIO1}
a(x,D):\HT^{s+2s(p)+m,p}_{FIO}(\Rn)\to\Hps
\end{equation}
and $\|a(x,D)\|_{\La(\HT^{s+2s(p)+m,p}_{FIO}(\Rn),\HT^{s,p}(\Rn))}\leq C\|a\|_{C^{r}_{*}S^{m,l}_{1,\delta}}$.
\item\label{it:pseudoclassFIO2} If $n\max(0,\frac{1}{p}-1)-r+s(p)<s<(1-\delta)r+s(p)$, then
\begin{equation}\label{eq:pseudoclassFIO2}
a(x,D)^{*}:\HT^{s,p}_{FIO}(\Rn)\to\HT^{s-2s(p)-m,p}_{FIO}(\Rn)
\end{equation}
and $\|a(x,D)^{*}\|_{\La(\HT^{s,p}_{FIO}(\Rn),\HT^{s-2s(p)-m,p}_{FIO}(\Rn))}\leq C\|a\|_{C^{r}_{*}S^{m,l}_{1,\delta}}$.
\item\label{it:pseudoclassFIO3} If $\delta<1$ and $a\in \Hrinf S^{m}_{1,\delta}$, then \eqref{eq:pseudoclassFIO1} also holds for $s=r-s(p)$. If, additionally, $p\geq1$, then \eqref{eq:pseudoclassFIO2} also holds for $s=n\max(0,\frac{1}{p}-1)-r+s(p)$. 
\item\label{it:pseudoclassFIO4} If $\delta<1$ and $a=b^{\flat}_{\delta}$ for some $b\in\Hrinf (\Rn)$, and if $p\geq1$, then \eqref{eq:pseudoclassFIO1} holds for all $-(1-\delta)r-s(p)\leq s\leq r-s(p)$, with $m=-\delta r$.
\end{enumerate}
\end{corollary}

\begin{remark}\label{rem:flatflat}
We will also use that the conclusion of Corollary \ref{cor:pseudoclassFIO} \eqref{it:pseudoclassFIO4} holds if $a=((b^{\flat}_{\delta'})^{\flat}_{\delta'})^{\flat}_{\delta}$, or if $a=(((b^{\flat}_{\delta''})^{\flat}_{\delta'})^{\flat}_{\delta'})^{\flat}_{\delta}$, for some $b\in\Hrinf(\Rn)$ and $0\leq \delta''\leq \delta'\leq \delta<1$. The statement in the former case was already mentioned in \cite[Remark 3.4]{Rozendaal23a} and \cite[Remark 4.6]{Rozendaal22}, and the argument for the latter case is analogous.
\end{remark}

Finally, we state a proposition concerning the main result of \cite{Rozendaal22}.

\begin{proposition}\label{prop:pseudolargep}
Let $r,c>0$, $m,s\in\R$ and $p\in(1,\infty)$. Set
\[
\sigma:=\begin{cases}
0&\text{if }r>4s(p),\\
2s(p)-\frac{r}{2}+\veps&\text{if }r\leq 4s(p),
\end{cases}
\]
for $\veps\in(0,r/2]$, and $\gamma:=\frac{1}{2}+\frac{2s(p)-\sigma}{r}$. Then there exist $l\in\N$ and $C\geq 0$ such that the following statements hold for each $a\in C^{r}_{*}S^{m,l}_{1,1/2}$ satisfying
\[
\supp(\F a(\cdot,\eta))\subseteq\{\xi\in\Rn\mid c|\eta|^{1/2}\leq|\xi|\leq \tfrac{1}{16}(1+|\eta|)^{\gamma}\}
\]
for all $\eta\in\Rn$.
\begin{enumerate}
\item\label{it:pseudolargep1} One has 
\[
a(x,D):\HT^{s+\sigma+m,p}_{FIO}(\Rn)\to \Hps
\]
and $\|a(x,D)\|_{\La(\HT^{s+\sigma+m,p}_{FIO}(\Rn),\Hps)}\leq C\|a\|_{C^{r}_{*}S^{m,l}_{1,1/2}}$.
\item\label{it:pseudolargep2}
One has 
\[
a(x,D)^{*}:\HT^{s,p}_{FIO}(\Rn)\to\HT^{s-\sigma-m,p}_{FIO}(\Rn)
\]
and $\|a(x,D)^{*}\|_{\La(\HT^{s,p}_{FIO}(\Rn),\HT^{s-\sigma-m,p}_{FIO}(\Rn))}\leq C\|a\|_{C^{r}_{*}S^{m,l}_{1,1/2}}$.
\end{enumerate}
\end{proposition}
\begin{proof}
By duality, cf.~\eqref{eq:dualHpFIO}, \eqref{it:pseudolargep2} follows from \eqref{it:pseudolargep1}. On the other hand, without the parameter $l$ and the norm bounds, \eqref{it:pseudolargep1} is the core of the proof of \cite[Theorem 5.1]{Rozendaal22}. In fact, the first step of the proof of that theorem reduces matters to the setting of this proposition. As is already stated before the theorem, one can then indeed add the parameter $l$ and the norm bounds.
\end{proof}


\subsection{A preliminary bound on $\Hp$ for $p\leq 1$}

In this subsection we obtain a first bound for rough pseudodifferential operators acting on $\Hp$ for $p\leq 1$. Although this estimate is weaker than that in our main result, our proof of the stronger bound relies crucially on the estimates that we will obtain here.

\begin{theorem}\label{thm:pseudosmallp}
Let $r,c>0$, $m,s\in\R$, $p\in(0,1]$ and $\gamma\in[1/2,1]$. Set
\begin{equation}\label{eq:tau}
\tau:=\begin{cases}
0&\text{if }r>\frac{2n}{p},\\
(\frac{2n}{p}-r)(\gamma-\frac{1}{2})+\veps&\text{if }r\leq \frac{2n}{p},
\end{cases}
\end{equation} 
for $\veps>0$. Then there exist $l\in\N$ and $C\geq0$ such that the following statements hold for each $a\in C^{r}_{*}S^{m,l}_{1,1/2}$ satisfying
\[
\supp(\F a(\cdot,\eta))\subseteq\{\xi\in\Rn\mid c|\eta|^{1/2}\leq|\xi|\leq \tfrac{1}{16}(1+|\eta|)^{\gamma}\}
\]
for all $\eta\in\Rn$.
\begin{enumerate}
\item\label{it:pseudosmallp1} One has 
\[
a(x,D):\HT^{s+\tau+m,p}_{FIO}(\Rn)\to \Hps
\]
and $\|a(x,D)\|_{\La(\HT^{s+\tau+m,p}_{FIO}(\Rn),\Hps)}\leq C\|a\|_{C^{r}_{*}S^{m,l}_{1,1/2}}$. 
\item\label{it:pseudosmallp2} One has
\[
a(x,D)^{*}:\HT^{s,p}_{FIO}(\Rn)\to\HT^{s-\tau-m,p}_{FIO}(\Rn)
\]
and $\|a(x,D)^{*}\|_{\La(\HT^{s,p}_{FIO}(\Rn),\HT^{s-\tau-m,p}_{FIO}(\Rn))}\leq C\|a\|_{C^{r}_{*}S^{m,l}_{1,1/2}}$.
\end{enumerate}
\end{theorem}
\begin{proof}
The structure of the proof is similar to that of \cite[Theorem 4.1]{Rozendaal23a}. We choose $l>|s|+n+4+4n/p$, but this bound can be improved.

\subsubsection{Reduction steps}

We may suppose that $m=0$ and that $\|a\|_{C^{r}_{*}S^{0,l}_{1,1/2}}=1$. We may also suppose that
\[
a(x,\eta)=\sum_{k=0}^{\infty}a_{k}(x)\chi_{k}(\eta)
\]
for all $x,\eta\in\Rn$, where $(a_{k})_{k=0}^{\infty}\subseteq C^{r}_{*}(\Rn)$ and $(\chi_{k})_{k=0}^{\infty}\subseteq C^{\infty}_{c}(\Rn)$ are as in Lemma \ref{lem:symdecomp}, for a fixed $\beta$. Finally, since $\Sw(\Rn)$ lies dense in $\Hps$, we may fix $f\in\Sw(\Rn)$ and obtain suitable quasi-norm bounds for $a(x,D)f$ and $a(x,D)^{*}f$. 

Let $q\in C^{\infty}_{c}(\Rn)$ be the low-frequency cutoff from before, satisfying $q(\eta)=1$ for $|\eta|\leq 2$. Then there exists an $N\in\N$, dependent only on the support of $q$, such that
\[
a(x,D)q(D)f(x)=\sum_{k=0}^{N}a_{k}(x)\chi_{k}(D)q(D)f(x)
\]
for all $x\in\Rn$. Hence \eqref{eq:Sobolev}, Remark \ref{rem:decompadditional} and Corollary \ref{cor:pseudoclassFIO} combine to show that $a(x,D)q(D):\HT^{s+\tau,p}_{FIO}(\Rn)\to \Hps$ is bounded. 
Due to the conditions on the Fourier support of the $a_{k}$, the same argument shows that $a(x,D)^{*}q(D):\HT^{s+\tau,p}_{FIO}(\Rn)\to\Hps$. Moreover, $1-q(D)$ acts boundedly on $\HT^{s+\tau,p}_{FIO}(\Rn)$, by \cite{LiRoSo25b}. Hence it suffices to derive the required bounds with $f$ replaced by $(1-q(D))f$. For simplicity of notation, we will continue working with $f$ and we merely assume in addition that $\wh{f}(\xi)=0$ for $|\xi|\leq 2$. In particular, then $\chi_{0}(D)f=0$.

The same reasoning, combined with another application of \eqref{eq:Sobolev}, shows that $q(D)a(x,D)$ and $q(D)a(x,D)^{*}$ map $\HT^{s+\tau,p}_{FIO}(\Rn)$ to $\HT^{s,p}(\Rn)$, and thus also to $L^{p}(\Rn)$. Hence, by \eqref{eq:equivchar}, it suffices to prove that
\begin{equation}\label{eq:toprovepseudosmallp1}
\Big(\int_{S^{n-1}}\|\ph_{\w}(D)a(x,D)f\|_{\HT^{s,p}(\Rn)}^{p}\ud\w\Big)^{1/p}\lesssim \|f\|_{\HT^{s+\tau,p}_{FIO}(\Rn)}
\end{equation}
and
\begin{equation}\label{eq:toprovepseudosmallp2}
\Big(\int_{S^{n-1}}\|\ph_{\w}(D)a(x,D)^{*}f\|_{\HT^{s,p}(\Rn)}^{p}\ud\w\Big)^{1/p}\lesssim \|f\|_{\HT^{s+\tau,p}_{FIO}(\Rn)}.
\end{equation}

\subsubsection{Proof of \eqref{eq:toprovepseudosmallp1}}

Let $M\geq3$ be such that $2^{M}\leq c/2$, and let $\wt{\psi}\in C^{\infty}_{c}(\Rn\setminus\{0\})$ be such that $\wt{\psi}(\xi)=1$ if $|\xi|\in[1/4,4]$. Set $\wt{\psi}_{k}(\xi):=\wt{\psi}(2^{-k+1}\xi)$, $f_{k}:=\chi_{k}(D)f$ and $a_{kj}:=\psi_{j}(D)a_{k}$ for $k\in\N$, $j\in\Z_{+}$ and $\xi\in\Rn$, and write $a_{kj}:=0$ for $j<0$. Then, by the properties of the $a_{k}$ from Lemma \ref{lem:symdecomp} and by the assumption on $f$, one has 
\begin{equation}\label{eq:touse0}
\|a_{kj}\|_{L^{\infty}(\Rn)}\leq 2^{-jr}\|a_{k}\|_{C^{r}_{*}(\Rn)}\lesssim 2^{(k/2-j)r}\|a\|_{C^{r}_{*}S^{0,1+\lceil n/p\rceil}_{1,1/2}}
\end{equation}
for all $k,j\in\Z_{+}$, and
\begin{equation}\label{eq:touse1}
a(x,D)f=\sum_{k=1}^{\infty}a_{k}f_{k}=\sum_{k=1}^{\infty}\wt{\psi}_{k}(D)\Big(\sum_{j=\lceil k/2\rceil-M}^{\lceil k\gamma\rceil-2}a_{kj}f_{k}\Big).
\end{equation}
Next, for $\w\in S^{n-1}$ and $k\in\Z_{+}$ write $\wt{\psi}_{k,\w}:=\wt{\psi}_{k}\ph_{\w}$. Then, just as in Lemma \ref{lem:packetbounds}, using also \eqref{eq:equivmetric}, for every $N\geq0$ one has
\begin{equation}\label{eq:touse2}
|\F^{-1}(\wt{\psi}_{k,\w})(x-y)|\lesssim 2^{k\frac{3n+1}{4}}(1+2^{k}d((x,\w),(y,\w))^{2})^{-N}
\end{equation}
for all $\w\in S^{n-1}$, $k\in\Z_{+}$ and $x,y\in\Rn$. Finally, let $m(D)$ be as in \eqref{eq:phiproperties3}. Then we can write
\begin{equation}\label{eq:touse3}
f_{k}=2^{k\frac{n-1}{4}}\int_{S^{n-1}}f_{k,\nu}\,\ud\nu,
\end{equation}
where
\[
f_{k,\nu}(y):=2^{-k\frac{n-1}{4}}m(D)\ph_{\nu}(D)\chi_{k}(D)f(y)
\]
for $\nu\in S^{n-1}$.

Now, by \eqref{eq:touse1}, \eqref{eq:touse3} and a straightforward calculation (see also \cite[equation (4.17)]{Rozendaal23a}), one has
\begin{equation}\label{eq:angulardecomp}
\ph_{\w}(D)a(x,D)f=\sum_{k=1}^{\infty}2^{k\frac{n-1}{4}}\wt{\psi}_{k,\w}(D)\Big(\sum_{j=\lceil k/2\rceil-M}^{\lceil k\gamma\rceil-2}a_{kj}\int_{F_{\w,k,j}}f_{k,\nu}\,\ud\nu\Big)
\end{equation}
for all $\w\in S^{n-1}$, where $F_{\w,k,j}:=\{\nu\in S^{n-1}\mid |\nu-\w|\leq 2^{3+M+j-k}\}$ for $k,j\in\Z_{+}$. Moreover, by \eqref{eq:equivmetric},
\begin{align*}
1+2^{k}d((x,\w),(y,\omega))^{2}&\simeq 1+2^{k}|x-y|^{2}+2^{k}|(x-y)\cdot\w|\\
&\simeq 1+2^{k}|x-y|^{2}+2^{k}| (x-y)\cdot\w|+2^{2k-2j}|\nu-\omega|^{2}\\
&\gtrsim 2^{k-2j}(1+2^{k}d((x,\w),(y,\nu))^{2})
\end{align*}
for all $\nu\in F_{\w,k,j}$ and $x,y\in\Rn$. For given $\lambda>0$ and $N>n/\la$, to be chosen later, we can combine this with \eqref{eq:touse2}, \eqref{eq:touse0}, Lemma \ref{lem:maxineq} and Remark \ref{rem:maxineq}: 
\begin{align*}
&\Big|2^{k\frac{n-1}{4}}\wt{\psi}_{k,\w}(D)\Big(\sum_{j=\lceil k/2\rceil-M}^{\lceil k\gamma\rceil-2}a_{kj}\int_{F_{\w,k,j}}f_{k,\nu}\,\ud\nu\Big)(x)\Big|\\
&\lesssim 2^{kn}\int_{\Rn}(1+2^{k}d((x,\w),(y,\w))^{2})^{-N}\sum_{j=\lceil k/2\rceil-M}^{\lceil k\gamma\rceil-2}|a_{kj}(y)|\int_{F_{\w,k,j}}|f_{k,\nu}(y)|\ud\nu\ud y\\
&\lesssim \sum_{j=\lceil k/2\rceil-M}^{\lceil k\gamma\rceil-2}2^{(k/2-j)r}2^{kn}\int_{\Rn}\int_{F_{\w,k,j}}(1+2^{k}d((x,\w),(y,\w))^{2})^{-N}|f_{k,\nu}(y)|\ud\nu\ud y\\
&\lesssim \sum_{j=\lceil k/2\rceil-M}^{\lceil k\gamma\rceil-2}2^{(k/2-j)(r-2N)}2^{kn}\int_{\Rn}\int_{F_{\w,k,j}}(1+2^{k}d((x,\w),(y,\nu))^{2})^{-N}|f_{k,\nu}(y)|\ud\nu\ud y\\
&\lesssim \sum_{j=\lceil k/2\rceil-M}^{\lceil k\gamma\rceil-2}2^{(k/2-j)(r-2N)}\Ma_{\la}(g_{k})(x,\w),
\end{align*}
where $g_{k}(y,\nu):=f_{k,\nu}(y)$ for $(y,\nu)\in\Sp$. 

If $r>2n/p$, then we can choose $\lambda\in(0,p)$ and $N>n/\la$ such that $r>2N$. On the other hand, if $r\leq 2n/p$, then for $\delta>0$ one can set $N=n/p+\delta$ and also find a $\lambda\in(0,p)$ such that $N>n/\la$. In either case, by definition of $\tau$ in \eqref{eq:tau} and for $\delta$ small enough, we have
\[
\sum_{j=\lceil k/2\rceil-M}^{\lceil k\gamma\rceil-2}2^{(k/2-j)(r-2N)}\lesssim 2^{k\tau}.
\]
Hence the considerations above imply that, for all $(x,\w)\in \Sp$, one has
\begin{align*}
&\sum_{k=1}^{\infty}2^{2ks}\Big|2^{k\frac{n-1}{4}}\wt{\psi}_{k,\w}(D)\Big(\sum_{j=\lceil k/2\rceil-M}^{\lceil k\gamma\rceil-2}a_{kj}\int_{F_{\w,k,j}}f_{k,\nu}\,\ud\nu\Big)(x)\Big|^{2}\\
&\lesssim \sum_{k=1}^{\infty}2^{2ks}\Big|\sum_{j=\lceil k/2\rceil-M}^{\lceil k\gamma\rceil-2}2^{(k/2-j)(r-2N)}\Ma_{\la}(g_{k})(x,\w)\Big|^{2}\\
&\lesssim \sum_{k=1}^{\infty}2^{2k(s+\tau)}(\Ma_{\la}(g_{k})(x,\w))^{2}.
\end{align*}
Taking into account \eqref{eq:angulardecomp}, we can now use a standard square function estimate from Littlewood--Paley theory (see e.g.~\cite[Section 2.5.2]{Triebel10}), and the boundedness of the Hardy-Littlewood maximal operator $\Ma$ on $L^{p/\la}(\Sp;\ell^{2/\la})$, to write
\begin{align*}
&\int_{S^{n-1}}\|\ph_{\w}(D)a(x,D)f\|_{\HT^{s,p}(\Rn)}^{p}\ud\w\\
&\eqsim \int_{\Sp}\Big(\sum_{k=1}^{\infty}2^{2ks}\Big|2^{k\frac{n-1}{4}}\wt{\psi}_{k,\w}(D)\Big(\sum_{j=\lceil k/2\rceil-M}^{\lceil k\gamma\rceil-2}a_{kj}\int_{F_{\w,k,j}}f_{k,\nu}\,\ud\nu\Big)(x)\Big|^{2}\Big)^{p/2}\ud x\ud\w\\
&\lesssim \int_{\Sp}\Big(\sum_{k=1}^{\infty}2^{2k(s+\tau)}(\Ma(|g_{k}|^{\la})(x,\w))^{2/\la}\Big)^{p/2}\ud x\ud\w\\
&\lesssim \int_{\Sp}\Big(\sum_{k=1}^{\infty}2^{2k(s+\tau)}|g_{k}(y,\nu)|^{2}\Big)^{p/2}\ud y\ud\nu\\
&\lesssim \int_{\Sp}\Big(\sum_{k=1}^{\infty}2^{2k(s+\tau-\frac{n-1}{4})}|\chi_{k}(D)m(D)\ph_{\nu}(D)f(y)|^{2}\Big)^{p/2}\ud y\ud\nu\\
&\eqsim \int_{S^{n-1}}\Big\|\sum_{k=1}^{\infty}2^{-k\frac{n-1}{4}}m(D)\chi_{k}(D)\ph_{\nu}(D)f\Big\|_{\HT^{s+\tau,p}(\Rn)}^{p}\ud\nu,
\end{align*}
where the resulting implicit constant does not depend on $a$.

Finally, by the properties of the $\chi_{k}$ from Lemma \ref{lem:symdecomp},  one has
\[
\sup_{\eta\in\Rn}\Big|\lb\xi\rb^{|\alpha|}\partial_{\eta}^{\alpha}\Big(\sum_{k=1}^{\infty}2^{-k\frac{n-1}{4}}m(\eta)\chi_{k}(\eta)\Big)\Big|<\infty
\]
for each $\alpha\in\Z_{+}^{n}$ with $|\alpha|\leq l-1-\lceil n/p\rceil$. Since $l-1-\lceil n/p\rceil>|s|+n+2+3n/p$, the Fourier multiplier theorem in \cite[Section 2.3.7]{Triebel10} yields
\[
\Big\|\sum_{k=1}^{\infty}2^{-k\frac{n-1}{4}}m(D)\chi_{k}(D)\ph_{\nu}(D)f\Big\|_{\HT^{s+\tau,p}(\Rn)}\lesssim \|\ph_{\nu}(D)f\|_{\HT^{s+\tau,p}(\Rn)}
\]
for all $\nu\in S^{n-1}$. By combining everything we have shown with \eqref{eq:equivchar}, we thus find that
\begin{align*}
\Big(\int_{S^{n-1}}\|\ph_{\w}(D)a(x,D)f\|_{\HT^{s,p}(\Rn)}^{p}\ud\w\Big)^{1/p}&\lesssim \Big(\int_{S^{n-1}}\|\ph_{\nu}(D)f\|_{\HT^{s+\tau,p}(\Rn)}^{p}\ud\nu\Big)^{1/p}\\
&\lesssim \|f\|_{\HT^{s+\tau,p}_{FIO}(\Rn)}
\end{align*}
for implicit constants independent of $a$ and $f$, as is required for \eqref{eq:toprovepseudosmallp1}.

\subsubsection{Proof of \eqref{eq:toprovepseudosmallp2}} Due to the condition on the Fourier support of the $a_{k}$, one has
\[
a(x,D)^{*}f=\sum_{k=1}^{\infty}\chi_{k}(D)(a_{k}\wt{f}_{k})=\sum_{k=1}^{\infty}\overline{\chi_{k}}(D)\Big(\sum_{j=\lceil k/2\rceil-M}^{\lceil k\gamma\rceil-2}\overline{a_{kj}}\wt{f}_{k}\Big),
\]
where $\wt{f}_{k}:=\wt{\psi}_{k}(D)f$ for $k\in\N$. This expression is similar to that in  \eqref{eq:touse1}, with the 
main change 
being that the functions involved have slightly modified supports. However, this change has no meaningful impact on the argument from before, and the proof is thus completely analogous to that of \eqref{eq:toprovepseudosmallp1}.
\end{proof}

\subsection{Main result for pseudodifferential operators}\label{subsec:mainpseudo}

In this subsection we prove our main result for rough pseudodifferential operators. To this end, we first collect two lemmas that will be used for the interpolation procedure in the proof.

\begin{lemma}\label{lem:intparameters}
Let $p\in(0,1]$. For $\delta\in(0,p)$, let $\theta\in(0,1)$ be such that $\frac{1}{p}=\frac{1-\theta}{\delta}+\frac{\theta}{1+\delta}$. 
Then the following assertions hold as $\delta\to0$:
\begin{enumerate}
\item\label{it:intparameters1} $\theta\to 1$;
\item\label{it:intparameters2} $\frac{1-\theta}{\delta}\to \frac{1}{p}-1$;
\item\label{it:intparameters3} $s(1+\delta)\to s(1)$;
\item\label{it:intparameters4} $(1-\theta)(\tfrac{2n}{\delta}+\delta)+\theta (n-1)\to 4s(p)+2(\frac{1}{p}-1)$.
\end{enumerate} 
\end{lemma}
\begin{proof}
The first statement follows by noting that
\[
\theta=\frac{\frac{1}{p}-\frac{1}{\delta}}{\frac{1}{1+\delta}-\frac{1}{\delta}}=\frac{\frac{\delta}{p}-1}{\frac{\delta}{1+\delta}-1}\to 1.
\]
The second statement can be obtained in the same way:
\begin{align*}
\frac{1-\theta}{\delta}=\frac{1}{\delta}\frac{\frac{\delta}{1+\delta}-1-(\frac{\delta}{p}-1)}{\frac{\delta}{1+\delta}-1}=\frac{\frac{1}{1+\delta}-\frac{1}{p}}{\frac{\delta}{1+\delta}-1}\to \frac{1}{p}-1.
\end{align*}
The third statement in turn is immediate. And finally, \eqref{it:intparameters4} follows from \eqref{it:intparameters1} and \eqref{it:intparameters2}, upon noting that $2n(\tfrac{1}{p}-1)+n-1=4s(p)+2(\frac{1}{p}-1)$.
\end{proof}


\begin{lemma}\label{lem:contanalytic}
Let $r>0$ and $\kappa,\lambda\in\R$. 
Let $a\in C^{r}_{*}S^{0}_{1,1/2}$ be such that $\supp(\F a(\cdot,\eta))\subseteq\{\xi\in\Rn\mid c|\eta|^{1/2}\leq |\xi|\leq \frac{1}{16}(1+|\eta|)\}$ for all $\eta\in\Rn$, and suppose that there exist $b\in C^{r}_{*}(\Rn)$ and $\chi\in C^{\infty}_{c}(\Rn)$ such that  $a(x,\eta)=b(x)\chi(\eta)$ for all $x,\eta\in\Rn$.  For $z\in\C$ and $x,\eta\in\Rn$, set
\[
a_{z}(x,\eta):=e^{(\kappa z+\lambda)^{2}}\lb \eta\rb^{-(\kappa z+\lambda)/2}\big(\lb D\rb^{\kappa z+\lambda}a(\cdot,\eta)\big)(x).
\]
Then, for all $p\in(0,\infty)$ and $s\in\R$, and for each compact interval $I\subseteq\R$, one has
\begin{equation}\label{eq:contanalytic}
\sup_{\Real(z)\in I}\|a_{z}(x,D)\|_{\La(\Hps)}+\|a_{z}(x,D)^{*}\|_{\La(\Hps)}<\infty.
\end{equation}
Moreover, for each $f\in\Hps$, the $\Hps$-valued maps $z\mapsto a_{z}(x,D)f$ and $z\mapsto a_{z}(x,D)^{*}f$ are analytic on $\C$.
\end{lemma}
Analyticity of a function with values in a quasi-Banach space is defined in terms of absolutely convergent power series (see \cite[Appendix A]{LiRoSo25b}).
\begin{proof}
We may suppose that $\chi\neq 0$. Then, by assumption, $b$ has compact Fourier support. Hence $b\in C^{t}_{*}(\Rn)$ for all $t>0$, and since $\chi$ has compact support we in fact have $a\in C^{t}_{*}S^{m}_{1,1/2}$ for all $m\in\R$. Lemma \ref{lem:symbolinter} and Corollary \ref{cor:pseudoclassFIO} then yield \eqref{eq:contanalytic}. Due to the density of $\Sw(\Rn)$ in $\Hps$, one can in turn apply \cite[Proposition 3.2]{Kalton-Mitrea98} to see that we may prove analyticity for $f\in\Sw(\Rn)$. 

The support assumptions on $\wh{b}$ and $\chi$ also imply that $a_{z}(x,D)f$ and $a_{z}(x,D)^{*}f$ have compact Fourier support, independent of $z\in\C$. Hence \eqref{eq:Sobolev} and \cite[Theorem 1.4.1]{Triebel10} yield
\begin{align*}
\|a_{z}(x,D)f\|_{\Hps}&\eqsim \|a_{z}(x,D)f\|_{L^{p}(\Rn)},\\
\|a_{z}(x,D)^{*}f\|_{\Hps}&\eqsim \|a_{z}(x,D)^{*}f\|_{L^{p}(\Rn)}.
\end{align*}
It therefore suffices to prove analyticity in the $L^{p}(\Rn)$ quasi-norm. 
We consider $a_{z}(x,D)f$; an analogous argument works for $a_{z}(x,D)^{*}f$.

Let $N\in\N$ be such that $2Np>n$. Then, for each $\eta\in\Rn$, the power series for $e^{(\kappa z+\lambda)^{2}}(1-\Delta)^{N}(\lb \eta\rb^{-(\kappa z+\lambda)/2}\chi(\eta))$ converges, locally uniformly in $z\in\C$. Since $\wh{f}\in\Sw(\Rn)$, the dominated convergence theorem implies that the power series for
\begin{align*}
&\lb x\rb^{2N}e^{(\kappa z+\lambda)^{2}}\lb D\rb^{-(\ka z+\la)/2}\chi(D)f(x)\\
&=\frac{1}{(2\pi)^{n}}\int_{\Rn}e^{ix\cdot\eta}e^{(\kappa z+\lambda)^{2}}(1-\Delta)^{N}\big(\lb \eta\rb^{-(\ka z+\la)/2}\chi(\eta)\wh{f}(\eta)\big)\ud\eta
\end{align*}
also converges, for each $x\in\Rn$. In fact, the series converges in $L^{\infty}(\Rn)$. Moreover, since $b$ has compact Fourier support, there exists a $\psi\in C^{\infty}_{c}(\Rn)$ such that 
\begin{align*}
\lb D\rb^{\ka z+\la}b(x)&=\lb D\rb^{\ka z+\la}\psi(D)b(x)=\int_{\Rn}\F^{-1}(\psi_{z})(x-y)b(y)\ud y\\
&=\frac{1}{(2\pi)^{n}}\int_{\Rn}\int_{\Rn}e^{-iy\xi}(1-\Delta)^{n}\big(e^{ix\xi}\psi_{z}(\xi)\big)\ud\xi \lb y\rb^{-2n}b(y)\ud y
\end{align*}
for each $x\in\Rn$, where $\psi_{z}(\xi):=\lb \xi\rb^{\ka z+\la}\psi(\xi)$ for $\xi\in\Rn$ and $z\in\C$. Since $\psi\in C^{\infty}_{c}(\Rn)$ and $b$ is bounded, one can apply the dominated convergence theorem, twice, to see that the power series for $\lb D\rb^{\ka z+\la}b(x)$ converges, again in $L^{\infty}(\Rn)$. By the choice of $N$, another application of the dominated convergence theorem shows that the power series for $a_{z}(x,D)f$ converges in $L^{p}(\Rn)$, as required.
\end{proof}

We are now ready to prove the main result of this section, an extension of \cite[Theorem 5.1]{Rozendaal22} from $1<p<\infty$ to $0<p\leq\infty$. Write
\[
r(p):=4s(p)+2\max(0,\tfrac{1}{p}-1),
\]
so that $r(p)=4s(p)$ if $p\geq1$, and $r(p)=4s(p)+2(\tfrac{1}{p}-1)$ if $p<1$.

\begin{theorem}\label{thm:mainpseudo}
Let $r>0$, $m,s\in\R$ and $p\in(0,\infty]$. 
Set
\begin{equation}\label{eq:sigma}
\sigma:=\begin{cases}
0&\text{if }r>r(p),\\
\frac{r(p)-r}{2}+\veps&\text{if }2n\max(0,\frac{1}{p}-1)<r\leq r(p),
\end{cases}
\end{equation}
for $\veps\in(0,2s(p)-\frac{r(p)-r}{2}]$. 
 Then the following statements hold for each $a\in C^{r}_{*}S^{m}_{1,1/2}$.
\begin{enumerate}
\item\label{it:mainpseudo1}
If $n\max(0,\tfrac{1}{p}-1)-\tfrac{r}{2}+s(p)-\sigma<s<r-s(p)$, then
\begin{equation}\label{eq:mainpseudo1}
a(x,D):\HT^{s+\sigma+m,p}_{FIO}(\Rn)\to \Hps.
\end{equation}
\item\label{it:mainpseudo2}
If $n\max(0,\tfrac{1}{p}-1)-r+s(p)<s<\frac{r}{2}-s(p)+\sigma$, then
\begin{equation}\label{eq:mainpseudo2}
a(x,D)^{*}:\HT^{s,p}_{FIO}(\Rn)\to \HT^{s-\sigma-m,p}_{FIO}(\Rn).
\end{equation}
\item\label{it:mainpseudo3}If $a\in\Hrinf S^{m}_{1,1/2}$, then \eqref{eq:mainpseudo1} also holds for $s=r-s(p)$. If, additionally, $p\geq1$, then \eqref{eq:mainpseudo2} also holds for $s=n\max(0,\tfrac{1}{p}-1)-r+s(p)$.
\item\label{it:mainpseudo4} If $a=b^{\flat}_{1/2}$ for some $b\in \Hrinf (\Rn)$, and if $p\geq 1$, then \eqref{eq:mainpseudo1} holds for all  $-\tfrac{r}{2}+s(p)-\sigma\leq s\leq r-s(p)$, with $m=-r/2$.
\end{enumerate}
\end{theorem}

{In each of the cases in \eqref{eq:sigma}, the intervals for $s$ in \eqref{it:mainpseudo1} and \eqref{it:mainpseudo2} are not empty. On the other hand, the statements are void for $r\leq 2n\max(0,\frac{1}{p}-1)$.}

\begin{proof}
The strategy of the proof is similar to that of \cite[Theorem 5.1]{Rozendaal22}.  

\subsubsection{Reduction steps}


{We may consider $r>2n(\frac{1}{p}-1)$. After replacing $a(x,D)$ by $a(x,D)\lb D\rb^{-m}$, we may 
also suppose that $m=0$.} 
Let 
\begin{equation}\label{eq:gamma}
\gamma:=\frac{1}{2}+\frac{2s(p)-\sigma}{r}.
\end{equation}
and note that $\gamma\in[1/2,1)$. We claim that it suffices to prove the following statement. If $a\in C^{r}_{*}S^{0}_{1,1/2}$ is such that, for some $c>0$ and all $\eta\in\Rn$, one has
\begin{equation}\label{eq:supportconditiona}
\supp(\F a(\cdot,\eta))\subseteq\{\xi\in\Rn\mid c|\eta|^{1/2}\leq |\xi|\leq  \tfrac{1}{16}(1+|\eta|)^{\gamma}\},
\end{equation} 
then \eqref{eq:mainpseudo1} and \eqref{eq:mainpseudo2} hold for all $s\in\R$.

To prove this claim, apply the symbol smoothing procedure from Lemma \ref{lem:smoothing} to a general $a\in C^{r}_{*}S^{0}_{1,1/2}$, twice, to write 
\begin{equation}\label{eq:adecomp}
a=a_{1/2}^{\sharp}+a^{\flat}_{1/2}=a_{1/2}^{\sharp}+(a^{\flat}_{1/2})^{\sharp}_{\gamma}+(a^{\flat}_{1/2})^{\flat}_{\gamma}.
\end{equation}
By Lemmas \ref{lem:smoothing} and \ref{lem:pseudosmooth}, one has
\[
a^{\sharp}_{1/2}(x,D):\HT^{s+\sigma,p}_{FIO}(\Rn)\subseteq\Hps\to\Hps
\]
and 
\[
a^{\sharp}_{1/2}(x,D)^{*}:\Hps\to\Hps\subseteq\HT^{s-\sigma,p}_{FIO}(\Rn).
\]
Moreover, $(a^{\flat}_{1/2})^{\flat}_{\gamma}\in C^{r}_{*}S^{-(\gamma-1/2)r}_{1,\gamma}$, by Lemma \ref{lem:smoothing}. Since $2s(p)-(\gamma-1/2)r=\sigma$, Corollary \ref{cor:pseudoclassFIO} \eqref{it:pseudoclassFIO1} thus implies that 
\begin{equation}\label{eq:flathigh1}
(a^{\flat}_{1/2})^{\flat}_{\gamma}(x,D):\HT^{s+\sigma,p}_{FIO}(\Rn)\to\Hps
\end{equation}
for $n\max(0,\frac{1}{p}-1)-\frac{r}{2}+s(p)-\sigma<s<r-s(p)$, and Corollary \ref{cor:pseudoclassFIO} \eqref{it:pseudoclassFIO2} that
\begin{equation}\label{eq:flathigh2}
(a^{\flat}_{1/2})^{\flat}_{\gamma}(x,D)^{*}:\Hps\to\HT^{s-\sigma,p}_{FIO}(\Rn)
\end{equation}
for $n\max(0,\frac{1}{p}-1)-r+s(p)<s<\frac{r}{2}-s(p)+\sigma$. By Lemma \ref{lem:smoothing} and Corollary \ref{cor:pseudoclassFIO} \eqref{it:pseudoclassFIO3}, if $a\in \HT^{r,\infty}S^{0}_{1,1/2}$, then \eqref{eq:flathigh1} also holds for $s=r-s(p)$. If, additionally, $p\geq1$, then \eqref{eq:flathigh2} also holds for $s=n\max(0,\frac{1}{p}-1)-r+s(p)$. Finally, if $a(x,D)=b_{1/2}^{\flat}(x,D)\lb D\rb^{r/2}$ for some $b\in \Hrinf (\Rn)$, and if $p\geq1$, then Remark \ref{rem:flatflat} shows that \eqref{eq:flathigh1} also holds for $s=-\frac{r}{2}+s(p)-\sigma$. 

Next, note that $a^{\flat}_{1/2}\in C^{r}_{*}S^{0}_{1,1/2}$ and $(a^{\flat}_{1/2})^{\flat}_{\gamma}\in C^{r}_{*}S^{-(\gamma-1/2)r}_{1,\gamma}\subseteq C^{r}_{*}S^{0}_{1,1/2}$, by Lemma \ref{lem:smoothing}. Then $(a^{\flat}_{1/2})^{\sharp}_{\gamma}\in C^{r}_{*}S^{0}_{1,1/2}$ as well, by \eqref{eq:adecomp}. Moreover, for $\ph$ with sufficiently small support (independent of $a$), \eqref{eq:supportconditiona} holds with $a$ replaced by $(a^{\flat}_{1/2})^{\sharp}_{\gamma}$. 
The claim thus follows by replacing $a$ by $(a^{\flat}_{1/2})^{\sharp}_{\gamma}$.

The rest of the proof is dedicated to showing that \eqref{eq:mainpseudo1} and \eqref{eq:mainpseudo2} hold for all $s\in\R$ and $a\in C^{r}_{*}S^{0}_{1,1/2}$ satisfying \eqref{eq:supportconditiona}. For $p\in(1,\infty)$, this immediately follows from Proposition \ref{prop:pseudolargep}. Moreover, for $p=\infty$, one may rely on \eqref{eq:dualHpFIO} after dealing with the case where $p=1$. Hence, in the remainder, we will consider $p\leq 1$.


\subsubsection{Symbol decomposition}

The key tool in the rest of the proof will be interpolation, between Proposition \ref{prop:pseudolargep} and Theorem \ref{thm:pseudosmallp}. To this end, it will be convenient to prove an a priori slightly stronger statement: for each $s\in\R$ there exists an $l\in\Z_{+}$ such that \eqref{eq:mainpseudo1} and \eqref{eq:mainpseudo2} hold for all $a\in C^{r}_{*}S^{0}_{1,1/2}$ satisfying \eqref{eq:supportconditiona}, with operator quasi-norms bounded by a constant multiple of $\|a\|_{C^{r}_{*}S^{0,l}_{1,1/2}}$. In fact, by Lemma \ref{lem:symdecomp} it suffices to prove this stronger statement in the case where $a(x,\eta)=\sum_{k=0}^{\infty}a_{k}(x)\chi_{k}(\eta)$ for all $x,\eta\in\Rn$, with $(a_{k})_{k=0}^{\infty}\subseteq C^{r}_{*}(\Rn)$ and $(\chi_{k})_{k=0}^{\infty}\subseteq C^{\infty}_{c}(\Rn)$ as in Lemma \ref{lem:symdecomp}. Moreover, we may show that 
\begin{equation}\label{eq:mainpseudo11}
\|a(x,D)f\|_{\Hps}\lesssim \|a\|_{C^{r}_{*}S^{0,l}_{1,1/2}}\|f\|_{\HT^{s+\sigma,p}_{FIO}(\Rn)}
\end{equation}
and
\begin{equation}\label{eq:mainpseudo21}
\|a(x,D)^{*}f\|_{\HT^{s,p}_{FIO}(\Rn)}\lesssim \|a\|_{C^{r}_{*}S^{0,l}_{1,1/2}}\|f\|_{\HT^{s+\sigma,p}_{FIO}(\Rn)},
\end{equation}
for all $f\in\Sw(\Rn)$ with compact Fourier support, given that the latter class lies dense in $\HT^{s+m+\sigma,p}_{FIO}(\Rn)$ and $\Hps$. Note that we have replaced $s$ by $s+\sigma$ to arrive at \eqref{eq:mainpseudo21}, which we may do without loss of generality and which will simplify notation somewhat later on.

Now, using the conditions on the Fourier support of $f$ and the $a_{k}$, as well as the fact that $\gamma<1$, to prove \eqref{eq:mainpseudo11} and \eqref{eq:mainpseudo21} we may then additionally suppose that there exists a $K\in\Z_{+}$ such that $a_{k}=0$ for all $k>K$. Of course, we will obtain implicit constants in \eqref{eq:mainpseudo11} and \eqref{eq:mainpseudo21} that do not depend on $K$, but working with a finite sum $a=\sum_{k=0}^{K}a_{k}\chi_{k}$ allows us to apply Lemma \ref{lem:contanalytic} to deal with subtleties regarding analyticity.

\subsubsection{Interpolation}

Let $\delta\in(0,\min(p,2n/r))$ and $\theta\in(0,1)$ be such that $\frac{1}{p}=\frac{1-\theta}{\delta}+\frac{\theta}{1+\delta}$. 
 Set $r_{0}:=\frac{2n}{\delta}+\delta$, and let $r_{1}\in\R$ be such that $r=(1-\theta)r_{0}+\theta r_{1}$. We will let $\delta\to0$, and by Lemma \ref{lem:intparameters} we may thus suppose that
\[
r_{1}=\theta^{-1}\left[r-(1-\theta)\big(\tfrac{2n}{\delta}+\delta\big)\right]\in(0,r),
\]
where we also used that $r>2n(\tfrac{1}{p}-1)$.

  As in Lemma \ref{lem:symbolinter}, set 
\begin{equation}\label{eq:formaz}
\begin{aligned}
a_{z}(x,\eta):=&\,e^{(\kappa z+\lambda)^{2}}\lb \eta\rb^{-(\kappa z+\lambda)/2}\big(\lb D\rb^{\kappa z+\lambda}a(\cdot,\eta)\big)(x)\\
=&\,\sum_{k=0}^{K}e^{(\kappa z+\lambda)^{2}}\big(\lb D\rb^{\kappa z+\lambda}a_{k}\big)(x)\lb \eta\rb^{- (\kappa z+\lambda)/2}\chi_{k}(\eta)
\end{aligned}
\end{equation}
for $z\in \overline{\St}$ and $x,\eta\in\Rn$, where $\ka:=r_{0}-r_{1}$ and $\la:=r-r_{0}$. Then $a_{\theta}=a$, and for every $l\in\Z_{+}$ there exists a $C_{l}\geq0$ such that $a_{z}\in C^{r-\Real(\kappa z+\lambda)}_{*}S^{0,l}_{1,1/2}$ for each $z\in\overline{\St}$, with
\begin{equation}\label{eq:aznorm}
\sup\{\|a_{z}\|_{C^{r-\Real(\kappa z+\lambda)}_{*}S^{0,l}_{1,1/2}}\mid z\in\overline{\St}\}\leq C_{l}\|a\|_{C^{r}_{*}S^{0,l}_{1,1/2}} <\infty.
\end{equation}
Also, it follows from \eqref{eq:supportconditiona} that
\begin{equation}\label{eq:supportconditionaz}
\supp(\F a_{z}(\cdot,\eta))\subseteq\{\xi\in\Rn\mid c|\eta|^{1/2}\leq |\xi|\leq  \tfrac{1}{16}(1+|\eta|)^{\gamma}\}
\end{equation} 
for all $z\in\overline{\St}$ and $\eta\in\Rn$. 

Now set $X_{0}:=\HT^{s,\delta}_{FIO}(\Rn)$, $Y_{0}:=\HT^{s,\delta}_{FIO}(\Rn)$, $Y_{1}:=\HT^{s,1+\delta}_{FIO}(\Rn)$ and $X_{1}:=\HT^{s+\rho,1+\delta}_{FIO}(\Rn)$, where 
\begin{equation}\label{eq:rho}
\rho:=\begin{cases}
0&\text{if }r_{1}>4s(1+\delta),\\
2s(1+\delta)-\frac{r_{1}}{2}+\veps'&\text{if }r_{1}\leq 4s(1+\delta),
\end{cases}
\end{equation}
for some $\veps'\in(0,r_{1}/2]$, to be chosen later. Then \eqref{eq:aznorm}, \eqref{eq:supportconditionaz} and Theorem \ref{thm:pseudosmallp} yield an $l\in\Z_{+}$ such that $a_{it}(x,D):X_{0}\to Y_{0}$ and $a_{it}(x,D)^{*}:X_{0}\to Y_{0}$ for all $t\in\R$, with
\begin{equation}\label{eq:stripbounds1}
\sup_{t\in\R}\|a_{it}(x,D)\|_{\La(X_{0},Y_{0})}+\sup_{t\in\R}\|a_{it}(x,D)^{*}\|_{\La(X_{0},Y_{0})}\lesssim \|a\|_{C^{r}_{*}S^{0,l}_{1,1/2}}.
\end{equation}
Similarly, after possibly enlarging $l$, it follows from \eqref{eq:aznorm}, \eqref{eq:supportconditionaz} and Proposition \ref{prop:pseudolargep} that $a_{1+it}(x,D):X_{1}\to Y_{1}$ and $a_{1+it}(x,D)^{*}:X_{1}\to Y_{1}$ for all $t\in\R$, with
\begin{equation}\label{eq:stripbounds2}
\sup_{t\in\R}\|a_{1+it}(x,D)\|_{\La(X_{1},Y_{1})}+\sup_{t\in\R}\|a_{1+it}(x,D)^{*}\|_{\La(X_{1},Y_{1})}\lesssim \|a\|_{C^{r}_{*}S^{0,l}_{1,1/2}}.
\end{equation}
Here we used that $r_{1}>0$. Note that $l$ and the implicit constants do not depend on $a$ and $K$. 

Finally, by \eqref{eq:formaz} and Lemma \ref{lem:contanalytic},
\[
\sup_{z\in\overline{\St}}\|a_{z}(x,D)\|_{\La(X_{0}+X_{1},Y_{0}+Y_{1})}+\sup_{z\in\overline{\St}}\|a_{z}(x,D)^{*}\|_{\La(X_{0}+X_{1},Y_{0}+Y_{1})}<\infty.
\]
Of course, a priori the supremum depends on $K$ here, but this will not be a problem. Note also that, by Lemma \ref{lem:contanalytic}, the $Y_{0}+Y_{1}$-valued maps $z\mapsto a_{z}(x,D)f$ and $z\mapsto a_{z}(x,D)^{*}f$ are analytic on $\St$ for each $f\in X_{0}+X_{1}$, and continuous and bounded on $\overline{\St}$. Moreover, for $j\in\{0,1\}$ and $f\in X_{j}$, the $Y_{j}$-valued maps $t\mapsto a_{j+it}(x,D)f$ and $t\mapsto a_{j+it}(x,D)^{*}f$ are continuous and bounded. 


Now we can combine a version in quasi-Banach spaces of the standard result on interpolation of analytic families of operators (see \cite[Remark 3.4 and Lemma A.1]{LiRoSo25b}) to \eqref{eq:stripbounds1} and \eqref{eq:stripbounds2}, to obtain
\begin{equation}\label{eq:conclusion1}
\|a(x,D)\|_{\La([X_{0},X_{1}]_{\theta},[Y_{0},Y_{1}]_{\theta})}=\|a_{\theta}(x,D)\|_{\La([X_{0},X_{1}]_{\theta},[Y_{0},Y_{1}]_{\theta})}\lesssim \|a\|_{C^{r}_{*}S^{0,l}_{1,1/2}}
\end{equation}
and 
\begin{equation}\label{eq:conclusion2}
\|a(x,D)^{*}\|_{\La([X_{0},X_{1}]_{\theta},[Y_{0},Y_{1}]_{\theta})}=\|a_{\theta}(x,D)^{*}\|_{\La([X_{0},X_{1}]_{\theta},[Y_{0},Y_{1}]_{\theta})}\lesssim \|a\|_{C^{r}_{*}S^{0,l}_{1,1/2}}.
\end{equation}

\subsubsection{Conclusion}

It remains to unwrap what we have proved. By \eqref{eq:intHpFIO} and the choice of $\theta$, $[X_{0},X_{1}]_{\theta}=\HT^{s+\theta \rho,p}_{FIO}(\Rn)$ and $[Y_{0},Y_{1}]_{\theta}=\Hps$. 

If $r>r(p)=4s(p)+2(\frac{1}{p}-1)$, then Lemma \ref{lem:intparameters} \eqref{it:intparameters4} implies that $r_{1}\geq n-1>4s(1+\delta)$ for $\delta$ sufficiently small. Hence \eqref{eq:rho}, \eqref{eq:conclusion1}, \eqref{eq:conclusion2} yield the required statement in the first case of \eqref{eq:sigma}.

On the other hand, if $r\leq r(p)$, then we can use \eqref{eq:rho} and parts \eqref{it:intparameters1}, \eqref{it:intparameters2} and \eqref{it:intparameters3} of Lemma \ref{lem:intparameters} to see that
\begin{align*}
\theta\rho&\leq \theta(2s(1+\delta)-\tfrac{r_{1}}{2}+\veps')=2\theta s(1+\delta)+\tfrac{(1-\theta)r_{0}-r}{2}+\theta\veps'\\
&=2\theta s(1+\delta)+\tfrac{n(1-\theta)}{\delta}+\tfrac{(1-\theta)\delta}{2}-\tfrac{r}{2}+\theta\veps'\to 2s(1)+n(\tfrac{1}{p}-1)-\tfrac{r}{2}+\veps'\\
&=\tfrac{n-1}{2}+(n-1)(\tfrac{1}{p}-1)+\tfrac{1}{p}-1-\tfrac{r}{2}+\veps'=\tfrac{r(p)-r}{2}+\veps'.
\end{align*}
Hence, by choosing $\delta,\veps'>0$ sufficiently small, \eqref{eq:conclusion1} and \eqref{eq:conclusion2} yield the required conclusion in the second case of \eqref{eq:sigma}.
\end{proof}

\begin{remark}\label{rem:constantpseudo}
One can weaken the assumptions for \eqref{it:mainpseudo1} and \eqref{it:mainpseudo2} to $a\in C^{r}_{*}S^{m,l}_{1,1/2}$ for some $l\in\Z_{+}$ large enough, and obtain quasi-norm bounds for $a(x,D)$ and $a(x,D)^{*}$ in terms of $\|a\|_{C^{r}_{*}S^{m,l}_{1,1/2}}$. An analogous statement applies to \eqref{it:mainpseudo3} and \eqref{it:mainpseudo4}.
\end{remark}

\begin{remark}\label{rem:flatflat2}
The conclusion of Theorem \ref{thm:mainpseudo} \eqref{it:mainpseudo4} also holds if $a=(b^{\flat}_{\delta})_{1/2}^{\flat}$ for some $b\in\Hrinf(\Rn)$ and $\delta\in[0,1/2]$. This follows from the same proof, relying on the second statement in Remark \ref{rem:flatflat} when dealing with \eqref{eq:flathigh1}.
\end{remark}

\begin{remark}\label{rem:suboptimalgamma}
One can also derive mapping properties as in \eqref{eq:mainpseudo1} and \eqref{eq:mainpseudo2} for certain $r\leq 2n(\frac{1}{p}-1)$, if $0<p<1$. For example, for all $r>\frac{2}{3}n(\frac{1}{p}-1)$ one can rely on Corollary \ref{cor:pseudoclassFIO}. Stronger bounds can 
be obtained by interpolating between Proposition \ref{prop:pseudolargep} and Theorem \ref{thm:pseudosmallp}, using different choices of $r_{0}$, $r_{1}$ and $\gamma$ than before.

Also note that, for all $0<p\leq \infty$, versions of \eqref{eq:mainpseudo1} and \eqref{eq:mainpseudo2} hold for values of $s$ that are not treated in Theorem \ref{thm:mainpseudo}. 
Again, one example is given by Corollary \ref{cor:pseudoclassFIO}, but other bounds can be obtained by varying the choice of $\gamma$ in \eqref{eq:gamma}. 

To avoid additional 
technicalities, we leave the details 
to the interested reader.
\end{remark}

\begin{corollary}\label{cor:mainpseudo}
Let $r>0$, $m,s\in\R$, $\delta\in[0,1/2]$ and $p\in(0,\infty]$. 
Set
\[
\rho:=\begin{cases}
0&\text{if }r>\frac{r(p)}{2(1-\delta)}\text{ and }r>2n\max(0,\frac{1}{p}-1),\\
\frac{r(p)-2(1-\delta)r}{2}+\veps&\text{if }2n\max(0,\frac{1}{p}-1)<r\leq \frac{r(p)}{2(1-\delta)},
\end{cases}
\]
for $\veps\in(0,2s(p)-\frac{r(p)-r}{2}]$, and let $\sigma$ be as in Theorem \ref{thm:mainpseudo}. 
 Then the following statements hold for each $a\in C^{r}_{*}S^{m}_{1,\delta}$.
\begin{enumerate}
\item\label{it:corpseudo1}
If $n\max(0,\tfrac{1}{p}-1)-\tfrac{r}{2}+s(p)-\sigma<s<r-s(p)$, then
\begin{equation}\label{eq:corpseudo1}
a(x,D):\HT^{s+\rho+m,p}_{FIO}(\Rn)\to \Hps.
\end{equation}
\item\label{it:corpseudo2}
If $n\max(0,\tfrac{1}{p}-1)-r+s(p)<s<\frac{r}{2}-s(p)+\sigma$, then
\begin{equation}\label{eq:corpseudo2}
a(x,D)^{*}:\HT^{s,p}_{FIO}(\Rn)\to \HT^{s-\rho-m,p}_{FIO}(\Rn).
\end{equation}
\item\label{it:corpseudo3}
If $a\in\Hrinf S^{m}_{1,\delta}$, then \eqref{eq:corpseudo1} also holds for $s=r-s(p)$. If, additionally, $p\geq1$, then \eqref{eq:corpseudo2} also holds for $s=n\max(0,\tfrac{1}{p}-1)-r+s(p)$.
\item\label{it:corpseudo4} If $a=b^{\flat}_{\delta}$ for some $b\in \Hrinf (\Rn)$, and if $p\geq 1$, then \eqref{eq:corpseudo1} holds for all $-\tfrac{r}{2}+s(p)-\sigma\leq s\leq r-s(p)$, with $m=-\delta r$.
\end{enumerate}
\end{corollary}
Note that $\rho=\max(0,\sigma-(1/2-\delta)r)$.
\begin{proof}
Write $a=a^{\sharp}_{1/2}+a^{\flat}_{1/2}$. Then $a^{\sharp}_{1/2}\in S^{m}_{1,1/2}$ and $a^{\flat}_{1/2}\in C^{r}_{*}S^{m-(1/2-\delta)r}_{1,1/2}$, by Lemma \ref{lem:smoothing}. Also, if $a\in \Hrinf S^{m}_{1,\delta}$ then $a^{\flat}_{1/2}\in \Hrinf S^{m-(1/2-\delta)r}_{1,1/2}$. By Lemma \ref{lem:pseudosmooth},
\[
a^{\sharp}_{1/2}(x,D):\HT^{s+\rho+m,p}_{FIO}(\Rn)\subseteq \HT^{s+m,p}_{FIO}(\Rn)\to\Hps
\]
for all $s\in\R$, and similarly for $a^{\sharp}_{1/2}(x,D)^{*}$. By Theorem \ref{thm:mainpseudo} and Remark \ref{rem:flatflat2},
\[
a^{\flat}_{1/2}(x,D):\HT^{s+\rho+m,p}_{FIO}(\Rn)\subseteq \HT^{s+\sigma+m-(1/2-\delta)r,p}_{FIO}(\Rn)\to\Hps
\]
for $s$ as in the statement of the corollary, and similarly for $a_{1/2}^{\flat}(x,D)^{*}$.
\end{proof}

\begin{remark}\label{rem:smallpproblem}
{Unlike in Theorem \ref{thm:mainpseudo}, where $\delta=1/2$, in Corollary \ref{cor:mainpseudo} it may happen that $\frac{r(p)}{2(1-\delta)}\leq 2n\max(0,\frac{1}{p}-1)$, for $p<1$ and $\delta<1/2$ small. In this case Corollary \ref{cor:mainpseudo} provides no information, but one can still obtain versions of \eqref{eq:corpseudo1} and \eqref{eq:corpseudo2}, using different methods. For example, one can appeal to Corollary \ref{cor:pseudoclassFIO} instead of Theorem \ref{thm:mainpseudo} in the proof above, or use more involved reasoning along the lines of Remark \ref{rem:suboptimalgamma}. The same strategy can be used to deal with certain values of $s$ that are not treated in Corollary \ref{cor:mainpseudo}.}
\end{remark}

\begin{remark}\label{rem:multiplication}
Corollary \ref{cor:mainpseudo} applies to multiplication with $C^{r}_{*}(\Rn)$ or $\Hrinf(\Rn)$ functions. In particular, if $r>\frac{r(p)}{2}$ and $r>2n\max(0,\frac{1}{p}-1)$, the multiplication operator with symbol $a\in C^{r}_{*}(\Rn)$ is bounded on $\Hps$ for $n\max(0,\frac{1}{p}-1)-r+s(p)<s<r-s(p)$. If $a\in  \Hrinf (\Rn)$ then one may also let 
$s=r-s(p)$. If, additionally, $p\geq1$, then $s=n\max(0,\frac{1}{p}-1)-r+s(p)$ is also allowed. This follows by applying \eqref{eq:corpseudo1} to $a$, and \eqref{eq:corpseudo2} to $\overline{a}$.
\end{remark}

\section{Wave equations with rough coefficients}\label{sec:roughwaves}

In this section we prove our main results for wave equations, and in particular Theorem \ref{thm:mainintro}. From now on we will restrict to $p\geq1$, in part due to issues related to integration in quasi-Banach spaces (see Remark \ref{rem:smallpnotwork}).

\subsection{Smooth first-order equations}\label{subsec:firstorder}

In this subsection we prove our main result for smooth first-order pseudodifferential equations.

To this end, we will rely on a parametrix construction from \cite{Hassell-Rozendaal23}. 
For $r>0$, $m\in\R$ and $\delta\in[0,1]$, recall the definition of the symbol class $\A^{r}S^{m}_{1,\delta}$ from Definition \ref{def:symbolsmooth}, as well as the notion of asymptotic homogeneity from Definition \ref{def:asymphom}. Then the relevant statement about a parametrix for first-order equations, contained in \cite[Theorem 7.1]{Hassell-Rozendaal23}, is as follows.

\begin{theorem}\label{thm:parametrix}
Let $b\in \A^{2}S^{1}_{1,1/2}$ be real-valued, elliptic and asymptotically homogeneous of degree $1$ with real-valued limit $a\in C^2_{-}S^1_{1,0}$, and let $p\in[1,\infty)$ and $s\in\R$. 
Then there exists a collection $(E_{t})_{t\in\R}$ of operators on $\Hps$ such that, for all $t_{0}>0$, there exists a $C\geq 0$ such that the following properties hold for all $f\in\Hps$:
\begin{enumerate}
\item\label{it:par1} $[t\mapsto E_{t}f]\in C^{k}(\R;\HT^{s-k,p}_{FIO}(\Rn))$ for $k\in\{0,1\}$;
\item\label{it:par2} $\|\partial_{t}^{k}E_{t}f\|_{\HT^{s-k,p}_{FIO}(\Rn)}\leq C\|f\|_{\Hps}$ for $k\in\{0,1\}$ and $t\in[-t_{0},t_{0}]$;
\item\label{it:par3} $E_{0}f=f$, 
\[
\|(D_{t}-b(x,D))E_{t}f\|_{\Hps}\leq C\|f\|_{\Hps}
\]
for all $t\in[-t_{0},t_{0}]$, and $[t\mapsto(D_{t}-b(x,D))E_{t}f]\in C(\R;\Hps)$.
\end{enumerate}
\end{theorem}

\begin{remark}\label{rem:parametrix}
The collection $(E_{t})_{t\in\R}$ in Theorem \ref{thm:parametrix} is independent of the choice of $p$ and $s$. In fact, one has $E_{t}:=\wt{V}\F_{t}\wt{W}$ 
for all $t\in\R$, where $\wt{W}$ and $\wt{V}$ are wave packet transforms similar to $W$, from \eqref{eq:defW}, and its adjoint $V$. Moreover, $\F_{t}$ is pullback via the bicharacteristic flow associated with the symbol $(x,\eta)\mapsto \chi(\eta)b(x,\eta)$, where $\chi\in C^{\infty}(\Rn)$ satisfies $\chi(\eta)=0$ for $|\eta|\leq 8$, and $\chi(\eta)=1$ for $|\eta|\geq 16$. Note that $\F_{t}$ acts on functions on $T^{*}(\Rn)\setminus o$, and the latter coincides with $\Sp\times(0,\infty)$ after the change of variables $(x,\xi)\mapsto (x,\hat{\xi},|\xi|^{-1})$.
\end{remark}

We can now prove our main result regarding smooth first-order equations.

\begin{theorem}\label{thm:firstorder}
Let $b_{1}\in \A^{2}S^{1}_{1,1/2}$ be real-valued, elliptic and asymptotically homogeneous of degree one with real-valued limit $a_1\in C^2_{-}S^1_{1,0}$, let $b_{2}\in S^{0}_{1,1/2}$, 
and set $b:=b_{1}+b_{2}$. Then there exists a unique collection $(\ex_{t})_{t\in\R}$ such that, for all $p\in[1,\infty)$, $s\in\R$, $k\in\Z_{+}$ and $t_{0}>0$, there exists a $C\geq 0$ such that the following properties hold for all $f\in\Hps$:
\begin{enumerate}
\item\label{it:firstorder0} $\ex_{t}:\Hps\to\Hps$ is a bounded operator for all $t\in\R$;
\item\label{it:firstorder1} $[t\mapsto \ex_{t}f]\in C^{k}(\R;\HT^{s-k,p}_{FIO}(\Rn))$;
\item\label{it:firstorder2} $\|\partial_{t}^{k}\ex_{t}f\|_{\HT^{s-k,p}_{FIO}(\Rn)}\leq C\|f\|_{\Hps}$ for all $t\in[-t_{0},t_{0}]$;
\item\label{it:firstorder3} $\ex_{0}f=f$, and $D_{t}\ex_{t}f=b(x,D)\ex_{t}f
$ for all $t\in\R$.
\end{enumerate}
\end{theorem}
The collection $(\ex_{t})_{t\in\R}$ is independent of the choice of $p$ and $s$, cf.~\eqref{eq:defe}, \eqref{eq:defVj} and Remark \ref{rem:parametrix}. 
\begin{proof}
In the case where 
$b_{2}=0$, the statement is \cite[Theorem 4.7]{Hassell-Rozendaal23}. However, the proof given there also works for general $b_{2}\in S^{0}_{1,1/2}$, and in fact such generality is required to prove the uniqueness statement (see \eqref{eq:adjointdecomp}). We will sketch the relevant steps, referring to the proof of \cite[Theorem 4.7]{Hassell-Rozendaal23} for additional details.



\subsubsection{Existence}

Let $(E_{t})_{t\in\R}$ be the parametrix from Theorem \ref{thm:parametrix}, associated with $b_{1}$. Set $V_{0}f(t):=-i(D_{t}-b(x,D))E_{t}f$ and, recursively,
\begin{equation}\label{eq:defVj}
V_{j+1}f(t):=-i\int_{0}^{t}(D_{t}-b(x,D))E_{t-\tau}V_{j}f(\tau)\ud \tau,
\end{equation}
for $f\in\Hps$, $j\geq0$ and $t\in\R$. Due to Theorem \ref{thm:parametrix}, 
there exists a $C_{0}\geq0$ such that $V_{j}f\in C(\R;\Hps)$ and
\begin{equation}\label{eq:iterationbound}
\|V_{j}f(t)\|_{\Hps}\leq \frac{C_{0}^{j+1}t^{j}}{j!}\|f\|_{\Hps}.
\end{equation}
Hence $V:=\sum_{k=0}^{\infty}V_{k}$ defines a bounded operator
\begin{equation}\label{eq:Vbound}
V:\Hps\to C([-t_{0},t_{0}];\Hps)\subseteq L^{1}([-t_{0},t_{0}];\Hps).
\end{equation}
In particular, 
$Vf\in C(\R;\Hps)\subseteq L^{1}_{\loc}(\R;\Hps)$ for all $f\in\Hps$.

Next, for $f\in\Hps$ and $t\in\R$, set
\begin{equation}\label{eq:defe}
\ex_{t}f:=E_{t}f+\int_{0}^{t}E_{t-\tau}Vf(\tau)\ud \tau.
\end{equation}
For $k=0$, \eqref{it:firstorder1} and \eqref{it:firstorder2} then follow by combining \eqref{eq:Vbound} with parts \eqref{it:par1} and \eqref{it:par2} of Theorem \ref{thm:parametrix}. This implies in particular \eqref{it:firstorder0}. The same reasoning yields
\begin{equation}\label{eq:exder}
\partial_{t}\ex_{t}f=\partial_{t}E_{t}f+Vf(t)+\int_{0}^{t}\partial_{t}E_{t-\tau}Vf(\tau)\ud\tau,
\end{equation}
and then that \eqref{it:firstorder1} and \eqref{it:firstorder2} hold for $k=1$. Next, by Theorem \ref{thm:parametrix} \eqref{it:par3} one has $\ex_{0}f=U_{0}f=f$, and \eqref{eq:exder} and the definition of $V$ imply that $(D_{t}-b(x,D))\ex_{t}f=0$, thereby proving \eqref{it:firstorder3}. Finally, by Lemma \ref{lem:pseudosmooth}, $b(x,D)^{k}:\Hps\to\HT^{s-k,p}_{FIO}(\Rn)$ for all $k\geq1$, so that \eqref{it:firstorder3} immediately yields \eqref{it:firstorder1} and \eqref{it:firstorder2} for $k\geq2$ as well. 

\subsubsection{Uniqueness} 

We will prove that, if
\[
u\in C(\R;\Hps)\cap C^{1}(\R;\Hpsm)
\]
is such that $u(0)=0$ and $(D_{t}-b(x,D))u(t)=0$ for all $t\in\R$, then $u\equiv 0$. Write $u_{+}(t):=\ind_{[0,\infty)}(t)u(t)$ and $u_{-}(t):=u(t)-u_{+}(t)$. Then $u=u_{+}+u_{-}$,
\[
u_{+},u_{-}\in C(\R;\Hps)\cap W^{1,1}_{\loc}(\R;\Hpsm),
\]
and $(D_{t}-b(x,D))u_{+}(t)=(D_{t}-b(x,D))u_{-}(t)=0$ for almost all $t\in\R$. The latter identity and Lemma \ref{lem:pseudosmooth} imply that in fact $u_{+},u_{-}\in C^{1}(\R;\Hpsm)$. By symmetry, it suffices to show that $\int_{\R}\lb u_{+}(t),G(t)\rb\ud t=0$ for all $G\in C^{\infty}_{c}(\R;C^{\infty}_{c}(\Rn))$. 

By Lemma \ref{lem:Aduality}, there exist $\tilde{b}_{1}\in \A^{2}S^{1}_{1,1/2}$, real-valued, elliptic and asymptotically homogeneous of degree $1$, and $\tilde{b}_{2}\in S^{0}_{1,1/2}$, such that 
\begin{equation}\label{eq:adjointdecomp}
b(x,D)^{*}=\tilde{b}_{1}(x,D)+\tilde{b}_{2}(x,D).
\end{equation}
Let $(\tilde{\ex}_{t})_{t\in\R}$ be as in the previous part of the proof, with $b$ replaced by $\tilde{b}:=\tilde{b}_{1}+\tilde{b}_{2}$, and let $t_{0}>0$ be such that $G(t)=0$ for $t\geq t_{0}$. Set $w(t):=-i\int_{t}^{t_{0}}\tilde{\ex}_{t-\tau}G(\tau)\ud\tau$ for $t\in\R$. Then $w\in C^{k}(\R;\HT^{\sigma,q}_{FIO}(\Rn))$ for all $k\geq0$, $\sigma\in\R$ and $q\in[1,\infty)$, with $w(t)=0$ for $t\geq t_{0}$ and $(D_{t}-\tilde{b}(x,D))w(t)=G(t)$ for all $t\in\R$. By \eqref{eq:fracintHpFIO1} 
 and Lemma \ref{lem:pseudosmooth}, 
\[
w\in C(\R;\HT^{1-s,p'}_{FIO}(\Rn))\cap C^{1}(\R;\HT^{-s,p'}_{FIO}(\Rn))
\]
and $\tilde{b}(x,D)w\in C(\R;\HT^{-s,p'}_{FIO}(\Rn))$. Hence
\begin{align*}
\int_{\R}\lb u_{+}(t),G(t)\rb\ud t&=\int_{\R}\lb u_{+}(t),(D_{t}-\tilde{b}(x,D))w(t)\rb\ud t\\
&=\int_{\R}\lb (D_{t}-b(x,D))u_{+}(t),w(t)\rb\ud t=0,
\end{align*}
where we used \eqref{eq:dualHpFIO} and the regularity and support conditions of $u_{+}$ and $w$.
\end{proof}

\begin{remark}\label{rem:commute}
It follows from Theorem \ref{thm:firstorder} and basic semigroup theory (see e.g.~\cite[Theorem II.6.7]{Engel-Nagel00}) that $(\ex_{t})_{t\in\R}$ is a strongly continuous group on $\Hps$, for all $p\in[1,\infty)$ and $s\in\R$, with generator $ib(x,D)$. 
In particular, $b(x,D)\ex_{t}f=\ex_{t}b(x,D)f$ for all $f\in\Hps$.  
\end{remark}

\begin{remark}\label{rem:smallpnotwork}
In \eqref{eq:iterationbound}, we used that $\Hps$ is a Banach space for $p\geq1$, when implicitly applying the triangle inequality. In fact, there are various fundamental issues concerning integration of functions with values in a quasi-Banach space (see e.g.~\cite{Albiac-Ansorena13}). Hence, for simplicity, we only consider $p\geq1$ in this section.
\end{remark}

\subsection{Rough second-order equations}\label{subsec:secondorder}

In this subsection we state and prove our main result for rough second-order equations. 

We consider the following differential operator: 
\[
Lf(x):=\sum_{i,j=1}^{n}D_{i}(a_{ij}D_{j}f)(x)+\sum_{j=1}^{n}a_{j}(x)D_{j}f(x)+a_{0}(x)f(x).
\]
Here $a_{ij}:\Rn\to\R$ is bounded and real-valued for all $1\leq i,j\leq n$, and there exists a $\kappa_{0}>0$ such that
\[
\sum_{i,j=1}^{n}a_{ij}(x)\eta_{i}\eta_{j}\geq \kappa_{0}|\eta|^{2}
\]
for all $x,\eta\in\Rn$. Moreover, $a_{j}:\Rn\to\C$ is bounded for all $0\leq j\leq n$. Crucially, we suppose that $(a_{ij})_{i,j=1}^{n}\subseteq {C^{2}_{-}(\Rn)\cap C^{r}_{*}(\Rn)}$ and $(a_{j})_{j=0}^{n}\subseteq C^{r}_{*}(\Rn)$ for some $r\geq 2$. We sometimes strengthen this assumption slightly to $(a_{ij})_{i,j=1}^{n}\subseteq C^{2}_{-}(\Rn)\cap \HT^{r,\infty}(\Rn)$ and $(a_{j})_{j=0}^{n}\subseteq \HT^{r,\infty}(\Rn)$, leading to stronger results. 

We will prove existence and uniqueness of solutions to the following Cauchy problem:
\begin{equation}\label{eq:divwave}
\begin{aligned}
(D_{t}^{2}-L)u(t,x)&=F(t,x),\\
u(0,x)&=u_{0}(x),\\
\partial_{t}u(0,x)&=u_{1}(x),
\end{aligned}
\end{equation}
for $u_{0}$, $u_{1}$ and $F(t,\cdot)$ in suitable Hardy spaces for Fourier integral operators.

Our proof makes crucial use of the symbol smoothing procedure from Section \ref{subsec:symbols}. As in Lemma \ref{lem:smoothing}, for $1\leq i,j\leq n$ we write $a_{ij}=(a_{ij})^{\sharp}_{1/2}+(a_{ij})^{\flat}_{1/2}$ with $(a_{ij})^{\sharp}_{1/2}\in \A^{2}S^{0}_{1,1/2}$ and $(a_{ij})^{\flat}_{1/2}\in C^{r}_{*} S^{-r/2}_{1,1/2}$, and $(a_{ij})^{\flat}_{1/2}\in \HT^{r,\infty} S^{-r/2}_{1,1/2}$ if $a_{ij}\in\HT^{r,\infty}(\Rn)$. Also write $L=L_{1,1}+L_{2,1}+L_{3}$, where 
\begin{equation}\label{eq:Ldecomp}
\begin{aligned}
L_{1,1}:=&\sum_{i,j=1}^{n}D_{i}(a_{ij})^{\sharp}_{1/2}(x,D)D_{j},\\
L_{2,1}:=&\sum_{i,j=1}^{n}D_{i}(a_{ij})^{\flat}_{1/2}(x,D)D_{j},\\
L_{3}:=&\sum_{j=1}^{n}a_{j}D_{j}+a_{0}. 
\end{aligned}
\end{equation}
Note that $(L_{1,1}+L_{2,1})^{*}=L_{1,2}+L_{2,2}$, where
\begin{equation}\label{eq:Ldecompextra}
\begin{aligned}
L_{1,2}:=&\sum_{i,j=1}^{n}D_{j}(a_{ij})^{\sharp}_{1/2}(x,D)D_{i},\\
L_{2,2}:=&\sum_{i,j=1}^{n}D_{j}(a_{ij})^{\flat}_{1/2}(x,D)D_{i}. 
\end{aligned}
\end{equation}
Finally, for $x,\eta\in\Rn$, set
\[
A(x,\eta):=\sum_{i,j=1}^{n}a_{ij}(x)\eta_{i}\eta_{j}.
\]
Then $A\in \Crtwo S^{2}_{1,0}$ is non-negative, elliptic and homogeneous of degree $2$. 

The following proposition, an extension of \cite[Proposition 5.1]{Hassell-Rozendaal23}, records some important properties of these operators. 

\begin{proposition}\label{prop:Lproperties}
Let $p\in[1,\infty]$ be such that $2s(p)+1<r$. Then the following statements hold for each $k\in\{1,2\}$.
\begin{enumerate}
\item\label{it:Lprop1} There exist a real-valued elliptic $b\in\A^{2}S^{1}_{1,1/2}$, and $e_{1,k},e_{2,k}\in S^{1}_{1,1/2}$, such that 
\begin{equation}\label{eq:Asharpdecomp}
L_{1,k}=b(x,D)^{2}+e_{1,k}(x,D)\text{ and }L_{1,k}^{*}=b(x,D)^{2}+e_{2,k}(x,D).
\end{equation}
One may choose $b$ to be independent of $k$ and asymptotically homogeneous of degree $1$ with limit given by $\sqrt{A(x,\eta)}$ for $x,\eta\in\Rn$ with $|\eta|\geq1$.
\item\label{it:Lprop2} There exists a $\delta>0$ such that
\begin{equation}\label{eq:Aflatbound1}
L_{2,k}:\Hps\to\Hpsm
\end{equation}
for all $r-s(p)-\delta<  s< r-s(p)$, and
\begin{equation}\label{eq:Aflatbound2}
L_{2,k}^{*}:\Hps\to\Hpsm
\end{equation}
for all $-r+s(p)+1<s<-r+s(p)+1+\delta$. One may also suppose that
\begin{equation}\label{eq:delta}
-r+s(p)+1<r-s(p)-\delta.
\end{equation}
If $(a_{ij})_{i,j=1}^{n}\subseteq \HT^{r,\infty}(\Rn)$, then \eqref{eq:Aflatbound1} also holds for $s=r-s(p)$, and \eqref{eq:Aflatbound2} also holds for $s=-r+s(p)+1$.
\item\label{it:Lprop3} One has 
\begin{equation}\label{eq:L3bound1}
L_{3}:\Hps\to\Hpsm
\end{equation}
for all $-r+s(p)+1<s<r-s(p)+1$, and
\begin{equation}\label{eq:L3bound2}
L_{3}^{*}:\Hps\to\Hpsm
\end{equation}
for all $-r+s(p)<s<r-s(p)$. If $(a_{j})_{j=0}^{n}\subseteq \HT^{r,\infty}(\Rn)$, then \eqref{eq:L3bound1} holds for all $-r+s(p)+1\leq s\leq r-s(p)+1$, and \eqref{eq:L3bound2} holds for all $-r+s(p)\leq s\leq r-s(p)$. 
\item\label{it:Lprop4} One has  
\begin{equation}\label{eq:Lmap}
L:\Hps\to\HT^{s-2,p}_{FIO}(\Rn)\text{ and }L^{*}:\Hps\to\HT^{s-2,p}_{FIO}(\Rn)
\end{equation}
for all $-r+s(p)+1< s< r-s(p)+1$. If $(a_{ij})_{i,j=1}^{n},(a_{j})_{j=0}^{n}\subseteq \HT^{r,\infty}(\Rn)$, then \eqref{eq:Lmap} holds for all $-r+s(p)+1\leq  s\leq  r-s(p)+1$. 
\end{enumerate}
\end{proposition}
\begin{proof}
\eqref{it:Lprop1}: The symbol of $L_{1,1}$ is given by
\[
(x,\eta)\mapsto\sum_{i,j=1}^{n}(D_{i}(a_{ij})^{^{\sharp}}_{1/2})(x,\eta)\eta_{j}+\sum_{i,j=1}^{n}(a_{ij})^{\sharp}_{1/2}(x,\eta)\eta_{i}\eta_{j}.
\]
The first term is an element of $\A^{1}S^{1}_{1,1/2}\subseteq S^{1}_{1,1/2}$, and the second term is $A^{\sharp}_{1/2}(x,\xi)$. 
Lemma \ref{lem:squareroot} yields a $b\in \A^{2}S^{1}_{1,1/2}$ with the stated properties, and an $e_{1}\in S^{1}_{1,1/2}$, such that $A^{\sharp}_{1/2}(x,D)=b(x,D)^{2}+e_{1}(x,D)$. This suffices for the first part of \eqref{eq:Asharpdecomp} if $k=1$. The argument for $k=2$ is identical, noting that the highest-order term in the symbol of $L_{1,2}$ is again $A^{\sharp}_{1/2}$. 

For the second part of \eqref{eq:Asharpdecomp}, one can additionally use Lemma \ref{lem:Aduality}, to write
\[
A^{\sharp}_{1/2}(x,D)^{*}=\overline{A^{\sharp}_{1/2}}(x,D)+e_{2}(x,D)=A^{\sharp}_{1/2}(x,D)+e_{2}(x,D)
\]
for some $e_{2}\in S^{1}_{1,1/2}$.

\eqref{it:Lprop2}: By Lemma \ref{lem:smoothing}, $(a_{ij})^{\flat}_{1/2}\in C^{r}_{*} S^{-r/2}_{1,1/2}$ for all $1\leq i,j\leq n$. 
Let $\sigma$ and $\veps$ be as in Theorem \ref{thm:mainpseudo}. Then $s-1\geq s-r/2+\sigma$ for $\veps$ sufficiently small, by the condition on $p$. Hence Theorem \ref{thm:mainpseudo} yields
\[
(a_{ij})^{\flat}_{1/2}(x,D):\Hpsm\subseteq \HT^{s-r/2+\sigma,p}_{FIO}(\Rn) \to\Hps
\]
for $-\frac{r}{2}+s(p)-\sigma<s<r-s(p)$. If $a_{ij}\in\HT^{r,\infty}(\Rn)$, then $(a_{ij})^{\flat}_{1/2}\in\HT^{r,\infty}S^{-r/2}_{1,1/2}$ and one may also let $s=r-s(p)$. This suffices for \eqref{eq:Aflatbound1}, since $-\frac{r}{2}+s(p)-\sigma<r-s(p)$ and $D_{i},D_{j}:\Hps\to \Hpsm$. The argument for \eqref{eq:Aflatbound2} is analogous, using \eqref{eq:mainpseudo2}. Finally, \eqref{eq:delta} follows from the fact that $-r+s(p)+1<r-s(p)$.

\eqref{it:Lprop3}: By Remark \ref{rem:multiplication}, $a_{j}$ and $\overline{a_{j}}$ act boundedly on $\Hps$ for all $0\leq j\leq n$ and $-r+s(p)<s<r-s(p)$, and for all $-r+s(p)\leq s\leq r-s(p)$ if $a_{j}\in\HT^{r,\infty}(\Rn)$. This suffices, since $L_{3}^{*}f=\sum_{j=1}^{n}D_{j}(\overline{a_{j}}f)+\overline{a_{0}}f$.

\eqref{it:Lprop4}:  Lemma \ref{lem:pseudosmooth} and \eqref{it:Lprop1} imply that $L_{1,1}$ and $L_{1,2}^{*}$ map $\Hps$ to $\HT^{s-2,p}_{FIO}(\Rn)$, for all $s\in\R$. Now write 
\[
L=L_{1,1}+L_{2,1}+L_{3}=L_{1,2}^{*}+L_{2,2}^{*}+L_{3},
\]
and apply \eqref{it:Lprop2} and \eqref{eq:L3bound1} to obtain the first part of \eqref{eq:Lmap} for $r-s(p)-\delta<s<r-s(p)$ and $-r+s(p)+1< s<-r+s(p)+1+\delta$. 
Then \eqref{eq:intHpFIO} in turn yields the first part of \eqref{eq:Lmap} for all $-r+s(p)+1< s<r-s(p)$. For $r-s(p)<s< r-s(p)+1$, use \eqref{eq:Aflatbound1}:
\[
L_{2,1}:\Hps\subseteq\HT^{s-1,p}_{FIO}(\Rn)\to\HT^{s-2,p}_{FIO}(\Rn),
\]
and combine this with \eqref{eq:L3bound1}. 

For the second part of \eqref{eq:Lmap}, write
\[
L^{*}=L_{1,1}^{*}+L_{2,1}^{*}+L_{3}^{*}=L_{1,2}+L_{2,2}+L_{3}^{*}.
\]
As before, one may then combine \eqref{it:Lprop1}, \eqref{it:Lprop2} and \eqref{eq:L3bound2} with \eqref{eq:intHpFIO} to obtain the second part of \eqref{eq:Lmap} for all $-r+s(p)+1< s<r-s(p)$. To also deal with $r-s(p)<s< r-s(p)+1$, one has to use that
\begin{align*}
L_{2,2}&:\Hps\subseteq\HT^{s-1,p}_{FIO}(\Rn)\to\HT^{s-2,p}_{FIO}(\Rn),\\
L_{3}^{*}&:\Hps\subseteq\HT^{s-1,p}_{FIO}(\Rn)\to\HT^{s-2,p}_{FIO}(\Rn),
\end{align*}
where we again applied \eqref{eq:Aflatbound1} and \eqref{eq:L3bound2}.

Finally, if $(a_{ij})_{i,j=1}^{n},(a_{j})_{j=0}^{n}\subseteq\HT^{r,\infty}(\Rn)$ then the same arguments let one include 
$s=-r+s(p)+1$ and $s=r-s(p)+1$. 
\end{proof}

We are now ready to prove our main result on the well-posedness of \eqref{eq:divwave}. 

\begin{theorem}\label{thm:mainwave}
There exist unique collections $(U_{0}(t))_{t\in\R},(U_{1}(t))_{t\in\R}$ such that, for all $p\in [1,\infty)$ and $s\in\R$ with $2s(p)+1<r$ and $-r+s(p)+1< s< r-s(p)$, and for all $u_{0}\in\Hps$, $u_{1}\in \Hpsm$, $F\in L^{1}_{\loc}(\R;\Hpsm)$ and $t_{0}>0$, the following properties hold:
\begin{enumerate}
\item\label{it:mainwave1} $U_{k}(t):\HT^{s-k,p}_{FIO}(\Rn)\to\Hps$ is bounded for all $t\in\R$ and $k\in\{0,1\}$, and $\sup_{|t|\leq t_{0}}\|U_{k}(t)\|_{\La(\HT^{s-k,p}_{FIO}(\Rn),\Hps)}<\infty$; 
\item\label{it:mainwave3} $[t\mapsto U_{0}(t)u_{0}],[t\mapsto U_{1}(t)u_{1}]\in C^{k}(\R;\HT^{s-k,p}_{FIO}(\Rn))$ for $k\in\{0,1,2\}$;
\item\label{it:mainwave4} Set $u(t):=U_{0}(t)u_{0}+U_{1}(t)u_{1}-\int_{0}^{t}U_{1}(t-s)F(s)\ud s$ for $t\in\R$. Then 
\begin{equation}\label{eq:regularityu}
u\in C(\R;\Hps)\cap C^{1}(\R;\Hpsm)\cap W^{2,1}_{\loc}(\R;\HT^{s-2,p}_{FIO}(\Rn)),
\end{equation}
$u(0)=u_{0}$, $\partial_{t}u(0)=u_{1}$, and
\begin{equation}\label{eq:eqdistr}
(D_{t}^{2}-L)u(t)=F(t)
\end{equation}
in $\HT^{s-2,p}_{FIO}(\Rn)$ for almost all $t\in \R$.
\end{enumerate}
If 
$(a_{ij})_{i,j=1}^{n},(a_{j})_{j=0}^{n}\subseteq \HT^{r,\infty}(\Rn)$, then this statement also holds for $s=-r+s(p)+1$ and $s=r-s(p)$. 
\end{theorem}

The operators $(U_{0}(t))_{t\in\R}$ and $(U_{1}(t))_{t\in\R}$ do not depend on the choice of $r$; only their mapping properties do. Moreover, although the operators do not depend on $p$ or $s$, it follows from the proof that the uniqueness statement does hold for all $p$ and $s$ separately. Note also that \eqref{eq:eqdistr} is well defined, by Proposition \ref{prop:Lproperties} \eqref{it:Lprop4}. 

\begin{proof} 
The proof is similar to that of \cite[Theorem 5.2]{Hassell-Rozendaal23}, but there are various twists. 

\subsubsection{Existence} 


This part of the proof involves similar arguments as used for \cite[Theorem 5.2]{Hassell-Rozendaal23}. However, we immediately obtain existence for the whole Sobolev range, and we explicitly construct the solution operators $U_{0}$ and $U_{1}$, instead of the solution to \eqref{eq:eqdistr}. We will indicate the main steps, referring to \cite{Hassell-Rozendaal23} for additional details.

We first do some preliminary work. Write 
\[
L=L_{1,1}+L_{2,1}+L_{3}=b(x,D)^{2}+e_{1,1}(x,D)+L_{2,1}+L_{3},
\]
as in \eqref{eq:Ldecomp} and Proposition \ref{prop:Lproperties} \eqref{it:Lprop1}. 
Let $(\ex_{t})_{t\in\R}$ be the collection of operators from Theorem \ref{thm:firstorder}, satisfying $D_{t}\ex_{t}f=b(x,D)\ex_{t}f$ for all $f\in\Hps$ and $t\in\R$. As in Lemma \ref{lem:pseudosmooth}, let $c>0$ be such that 
\[
\tilde{b}(x,D):=b(x,D)+ic:\Hps\to\Hpsm
\]
is invertible. Set $\tilde{\ex}_{t}:=e^{-ct}\ex_{t}$ and
\[
\tilde{L}:=\tilde{b}(x,D)^{2}-L=2icb(x,D)-c^{2}-e_{1,1}(x,D)-L_{2,1}-L_{3}.
\]
Then $\tilde{\ex}_{0}f=f$,
\[
D_{t}\tilde{\ex}_{t}f=(b(x,D)+ic)\tilde{\ex}_{t}f=\tilde{b}(x,D)\tilde{\ex}_{t}f
\]
and
\begin{equation}\label{eq:factorA}
\big(D_{t}^{2}-L\big)\tilde{\ex}_{t}f=(\tilde{b}(x,D)^{2}-L)\tilde{\ex}_{t}f=\tilde{L}\tilde{\ex}_{t}f.
\end{equation}
Moreover,
\begin{equation}\label{eq:factorA2}
\big(D_{t}^{2}-L\big)\tilde{\ex}_{-t}f=(\tilde{b}(x,D)^{2}-L)\tilde{\ex}_{-t}f
=\tilde{L}\tilde{\ex}_{-t}f
\end{equation}
as well.

Now, by \eqref{eq:Aflatbound1}, \eqref{eq:delta} and \eqref{eq:L3bound1}, $L_{2,1}$ and $L_{3}$ map $\Hps$ to $\Hpsm$ for $r-s(p)-\delta<s< r-s(p)$. 
Hence Lemma \ref{lem:pseudosmooth} implies that
\begin{equation}\label{eq:mappingLtilde}
\tilde{L}:\Hps\to\Hpsm
\end{equation}
for such $s$. If $(a_{ij})_{i,j=1}^{n},(a_{j})_{j=0}^{n}\subseteq\HT^{r,\infty}(\Rn)$, then \eqref{eq:mappingLtilde} also holds for $s=r-s(p)$. On the other hand, using the decomposition from \eqref{eq:Ldecompextra}, Proposition \ref{prop:Lproperties} \eqref{it:Lprop1} yields
\begin{equation}\label{eq:Ldecomp2}
L=L_{1,2}^{*}+L_{2,2}^{*}+L_{3}=b(x,D)^{2}+e_{2,2}(x,D)+L_{2,2}^{*}+L_{3}.
\end{equation}
In particular, 
\[
e_{2,2}(x,D)+L_{2,2}^{*}+L_{3}=L-b(x,D)^{2}=e_{1,1}(x,D)+L_{2,1}+L_{3},
\]
and therefore
\begin{equation}\label{eq:tildeL2}
\tilde{L}=2icb(x,D)-c^{2}-e_{2,2}(x,D)-L_{2,2}^{*}-L_{3}.
\end{equation}
By combining this with \eqref{eq:Aflatbound2} and \eqref{eq:L3bound1}, as well as Lemma \ref{lem:pseudosmooth}, one sees that \eqref{eq:mappingLtilde} also holds for $-r+s(p)+1<s<-r+s(p)+1+\delta$, and for $s=-r+s(p)+1$ if $(a_{ij})_{i,j=1}^{n},(a_{j})_{j=0}^{n}\subseteq\HT^{r,\infty}(\Rn)$. 
Finally, one can use \eqref{eq:intHpFIO} and interpolate to see that \eqref{eq:mappingLtilde} holds for all $-r+s(p)+1<s<r-s(p)$, with the endpoints included if $(a_{ij})_{i,j=1}^{n},(a_{j})_{j=0}^{n}\subseteq\HT^{r,\infty}(\Rn)$. 

 Next, set
\begin{equation}\label{eq:defcossin}
\cs_{t}:=\frac{\tilde{\ex}_{t}+\tilde{\ex}_{-t}}{2}\quad\text{and}\quad\sn_{t}:=\frac{\tilde{\ex}_{t}-\tilde{\ex}_{-t}}{2i}\tilde{b}(x,D)^{-1}
\end{equation}
for $t\in\R$. Then, partly due to Remark \ref{rem:commute},
\begin{equation}\label{eq:cossinprop}
\cs_{0}u_{0}=u_{0},\quad \partial_{t}\cs_{t}u_{0}|_{t=0}=0,\quad \sn_{0}u_{1}=0,\quad \partial_{t}\sn_{t}u_{1}|_{t=0}=u_{1}.
\end{equation}
Moreover, by \eqref{eq:factorA} and \eqref{eq:factorA2}, for all $t\in\R$ one has
\begin{equation}\label{eq:dercossin}
(D_{t}^{2}-L)\cs_{t}u_{0}=\tilde{L}\cs_{t}u_{0}\quad\text{and}\quad(D_{t}^{2}-L)\sn_{t}u_{1}=\tilde{L}\sn_{t}u_{1}.
\end{equation}
Finally, by Theorem \ref{thm:firstorder}, for all $t_{0}>0$ and $k\geq0$ one has
\begin{equation}\label{eq:cossinbounds}
\sup_{|t|\leq t_{0}}\|\partial_{t}^{k}\cs_{t}\|_{\La(\Hps,\HT^{s-k,p}_{FIO}(\Rn))}+\|\partial_{t}^{k}\sn_{t}\|_{\La(\Hpsm,\HT^{s-k,p}_{FIO}(\Rn))}<\infty,
\end{equation}
where the derivatives are taken in the strong operator topology. 

For $t\in\R$, set $v_{0,0}(t):=\tilde{L}\cs_{t}u_{0}$ and $v_{1,0}(t):=\tilde{L}\sn_{t}u_{1}$. Recursively, let
\[
v_{j,k+1}(t):=\int_{0}^{t}\tilde{L}\sn_{t-\tau}v_{j,k}(\tau)\ud \tau
\] 
for $j\in\{0,1\}$ and $k\geq0$. Using \eqref{eq:mappingLtilde} and \eqref{eq:cossinbounds}, one can show that $v_{j}:=\sum_{k=0}^{\infty}v_{j,k}$ is well defined in $C(\R;\Hpsm)\subseteq L^{1}_{\loc}(\R;\Hpsm)$. Moreover, 
\begin{equation}\label{eq:boundv}
\sup_{|t|\leq t_{0}}\|v_{j}(t)\|_{\Hpsm}
\lesssim \|u_{j}\|_{\HT^{s-j,p}_{FIO}(\Rn)}
\end{equation}
for each $t_{0}>0$, where the implicit constant depends on $t_{0}$ but not on $u_{j}$.

We can now define our solution operators. For $t\in\R$, set
\begin{equation}\label{eq:defU0}
U_{0}(t)u_{0}:=\cs_{t}u_{0}+\int_{0}^{t}\sn_{t-\tau}v_{0}(\tau)\ud \tau
\end{equation}
and
\begin{equation}\label{eq:defU1}
U_{1}(t)u_{1}:=\sn_{t}u_{1}+\int_{0}^{t}\sn_{t-\tau}v_{1}(\tau)\ud \tau.
\end{equation}
Note that $v_{j}$ depends on $u_{j}$, but $U_{0}$ and $U_{1}$ do not depend on the choice of $p$ and $s$. By \eqref{eq:cossinbounds} and \eqref{eq:boundv},
\[
\sup_{|t|\leq t_{0}}\|U_{j}(t)u_{j}\|_{\Hps}\lesssim \|u_{j}\|_{\HT^{s-j,p}_{FIO}(\Rn)}
\]
for $j\in\{0,1\}$ and $t_{0}>0$, where the implicit constant depends on $t_{0}$ but not on $u_{j}$. This proves \eqref{it:mainwave1}.

Next, the dominated convergence theorem, \eqref{eq:boundv}, \eqref{eq:cossinbounds} and \eqref{eq:cossinprop} imply that
\[
[t\mapsto U_{j}(t)u_{j}]\in C(\R;\Hps)\cap C^{1}(\R;\Hpsm)\cap C^{2}(\R;\Hpsmm),
\]
with 
\begin{equation}\label{eq:firstderUj}
\begin{aligned}
\partial_{t}U_{0}(t)u_{0}&=\partial_{t}\cs_{t}u_{0}+\int_{0}^{t}\partial_{t}\sn_{t-\tau}v_{0}(\tau)\ud \tau,\\
\partial_{t}U_{1}(t)u_{1}&=\partial_{t}\sn_{t}u_{1}+\int_{0}^{t}\partial_{t}\sn_{t-\tau}v_{1}(\tau)\ud \tau,
\end{aligned}
\end{equation}
in $\Hpsm$, and
\begin{equation}\label{eq:secondderUj}
\begin{aligned}
\partial_{t}^{2}U_{0}(t)u_{0}&=\partial_{t}^{2}\cs_{t}u_{0}+v_{0}(t)+\int_{0}^{t}\partial_{t}^{2}\sn_{t-\tau}v_{0}(\tau)\ud \tau,\\
\partial_{t}^{2}U_{1}(t)u_{1}&=\partial_{t}^{2}\sn_{t}u_{1}+v_{1}(t)+\int_{0}^{t}\partial_{t}^{2}\sn_{t-\tau}v_{1}(\tau)\ud \tau,
\end{aligned}
\end{equation}
in $\Hpsmm$. This proves \eqref{it:mainwave3}. Moreover, combined with \eqref{eq:dercossin}, it implies that 
\begin{equation}\label{eq:solveUj}
\begin{aligned}
(D_{t}^{2}-L)U_{0}(t)u_{0}&=\tilde{L}\cs_{t}u_{0}-v_{0}(t)+\int_{0}^{t}\tilde{L}\sn_{t-\tau}v_{0}(\tau)\ud\tau=0,\\
(D_{t}^{2}-L)U_{1}(t)u_{1}&=\tilde{L}\sn_{t}u_{1}-v_{1}(t)+\int_{0}^{t}\tilde{L}\sn_{t-\tau}v_{1}(\tau)\ud\tau=0,
\end{aligned}
\end{equation}
where we also used the definition of $v_{j}$.

Finally, let $u$ be as in \eqref{it:mainwave4}. Then, by \eqref{eq:cossinprop}, \eqref{eq:defU0}, \eqref{eq:defU1} and \eqref{eq:firstderUj}, one has $u(0)=u_{0}$ and $\partial_{t} u(0)=u_{1}$. Moreover,  \eqref{eq:solveUj}, \eqref{eq:defU1}, \eqref{eq:firstderUj} and \eqref{eq:cossinprop}  yield
\begin{align*}
&(D_{t}^{2}-L)u(t)=F(t)-\int_{0}^{t}(D_{t}^{2}-L)U_{1}(t-\tau)F(\tau)\ud\tau=F(t)
\end{align*}
for almost all $t\in\R$. Thus $u$ solves \eqref{eq:eqdistr}. In particular, since $D_{t}^{2}u=Lu+F$, Proposition \ref{prop:Lproperties} \eqref{it:Lprop4} implies that $u$ has the required regularity.

\subsubsection{Uniqueness} 

We will show that any solution $u$ as in \eqref{eq:regularityu} to \eqref{eq:eqdistr} is unique. To this end, as in the proof of Theorem \ref{thm:firstorder}, we solve an adjoint problem. However, due to the low regularity of the coefficients, additional difficulties arise. Moreover, because the spaces that we consider are not necessarily reflexive, we cannot reason as in the proof of \cite[Theorem 5.2]{Hassell-Rozendaal23}. Hence we modify an idea from \cite[Section 5]{Smith98b}.

Let $-r+s(p)+1<s<r-s(p)$, where the endpoints may be included if $(a_{ij})_{i,j=1}^{n},(a_{j})_{j=0}^{n}\subseteq\HT^{r,\infty}(\Rn)$. It suffices to prove that, if
\[
u\in C(\R;\Hps)\cap C^{1}(\R;\Hpsm)\cap W^{2,1}_{\loc}(\R;\Hpsmm)
\]
satisfies $u(0)=\partial_{t}u(0)=0$ and $(D_{t}^{2}-L)u(t)=0$ in $\Hpsmm$ for almost all $t\in\R$, then $u\equiv0$. Write $u_{+}(t):=\ind_{[0,\infty)}(t)u(t)$ and $u_{-}(t):=\ind_{(-\infty,0)}(t)u(t)$. Then 
\[
u_{+},u_{-}\in C(\R;\Hps)\cap C^{1}(\R;\Hpsm)\cap W^{2,1}_{\loc}(\R;\Hpsmm),
\]
and $(D_{t}^{2}-L)u_{+}(t)=(D_{t}^{2}-L)u_{-}(t)=0$ in $\Hpsmm$ for almost all $t\in\R$. The latter identity and Proposition \ref{prop:Lproperties} \eqref{it:Lprop4} imply that $u_{+},u_{-}\in C^{2}(\R;\Hpsmm)$. 

By symmetry, we only need to show that 
\begin{equation}\label{eq:toshowF}
\int_{\R}\lb G(t),u_{+}(t)\rb\ud t=0
\end{equation}
for each $G\in C^{\infty}_{c}(\R;C^{\infty}_{c}(\Rn))$. 
Let $t_{0}\in\R$ be such that $G(t)=0$ for $t\geq t_{0}$. We will use that, by \eqref{eq:dualHpFIO},
\begin{equation}\label{eq:touseF}
\int_{\R}\lb (D_{t}^{2}-L^{*})\tilde{w}(t),u_{+}(t)\rb\ud t=\int_{\R}\lb \tilde{w}(t),(D_{t}^{2}-L)u_{+}(t)\rb\ud t=0
\end{equation}
for all 
\[
\tilde{w}\in C(\R;\HT^{2-s,p'}_{FIO}(\Rn))\cap C^{1}(\R;\HT^{1-s,p'}_{FIO}(\Rn))\cap C^{2}(\R;\HT^{-s,p'}_{FIO}(\Rn))
\]
with $\tilde{w}(t)=0$ for $t\geq t_{0}$. Here smoothness is considered in the weak-star topology. In fact, to ensure that each of the terms in \eqref{eq:touseF} is well defined, one also requires norm bounds for $\tilde{w}$ and its weak-star derivatives, locally uniformly in time. These bounds will be implicit in the construction below. 

Let $\tilde{L}$ and $(\sn_{t})_{t\in\R}$ be as before. For $t\in\R$, set $\tilde{v}_{0}(t):=G(t)$ and, recursively,
\[
\lb \tilde{v}_{k+1}(t),g\rb:=-\int_{t}^{t_{0}}\lb \tilde{v}_{k}(\tau),\sn_{t-\tau}\tilde{L}g\rb\ud \tau
\] 
for $k\geq0$ and $g\in\Hps$. Then, using Theorem \ref{thm:firstorder} and arguing as in the previous part of the proof, one sees that $\tilde{v}:=\sum_{k=0}^{\infty}\tilde{v}_{k}\in L^{1}_{\loc}(\R;\HT^{-s,p'}_{FIO}(\Rn))$ and $\tilde{v}(t)=0$ for $t\geq t_0$, where we consider integrability in the weak-star topology. Moreover, setting 
\[
\lb w(t),g\rb_{\Rn}:=\int^{t_{0}}_{t}\lb \tilde{v}(\tau),\sn_{t-\tau}g\rb\ud\tau
\]
for $t\in\R$ and $g\in\Hpsm$, we obtain
\[
w\in C(\R;\HT^{1-s,p'}_{FIO}(\Rn))\cap C^{1}(\R;\HT^{-s,p'}_{FIO}(\Rn))\cap C^{2}(\R;\HT^{-1-s,p'}_{FIO}(\Rn))
\]
in the weak-star topology, and $w(t)=0$ for $t\geq t_{0}$. Using also Remark \ref{rem:commute}, \eqref{eq:Ldecomp2} and \eqref{eq:tildeL2}, one sees that $(D_{t}^{2}-L^{*})w(t)=G(t)$ as distributions, for almost all $t\in\R$. However, the regularity of $w$ is insufficient to rely on \eqref{eq:touseF}. 
On the other hand, one does have
\begin{equation}\label{eq:touseF2}
\lb G(t),g\rb=\lb \tilde{v}(t),g\rb+\int_{t}^{t_{0}}\lb \tilde{v}(t),\sn_{t-\tau}\tilde{L}g\rb\ud\tau
\end{equation}
for almost all $t\in\R$ and all $g\in\Hps$, as follows by approximation, using the regularity of $\tilde{v}$ to see that both sides are well defined for such $g$.

So let $\veps>0$, and recall that $\ph\in C^{\infty}_{c}(\Rn)$ satisfies $\ph\equiv 1$ in a neighborhood of zero. For $t\in\R$, set $\lb \tilde{v}^{\veps}(t),g\rb:=\lb \tilde{v}(t),\ph(\veps D)g\rb$ and
\[
\lb w^{\veps}(t),g\rb:=\int^{t_{0}}_{t}\lb \tilde{v}^{\veps}(\tau),\sn_{t-\tau}g\rb\ud\tau.
\]
This is in fact well defined for all $g\in \HT^{s-2,p}_{FIO}(\Rn)$, and 
\[
w^{\veps}\in C(\R;\HT^{2-s,p'}_{FIO}(\Rn))\cap C^{1}(\R;\HT^{1-s,p'}_{FIO}(\Rn))\cap C^{2}(\R;\HT^{-s,p'}_{FIO}(\Rn))
\]
in the weak-star topology, with $w^{\veps}(t)=0$ for $t\geq t_{0}$. Moreover, 
\[
\|\tilde{v}^{\veps}(t)\|_{\HT^{-s,p'}_{FIO}(\Rn)}\lesssim \|\tilde{v}(t)\|_{\HT^{-s,p'}_{FIO}(\Rn)}
\]
for an implicit constant independent of $\veps$ and $t$, and $\tilde{v}^{\veps}(t)\to \tilde{v}(t)$ in the weak-star topology on $\HT^{-s,p'}_{FIO}(\Rn)$, as $\veps\to 0$. Finally, just as in \eqref{eq:touseF2}, 
\begin{equation}\label{eq:touseF3}
\lb (D_{t}^{2}-L^{*})w^{\veps}(t),g\rb=\lb \tilde{v}^{\veps}(t),g\rb+\int_{t}^{t_{0}}\lb \tilde{v}^{\veps}(t),\sn_{t-\tau}\tilde{L}g\rb\ud\tau
\end{equation} 
for almost all $t\in\R$ and all $g\in\Hps$.

To conclude, we can now combine \eqref{eq:touseF}, \eqref{eq:touseF2}, \eqref{eq:touseF3}, the dominated convergence theorem and the support properties of $\tilde{v}$ and $u_{+}$:
\begin{align*}
&\int_{\R}\lb G(t),u_{+}(t)\rb\ud t=\int_{\R}\lb G(t)-(D_{t}^{2}-L^{*})w^{\veps}(t),u_{+}(t)\rb\ud t\\
&=\int_{\R}\Big(\lb \tilde{v}(t)-\tilde{v}^{\veps}(t),u_{+}(t)\rb+\int_{t}^{t_{0}}\lb \tilde{v}(t)-\tilde{v}^{\veps}(t),\sn_{t-\tau}\tilde{L}u_{+}(t)\rb\ud\tau\Big)\ud t\to 0
\end{align*}
as $\veps\to0$. This in turn means that \eqref{eq:toshowF} holds, as required.
\end{proof}


\begin{remark}\label{rem:lowerreg}
One can lower the regularity of the lower-order terms of $L$, at the cost of shrinking the Sobolev interval for $s$ in Theorem \ref{thm:mainwave}. Indeed, the assumption $(a_{j})_{j=0}^{n}\subseteq C^{r}_{*}(\Rn)$ is only used for \eqref{eq:L3bound1}, for which one in turn needs that $a_{j}:\Hpsm\to\Hpsm$ for $1\leq j\leq n$, and that $a_{0}:\Hps\to\Hpsm$. The former holds if $(a_{j})_{j=1}^{n}\subseteq C^{r_{1}}_{*}(\Rn)$ for some $r_{1}>2s(p)$, by Remark \ref{rem:multiplication}, and the latter if e.g.~$2s(p)\leq 1$ and $a_{0}\in C^{r_{2}}_{*}(\Rn)$ for some $r_{2}>0$, by Corollary \ref{cor:pseudoclassFIO} (see also Remarks \ref{rem:suboptimalgamma} and \ref{rem:smallpproblem} for another approach). In this case, using also duality, the condition on $s$ in Theorem \ref{thm:mainwave} becomes 
\[
-\min(r,r_{1},r_{2})+s(p)+1<s<\min(r,r_{1}+1,r_{2}+1)-s(p),
\]
and one has to guarantee that this interval is not empty. At the very least, if $2s(p)\leq 1$ then the choices $r_{1}=r-1$ and $r_{2}>s(p)$ are allowed for $s$ in an open interval containing $1$. 

If $(a_{j})_{j=1}^{n}\subseteq \HT^{r_{1},\infty}(\Rn)$ and $a_{0}\in\HT^{r_{2},\infty}(\Rn)$ then one may include the endpoints of the Sobolev interval, and one may choose $r_{2}=s(p)$ for $s=1$ as well.

%
\end{remark}

\begin{remark}\label{rem:boundedder}
By \eqref{eq:cossinbounds}, $\partial_{t}^{k}U_{1}(t):\Hpsm\to\HT^{s-k,p}_{FIO}(\Rn)$ is a bounded operator for all $k\in\{0,1,2\}$ and $t\in\R$, with $\sup_{|t|\leq t_{0}}\|\partial_{t}^{k}U_{1}(t)\|_{\La(\Hpsm,\HT^{s-k,p}_{FIO}(\Rn))}<\infty$ for each $t_{0}>0$, and similarly for $\partial_{t}^{k}U_{0}(t)$. Here the restriction $k\leq 2$ arises from the term $v_{j}$ in \eqref{eq:secondderUj}, the smoothness of which depends on the range of $s$ for which \eqref{eq:mappingLtilde} holds. For larger $r$, one may weaken this restriction on $k$ by using the identity $\partial_{t}^{2}U_{j}(t)=-LU_{j}(t)$, and it can be removed for smooth coefficients.
\end{remark}


\begin{remark}\label{rem:group}
It follows from Theorem \ref{thm:mainwave} and \cite[Theorem II.6.7]{Engel-Nagel00} that
\[
S(t)\left(\!\begin{array}{c}
\!u_{0}\!\\
\!u_{1}\!
\end{array}\!\right):=\left(\!\begin{array}{c}
\!U_{0}(t)u_{0}+U_{1}(t)u_{1}\!\\
\!\partial_{t}U_{0}(t)u_{0}+\partial_{t}U_{1}(t)u_{1}\!
\end{array}\!\right)=\left(\!\begin{array}{cc}
\!U_{0}(t)&U_{1}(t)\!\\
\!\partial_{t}U_{0}(t)&\partial_{t}U_{1}(t)\!
\end{array}\!\right)\left(\!\begin{array}{c}
\!u_{0}\!\\
\!u_{1}\!
\end{array}\!\right)
\]
defines a strongly continuous group on $\Hps\times \Hpsm$, for all $p\in[1,\infty)$ and $s$ as in Theorem \ref{thm:mainwave}. Its generator is
\begin{equation}\label{eq:generator}
\mathcal{A}:=\left(\!\begin{array}{cc}
\!0&I\!\\
\!-L&0\!
\end{array}\!\right). 
\end{equation}
Basic semigroup theory then implies that $\A S(t)=S(t)\A$ for all $t\in\R$. In particular, $LU_{0}(t)u_{0}=U_{0}(t)Lu_{0}$ and $LU_{1}(t)u_{1}=U_{1}(t)Lu_{1}$ for all $(u_{0},u_{1})\in D(\A)$.
%
\end{remark}

\begin{remark}\label{rem:waveadjoint}
The statement of Theorem \ref{thm:mainwave} also holds if one replaces $L$ by $L^{*}$ in \eqref{eq:eqdistr}, albeit using different solution operators $(\tilde{U}_{0}(t))_{t\in\R}, (\tilde{U}_{1}(t))_{t\in\R}$. This follows from analogous reasoning. More precisely, to show existence, one can use the same formulas for the solution operators, cf.~\eqref{eq:defU0} and \eqref{eq:defU1}, but one has to modify the definition of the error terms $v_{j}$, by replacing $\tilde{L}$ by 
\begin{equation}\label{eq:Ltildenew}
\bar{L}:=2icb(x,D)-c^{2}-e_{1,2}(x,D)-L_{2,2}-L_{3}^{*}.
\end{equation}
This is allowed because $\bar{L}$ and $\tilde{L}$ 
have the same mapping properties, by Proposition \ref{prop:Lproperties} \eqref{it:Lprop3}. The rest of the proof of existence is then identical. For the proof of uniqueness, one can also use the same arguments, again with $\tilde{L}$ replaced by $\bar{L}$.
\end{remark}

\begin{remark}\label{rem:parametrixwave}
Given that the solution operators $U_{0}(t)$ and $U_{1}(t)$ are defined in terms of the operators $\ex_{t}$ from Theorem \ref{thm:firstorder}, 
and given that the latter operators can in turn be expressed using the parametrix from Theorem \ref{thm:parametrix}, the proof of Theorem \ref{thm:mainwave} also yields a parametrix for \eqref{eq:divwave}. In fact, as in \cite{Smith98b}, one could express the solution to \eqref{eq:divwave} more directly using the parametrix from Theorem \ref{thm:parametrix}, without first going through Theorem \ref{thm:firstorder}, but this would have led to additional technicalities in the proof. 
Also, there is intrinsic interest in solving the first-order problem.
\end{remark}

\subsection{Corollaries}\label{subsec:corollaries}

In this subsection we deduce a few corollaries from our main result. Versions of Remarks \ref{rem:lowerreg} and \ref{rem:boundedder} apply to several of the results in this subsection.

In the following extension of Theorem \ref{thm:mainwave} to $p=\infty$, we use the notation $C_{w}^{k}(\R;\HT^{s,\infty}_{FIO}(\Rn))$, for $k\in\Z_{+}$ and $s\in\R$, to denote the space of $\HT^{s,\infty}_{FIO}(\Rn)$ valued functions that are $k$ times continuously differentiable in the weak-star sense. 

\begin{corollary}\label{cor:waveinfty}
Let $(U_{0}(t))_{t\in\R}$ and $(U_{1}(t))_{t\in\R}$ be as in Theorem \ref{thm:mainwave}, and suppose that $r>\frac{n+1}{2}$. Let $s\in\R$ be such that $-r+\frac{n-1}{4}+1< s< r-\frac{n-1}{4}$. Then, for all $u_{0}\in\HT^{s,\infty}_{FIO}(\Rn)$, $u_{1}\in \HT^{s-1,\infty}_{FIO}(\Rn)$, $F\in L^{1}_{\loc}(\R;\HT^{s-1,\infty}_{FIO}(\Rn))$ 
and $t_{0}>0$, the following properties hold:
\begin{enumerate}
\item\label{it:waveinfty1} {$U_{k}(t):\HT^{s-k,\infty}_{FIO}(\Rn)\to\HT^{s,\infty}_{FIO}(\Rn)$ is bounded for all $t\in\R$ and $k\in\{0,1\}$, and $\sup_{|t|\leq t_{0}}\|U_{k}(t)\|_{\La(\HT^{s-k,\infty}_{FIO}(\Rn),\HT^{s,\infty}_{FIO}(\Rn))}<\infty$;} 
\item\label{it:waveinfty3} $[t\mapsto U_{0}(t)u_{0}],[t\mapsto U_{1}(t)u_{1}]\in C^{k}_{w}(\R;\HT^{s-k,\infty}_{FIO}(\Rn))$ for $k\in\{0,1,2\}$;
\item\label{it:waveinfty4} Set $u(t):=U_{0}(t)u_{0}+U_{1}(t)u_{1}-\int_{0}^{t}U_{1}(t-s)F(s)\ud s$ for $t\in\R$. Then 
\[
u\in C_{w}(\R;\HT^{s,\infty}_{FIO}(\Rn))\cap C^{1}_{w}(\R;\HT^{s-1,\infty}_{FIO}(\Rn))\cap W^{2,1}_{\loc,w}(\R;\HT^{s-2,\infty}_{FIO}(\Rn)),
\]
$u(0)=u_{0}$, $\partial_{t}u(0)=u_{1}$ 
and 
$(D_{t}^{2}-L)u(t)=F(t)$ 
for almost all $t\in \R$.
\end{enumerate}
Moreover, $(U_{0}(t))_{t\in\R}$ and $(U_{1}(t))_{t\in\R}$ are the unique collections with these properties.

If, additionally, $(a_{ij})_{i,j=1}^{n},(a_{j})_{j=0}^{n}\subseteq \HT^{r,\infty}(\Rn)$, then these statements hold for all $-r+\frac{n-1}{4}+1\leq s\leq r-\frac{n-1}{4}$.
%
%
%
\end{corollary}
In \eqref{it:waveinfty4}, the identities involving integration or differentiation are to be interpreted in the natural weak-star sense. In dimension $n=2$, it is relevant  for this corollary to recall that we always consider $r\geq 2$. 
\begin{proof}
We first rely on Theorem \ref{thm:mainwave} to see that the required collections exist, and that these are in fact the same solution operators as in Theorem \ref{thm:mainwave}. We then prove uniqueness using a similar argument as in the proof of Theorem \ref{thm:mainwave}.

\subsubsection{Existence} Let $(S(t))_{t\in\R}$ be the strongly continuous group  from Remark \ref{rem:group}, with generator $\A$ as in \eqref{eq:generator}. 
Let $(\tilde{U}_{0}(t))_{t\in\R},(\tilde{U}_{1}(t))_{t\in\R}$ be the solution operators from Remark \ref{rem:waveadjoint}, having the same properties as in Theorem \ref{thm:mainwave} but with $L$ replaced by $L^{*}$. Let $(\tilde{S}(t))_{t\in\R}$ be the corresponding strongly continuous group 
from Remark \ref{rem:group}, with generator $\tilde{\A}$ given by \eqref{eq:generator} but again with $L$ replaced by $L^{*}$. 

Considering $(S(t))_{t\in\R}$ as acting on $L^{2}(\Rn)\times W^{-1,2}(\Rn)$ and after momentarily identifying $L^{2}(\Rn)\times W^{1,2}(\Rn)$ with $W^{1,2}(\Rn)\times L^{2}(\Rn)$, the adjoint operators are
\[
\left(\!\begin{array}{cc}
\!\partial_{t}U_{1}(t)^{*}&U_{1}(t)^{*}\!\\
\!\partial_{t}U_{0}(t)^{*}&U_{0}(t)^{*}\!
\end{array}\!\right)
\]
for $t\in\R$, and the generator of this group is equal to $\tilde{A}$. Since a generator determines the associated group uniquely, we can undo the identification of $L^{2}(\Rn)\times W^{1,2}(\Rn)$ with $W^{1,2}(\Rn)\times L^{2}(\Rn)$ to see that 
\[
\left(\!\begin{array}{cc}
\!U_{0}(t)^{*}&\partial_{t}U_{0}(t)^{*}\!\\
\!U_{1}(t)^{*}&\partial_{t}U_{1}(t)^{*}\!
\end{array}\!\right)=S(t)^{*}=\left(\!\begin{array}{cc}
\!\partial_{t}\tilde{U}_{1}(t)&\partial_{t}\tilde{U}_{0}(t)\!\\
\!\tilde{U}_{1}(t)&\tilde{U}_{0}(t)\!
\end{array}\!\right)
\]
as operators on $L^{2}(\Rn)\times W^{1,2}(\Rn)$. Hence $U_{0}(t)^{*}=\partial_{t}\tilde{U}_{1}(t)$ and $U_{1}(t)^{*}=\tilde{U}_{1}(t)$. 

Now we can use Remark \ref{rem:boundedder} to see that $U_{0}(t)^{*}:\HT^{s-1,1}_{FIO}(\Rn)\to\HT^{s-1,1}_{FIO}(\Rn)$ and $U_{1}(t)^{*}:\HT^{s-1,1}_{FIO}(\Rn)\to\HT^{s,1}_{FIO}(\Rn)$. Taking adjoints again, and replacing $s$ by $1-s$, we obtain \eqref{it:waveinfty1}. 
In the same way, Remark \ref{rem:waveadjoint} yields \eqref{it:waveinfty3} for $k=0$ and $k=1$, as well as $k=2$ for $U_{1}(t)$. To deal with the case $k=2$ for $U_{0}(t)$, note that 
\begin{equation}\label{eq:extrader}
\partial_{t}^{2}U_{0}(t)^{*}=\partial_{t}^{3}\tilde{U}_{1}(t)=-\partial_{t}\tilde{U}_{1}(t)L^{*},
\end{equation}
where we used Remark \ref{rem:group} to commute $L^{*}$ and $\tilde{U}_{1}(t)$. By Proposition \ref{prop:Lproperties} \eqref{it:Lprop4} and the analogue of Theorem \ref{thm:mainwave} \eqref{it:mainwave3} for $\tilde{U}_{1}(t)$, the right-most term in \eqref{eq:extrader} is bounded from $\HT^{s+1,1}_{FIO}(\Rn)$ to $\HT^{s-1,1}_{FIO}(\Rn)$, and strongly continuous as a function of $t$. Again, by taking adjoints and replacing $s$ by $1-s$, one now obtains \eqref{it:waveinfty3} for $U_{0}(t)$ for $k=2$.

Finally, \eqref{it:waveinfty4} follows by duality from the corresponding statement for $(\tilde{U}_{0}(t))_{t\in\R}$ and $(\tilde{U}_{1}(t))_{t\in\R}$ from Theorem \ref{thm:mainwave} \eqref{it:mainwave4} and Remark \ref{rem:waveadjoint}. 

\subsubsection{Uniqueness}
The argument is almost completely analogous to that in the proof of Theorem \ref{thm:mainwave}. We will only indicate where changes have to be made. Let $-r+\frac{n-1}{4}+1<s<r-\frac{n-1}{4}$, with the endpoints included if $(a_{ij})_{i,j=1}^{n},(a_{j})_{j=0}^{n}\subseteq\HT^{r,\infty}(\Rn)$.

It suffices to show that, if 
\[
u\in C_{w}(\R;\HT^{s,\infty}_{FIO}(\Rn))\cap C^{1}_{w}(\R;\HT^{s-1,\infty}_{FIO}(\Rn))\cap W^{2,1}_{\loc,w}(\R;\HT^{s-2,\infty}_{FIO}(\Rn))
\]
satisfies $u(0)=\partial_{t}u(0)=0$ and $(D_{t}^{2}-L)u(t)=0$ in $\HT^{s-2,\infty}_{FIO}(\Rn)$ for almost all $t\in\R$, then $u\equiv0$. Write $u_{+}(t):=\ind_{[0,\infty)}(t)u(t)$ and $u_{-}(t):=\ind_{(-\infty,0)}(t)u(t)$. Then, again by Proposition \ref{prop:Lproperties} \eqref{it:Lprop4},
\[
u_{+},u_{-}\in C_{w}(\R;\HT^{s,\infty}_{FIO}(\Rn))\cap C^{1}_{w}(\R;\HT^{s-1,\infty}_{FIO}(\Rn))\cap C^{2}_{w}(\R;\HT^{s-2,\infty}_{FIO}(\Rn)).
\]
We want to prove that 
$\int_{\R}\lb G(t),u_{+}(t)\rb\ud t=0$ 
for each $G\in C^{\infty}_{c}(\R;C^{\infty}_{c}(\Rn))$. 

Let $t_{0}\in\R$ be such that $G(t)=0$ for $t\geq t_{0}$, 
let $(\sn_{t})_{t\in\R}$ be as in \eqref{eq:defcossin}, and let $\bar{L}$ be as in \eqref{eq:Ltildenew}. For $t\in\R$, set $\tilde{v}_{0}(t):=G(t)$ and, recursively,
\[
\tilde{v}_{k+1}(t):=-\int_{t}^{t_{0}}\bar{L}\sn_{t-\tau}\tilde{v}_{k}(\tau)\ud \tau
\] 
for $k\geq0$. Then, as in the proof of Theorem \ref{thm:mainwave} (see also Remark \ref{rem:waveadjoint}), one sees that $\tilde{v}:=\sum_{k=0}^{\infty}\tilde{v}_{k}\in L^{1}_{\loc}(\R;\HT^{-s,1}_{FIO}(\Rn))$ and $\tilde{v}(t)=0$ for $t\geq t_0$. Note that, unlike in the proof of Theorem \ref{thm:mainwave}, here we consider integrability in the norm topology. 

Next, set $w(t):=\int^{t_{0}}_{t}\sn_{t-\tau}\tilde{v}(\tau)\ud\tau$ for $t\in\R$. Then
\[
w\in C(\R;\HT^{1-s,1}_{FIO}(\Rn))\cap C^{1}(\R;\HT^{-s,1}_{FIO}(\Rn))\cap C^{2}(\R;\HT^{-1-s,1}_{FIO}(\Rn)),
\]
and $w(t)=0$ for $t\geq t_{0}$. 
Moreover, $G(t)=\tilde{v}(t)+\int_{t}^{t_{0}}\bar{L}\sn_{t-\tau}\tilde{v}(t)\ud\tau$ for almost all $t\in\R$. 

Finally, for $\veps>0$ and $t\in\R$, set $\tilde{v}^{\veps}(t):=\ph(\veps D)\tilde{v}(t)$ and $w^{\veps}(t):=\int^{t_{0}}_{t}\sn_{t-\tau}\tilde{v}^{\veps}(\tau)\ud\tau$. Then one can argue just as in the proof of Theorem \ref{thm:mainwave} to see that
\begin{align*}
&\int_{\R}\lb G(t),u_{+}(t)\rb\ud t=\int_{\R}\lb G(t)-(D_{t}^{2}-L^{*})w^{\veps}(t),u_{+}(t)\rb\ud t\\
&=\int_{\R}\Big(\lb \tilde{v}(t)-\tilde{v}^{\veps}(t),u_{+}(t)\rb+\int_{t}^{t_{0}}\lb \bar{L}\sn_{t-\tau}(\tilde{v}(t)-\tilde{v}^{\veps}(t)),u_{+}(t)\rb\ud\tau\Big)\ud t\to 0
\end{align*}
as $\veps\to0$.
\end{proof}


By combining Theorem \ref{thm:mainwave} with the Sobolev embeddings for $\Hps$, we obtain a corollary for initial data in $\HT^{s,p}(\Rn)$.

\begin{corollary}\label{cor:divLp}
Let $(U_{0}(t))_{t\in\R}$ and $(U_{1}(t))_{t\in\R}$ be as in Theorem \ref{thm:mainwave}. Then, for all $p\in [1,\infty]$ and $s\in\R$ such that $2s(p)+1<r$ and $-r+s(p)+1< s< r-s(p)$, and for all $u_{0}\in\HT^{s+s(p),p}(\Rn)$, $u_{1}\in \HT^{s+s(p)-1,p}(\Rn)$, $F\in L^{1}_{\loc}(\R;\HT^{s+s(p)-1,p}(\Rn))$ and $t_{0}>0$, the following properties hold:
\begin{enumerate}
\item\label{it:divLp1} {$U_{k}(t):\HT^{s+s(p)-k,p}(\Rn)\to\HT^{s-s(p),p}(\Rn)$ is bounded for all $t\in\R$ and $k\in\{0,1\}$, and $\sup_{|t|\leq t_{0}}\|U_{k}(t)\|_{\La(\HT^{s+s(p)-k,p}(\Rn),\HT^{s-s(p),p}(\Rn))}<\infty$};
\item\label{it:divLp3} If $p<\infty$, then $[t\mapsto U_{0}(t)u_{0}],[t\mapsto U_{1}(t)u_{1}]\in C^{k}(\R;\HT^{s-s(p)-k,p}(\Rn))$ for $k\in\{0,1,2\}$;
\item\label{it:divLp4} Set $u(t):=U_{0}(t)u_{0}+U_{1}(t)u_{1}-\int_{0}^{t}U_{1}(t-s)F(s)\ud s$ for $t\in\R$.  If $p<\infty$, then
\[
u\in C(\R;\HT^{s-s(p),p}(\Rn))\cap C^{1}(\R;\HT^{s-s(p)-1,p}(\Rn))\cap W^{2,1}_{\loc}(\R;\HT^{s-s(p)-2,p}(\Rn)),
\]
$u(0)=u_{0}$, $\partial_{t}u(0)=u_{1}$, and $(D_{t}^{2}-L)u(t)=F(t)$ in $\HT^{s-s(p)-2,p}(\Rn)$ for almost all $t\in \R$.
\end{enumerate}
If $p<\infty$ and $s>-r+3s(p)+1$, then $(U_{0}(t))_{t\in\R}$ and $(U_{1}(t))_{t\in\R}$ are the unique collections with these properties.

Moreover, if $(a_{ij})_{i,j=1}^{n},(a_{j})_{j=0}^{n}\subseteq \HT^{r,\infty}(\Rn)$, then these statements also hold for $s=-r+s(p)+1$ and $s=r-s(p)$. 
\end{corollary}
\begin{proof}
The properties of $(U_{0}(t))_{t\in\R}$ and $(U_{1}(t))_{t\in\R}$ are a direct consequence of Theorem \ref{thm:mainwave}, Corollary \ref{cor:waveinfty} and \eqref{eq:Sobolev}. The uniqueness statement follows in the same manner, since such a $u$ satisfies
\[
u\in  C(\R;\HT^{s-2s(p),p}_{FIO}(\Rn))\cap C^{1}(\R;\HT^{s-1-2s(p),p}_{FIO}(\Rn))\cap W^{2,1}_{\loc}(\R;\HT^{s-2-2s(p),p}_{FIO}(\Rn)).
\]
For $-r+s(p)+1<s-2s(p)$, one now obtains the uniqueness from the proof of Theorem \ref{thm:mainwave}, and similarly if $(a_{ij})_{i,j=1}^{n},(a_{j})_{j=0}^{n}\subseteq \HT^{r,\infty}(\Rn)$.
\end{proof}

\begin{remark}\label{rem:Lpinfty}
By Corollary \ref{cor:waveinfty}, if $r>\frac{n+1}{2}$ then \eqref{it:divLp3} and \eqref{it:divLp4} of Corollary \ref{cor:divLp}, as well as the uniqueness statement, hold for $p=\infty$ in a weak-star sense.

Note 
that the collections $(U_{0}(t))_{t\in\R},(U_{1}(t))_{t\in\R}$ from Theorem \ref{thm:mainwave} have stronger regularity properties than those listed in Corollary \ref{cor:divLp}, and collections with such stronger properties are unique. The condition $s>-r+3s(p)+1$ is only imposed to ensure that collections with the weaker properties in Corollary \ref{cor:divLp} are also unique. We will not explore whether this condition is necessary.
\end{remark}

Theorem \ref{thm:mainwave} yields the following corollary about the boundedness of certain operators that arise in spectral calculus. 

\begin{corollary}\label{cor:spectral}
Suppose that $a_{ij}=a_{ji}$ and $a_{j}=0$ for all $1\leq i,j\leq n$, and that $a_{0}(x)\geq0$ for all $x\in\Rn$. Let $p\in [1,\infty]$ and $s\in\R$ be such that $2s(p)+1<r$ and $-r+s(p)+1< s<r-s(p)$. Then
\begin{equation}\label{eq:cosbounded}
\cos(t\sqrt{L}):\Hps\to\Hps
\end{equation}
and
\[
\sin(t\sqrt{L})L^{-1/2}:\Hpsm\to\Hps
\]
are bounded, locally uniformly in $t\in\R$. 
If $(a_{ij})_{i,j=1}^{n}\subseteq \HT^{r,\infty}(\Rn)$ and $a_{0}\in\HT^{r,\infty}(\Rn)$, then this statement 
also holds for $s=-r+s(p)+1$ and $s=r-s(p)$. 
\end{corollary}
The Sobolev embeddings in \eqref{eq:Sobolev} yield the associated sharp $\HT^{s,p}(\Rn)$ regularity for $\cos(t\sqrt{L})$ and $\sin(t\sqrt{L})L^{-1/2}$.
\begin{proof}
By assumption, $L$ is a positive operator on $L^{2}(\Rn)$. Hence $\cos(t\sqrt{L})$ and $\sin(t\sqrt{L})L^{-1/2}$ are well defined and bounded on 
$L^{2}(\Rn)=\HT^{2}_{FIO}(\Rn)$, for all $t\in\R$, through spectral calculus.

It follows from form theory (see \cite[Chapter 6]{Kato95}) that $\sqrt{L}:L^{2}(\Rn)\to W^{-1,2}(\Rn)$. Moreover, $L:L^{2}(\Rn)\to W^{-2,2}(\Rn)$, by 
Proposition \ref{prop:Lproperties} \eqref{it:Lprop4}. Now spectral calculus shows that $U_{0}(t):=\cos(t\sqrt{L})$ and $U_{1}(t):=\sin(t\sqrt{L})L^{-1/2}$ 
have the properties in Theorem \ref{thm:mainwave}, for $p=2$ and $s=0$. By uniqueness, these operators thus have the mapping properties contained in Theorem \ref{thm:mainwave} and Corollary \ref{cor:waveinfty}. 
\vanish{
 , for $f\in \Sw(\Rn)$, the function $t\mapsto u(t):=\cos(t\sqrt{L})f$ satisfies
\[
u\in C(\R;L^{2}(\Rn))\cap C^{1}(\R;W^{-1,2}(\Rn))\cap C^{2}(\R;W^{-2,2}(\Rn)),
\]
$u(0)=f$, $\partial_{t}u(0)=0$ and $(D_{t}^{2}-L)u(t)=0$. For $-r+s(p)+1\leq s\leq r-s(p)$ it thus follows from Remarks \ref{rem:additionalreg} and \ref{rem:formula} that
\[
u\in C(\R;\Hps)\cap C^{1}(\R;\Hpsm)\cap C^{2}(\R;\Hpsmm),
\]
and that for each $t_{0}>0$ there exists a $\con>0$, independent of $f$, such that $\|u(t)\|_{\Hps}\leq \con\|f\|_{\Hps}$ for all $t\in[-t_{0},t_{0}]$. Since $\Sw(\Rn)\subseteq\Hps$ is dense, this yields \eqref{eq:cosbounded} for $-r+s(p)+1< s< r-s(p)$. Then duality, cf.~\eqref{eq:duality}, extends \eqref{eq:cosbounded} to all $-r+s(p)< s< r-s(p)$. 

Applying the same reasoning to the function $t\mapsto v(t):=\sin(t\sqrt{L})L^{-1/2}f$, one obtains \eqref{eq:sinbounded} for all $-r+s(p)+1< s< r-s(p)$. The statements for $r\in\N$ follow in the same manner.}
\end{proof}

\begin{remark}\label{rem:largerrange}
Since $\cos(t\sqrt{L})$ is self adjoint, for $p\in[1,\infty]$ one can use duality to see that \eqref{eq:cosbounded} holds for all $-r+s(p)< s< r-s(p)$, with the endpoints included if $(a_{ij})_{i,j=1}^{n}\subseteq \HT^{r,\infty}(\Rn)$ and $a_{0}\in\HT^{r,\infty}(\Rn)$.
\end{remark}

\vanish{\begin{corollary}\label{cor:halfwave}
Let $p\in(1,\infty)$ and $s\in\R$ be such that $2s(p)+\max(1,\frac{1}{p})<r$ and $-r+s(p)+1< s< r-s(p)-1$. Then there exists a $c\in\R$ such that 
\[
e^{it\sqrt{L+c}}:\Hps\to\Hps
\]
boundedly. If $(a_{ij})_{i,j=1}^{n},(a_{j})_{j=0}^{n}\subseteq \HT^{r,\infty}(\Rn)$, then this statement holds for all $-r+s(p)+1\leq s\leq r-s(p)-1$.
\end{corollary}
As before, the Sobolev embeddings in Theorem \ref{thm:Sobolev} yields the associated sharp $\HT^{s,p}(\Rn)$ regularity for $e^{it\sqrt{L}}$. Note that $c$ is allowed to depend on $p$ and $s$. 
\begin{proof}
First note that $L$ is a closed operator on $\Hps$, by Proposition \ref{prop:Lproperties} \eqref{it:Lprop4}. 
Choose $c$ such that $L+c$ is a sectorial operator on $\Hps$
\end{proof}
}

\section*{Acknowledgments}\label{sec:acknowledge}

 The authors would like to thank Prof.~Lixin Yan for helpful  discussions and valuable support.

\bibliographystyle{plain}
\bibliography{rough}
\end{document}